\documentclass{book}
\usepackage{amsmath}
  \usepackage{paralist}
  \usepackage{graphics} %% add this and next lines if pictures should be in esp format
  \usepackage{epsfig} %For pictures: screened artwork should be set up with an 85 or 100 line screen
\usepackage{tikz}
\usepackage{caption}
\usepackage{graphicx}  \usepackage{epstopdf}%This is to transfer .eps figure to .pdf figure; please compile your paper using PDFLeTex or PDFTeXify.
\usepackage{amssymb}
\usepackage{amsthm}
\usepackage{epstopdf}
\usepackage{verbatim}
\usepackage{mathrsfs}
\usepackage{mathtools}
 \usepackage[colorlinks=true]{hyperref}
\hypersetup{urlcolor=blue, citecolor=red}

\newtheorem{theorem}{Theorem}[section]
\newtheorem{corollary}{Corollary}

\newtheorem{lemma}[theorem]{Lemma}
\newtheorem{proposition}{Proposition}
\newtheorem{assumption}{Assumption}
\newtheorem{example}{Example}

\theoremstyle{definition}
\newtheorem{definition}[theorem]{Definition}
\newtheorem{remark}{Remark}

\newcommand{\defeq}{\overset{\scriptscriptstyle\mathrm{def}}{=}}
\newcommand{\Aut}{\operatorname{Aut}}
\newcommand{\MIS}{\operatorname{MIS}}
\newcommand{\Hmis}{\mathcal{H}_{\mathrm{MIS}}}
\newcommand{\His}{\mathcal{H}_{\mathrm{IS}}}
\newcommand{\Jblo}{J_{\mathrm{blo}}}
\newcommand{\ham}{H_{\gamma}}
\newcommand{\bx}{\mathbf{x}}
\newcommand{\by}{\mathbf{y}}
\newcommand{\bz}{\mathbf{z}}
\newcommand{\bze}{\mathbf{0}}

\DeclareMathOperator{\spn}{span}

\newenvironment{abstract}%
{\begin{center}\textbf{Abstract}\end{center}\begin{quotation}\noindent\ignorespaces}%
{\end{quotation}}

\title{Towards a Control interpretation of Quantum Advantage}

\begin{document}
	\maketitle
	
	% Enter the first author's name and address:
	\centerline{\scshape Dario Pighin$^*$}
	\medskip
	{\footnotesize
		% please put the address of the first author
		\centerline{Irontec, Internet y sistemas sobre GNU/Linux, S.L.}
		%   \centerline{Other lines}
		\centerline{c/Uribitarte, 6-2º (Bilbao – 48001), Spain}
		\centerline{E-mail: \href{mailto:dpighin@irontec.com}{dpighin@irontec.com}}
	} % Do not forget to end the {\footnotesize by the sign }
	\begingroup
	\let\thefootnote\relax\footnotetext{* We gratefully acknowledge professor Enrique Zuazua for his precious suggestions, along the preparation of the notes.}
	\endgroup
%	
%	
%	\medskip
%	
%	\centerline{\scshape Enrique Zuazua}
%	\medskip
%	{\footnotesize
%		% please put the address of the first author
%		\centerline{DeustoTech, Fundaci\'on Deusto}
%		%   \centerline{Other lines}
%		\centerline{Avda. Universidades,
%			24, 48007, Bilbao, Basque Country, Spain}
%	} % Do not forget to end the {\footnotesize by the sign }
%	\medskip
%	{\footnotesize
%		% please put the address of the first author
%		\centerline{Departamento de Matem\'aticas, Universidad Aut\'onoma de Madrid}
%		%   \centerline{Other lines}
%		\centerline{28049 Madrid, Spain}
%	} % Do not forget to end the {\footnotesize by the sign }
%	\medskip
%	{\footnotesize
%		% please put the address of the first author
%		\centerline{Facultad de Ingenier\'ia, Universidad de Deusto}
%		%   \centerline{Other lines}
%		\centerline{Avda. Universidades,
%			24, 48007, Bilbao, Basque Country, Spain}
%	} % Do not forget to end the {\footnotesize by the sign }
%	
%	\bigskip
%	
%	% The name of the associate editor will be entered by an editorial staff
%	% "Communicated by the associate editor name" is not needed for special issue.
%	% \centerline{(Communicated by the associate editor name)}[ADD]
%	
%	
	%The abstract of your paper
	\begin{abstract}
		We develop a control-theoretic framework for understanding Quantum Advantage (QA), providing a systematic route to characterize when and how QA can arise. The bilinear controlled Schr\"{o}dinger equation is the common thread: the target quantum computation is recast as an operator controllability problem on the special unitary group $SU(N)$, and QA is identified with a polynomial-in-$n$ upper bound on the associated minimal-time function.
		
		We illustrate the framework on two paradigmatic problems:
		\begin{enumerate}
			\item[(a)] the \emph{Quantum Fourier Transform (QFT)} on superconducting digital quantum processors (such as IBM's \texttt{ibm\_brisbane}), for which we prove operator controllability by a Lie-algebraic argument and derive an $O(n^2)$ upper bound on the minimal time via a gate-concatenation lemma combined with the standard QFT circuit decomposition;
			\item[(b)] the \emph{Maximum Independent Set (MIS)} problem on neutral-atom analog quantum processors (such as Pasqal's hardware), for which we analyze the Rydberg-blockade Hamiltonian as a bilinear control system and reformulate Quantum Approximate Optimization Algorithm (QAOA) as a continuous-time optimal control problem. By a controllability result, we show how the problem can be solved on Pasqal Quantum Computers and we introduce a control-based definition of Quantum Advantage for MIS.
		\end{enumerate}
		
		%Building on this analysis, we introduce a practical, checkable necessary condition for QA in terms of a surrogate optimal control problem on the commutator $[U(t)\Gamma^*, H_0]$, and we derive quantum speed limits through the Riemannian geometry of $SU(N)$.
		
		We conclude by outlining several open problems that chart directions for future research at the intersection of Control Theory and Quantum Computing.
		
	\end{abstract}
	\medskip
	
	\tableofcontents
		
		\chapter{Introduction}
		\label{chapter:1}
		
		The purpose of these notes is to address the following question:
		\begin{itemize}
			\item[]\textit{For a given problem, is there a Quantum Advantage (QA) in solving it on a quantum computer ?}
		\end{itemize}
		
		Under appropriate assumptions, the answer is
		%yes
		positive. A paradigmatic example is the Quantum Fourier Transform (QFT) (see, e.g. \cite[section 5.1]{nielsen2010quantum}), the main building block of several celebrated quantum algorithms, such as Shor's algorithm for prime factorization \cite{shor1999polynomial} or the HHL algorithm \cite{harrow2009quantum} to solve linear systems of equations. A second, structurally different example is the Maximum Independent Set (MIS) problem (see figure \ref{fig:mis_example}), which is NP-hard in the classical sense \cite{karp1972reducibility} and admits a natural analog quantum encoding on neutral-atom hardware via the Rydberg blockade mechanism \cite{Henriet2020quantumcomputing,pichler2018quantum,ebadi2022quantum}. Our goal is to answer the above question by control-theoretic tools, with the hope of building a systematic approach to assess whether Quantum Advantage (QA) holds, for a wide class of algorithms and hardware platforms.
		
		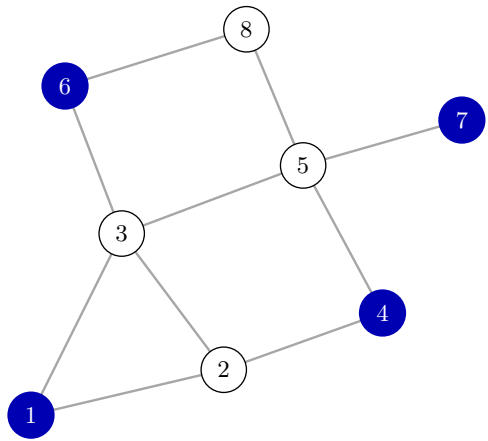
\begin{figure}[htp]
			\begin{center}
				\begin{tikzpicture}[x=1.5cm,y=1.5cm,
					every node/.style={font=\small},
					atom/.style={circle,draw=black,fill=white,line width=0.5pt,inner sep=0pt,minimum size=6mm},
					mis/.style={circle,draw=blue!70!black,fill=blue!70!black,text=white,line width=0.7pt,inner sep=0pt,minimum size=6mm},
					gedge/.style={gray!70,line width=0.9pt},
					blockade/.style={draw=blue!45!black,dashed,line width=0.4pt}]
					% Atom layout (arbitrary units); unit-disk threshold R_b=1.8.
					\coordinate (a1) at (0.0,0.0);
					\coordinate (a2) at (1.7,0.4);
					\coordinate (a3) at (0.8,1.6);
					\coordinate (a4) at (3.1,0.9);
					\coordinate (a5) at (2.4,2.2);
					\coordinate (a6) at (0.3,2.9);
					\coordinate (a7) at (3.8,2.6);
					\coordinate (a8) at (1.9,3.4);
					% Edges: atom pairs within R_b.
					\foreach \a/\b in {1/2,1/3,2/3,2/4,3/5,3/6,4/5,5/7,5/8,6/8}
					\draw[gedge] (a\a)--(a\b);
					% Blockade disks (radius R_b/2=0.9) around the chosen MIS atoms.
					% \foreach \v in {1,4,6,7} \draw[blockade] (a\v) circle (0.9);
					% Vertices.
					\foreach \v in {2,3,5,8} \node[atom] at (a\v) {\v};
					\foreach \v in {1,4,6,7} \node[mis]  at (a\v) {\v};
				\end{tikzpicture}
				\caption{An instance of the Maximum Independent Set (MIS) problem on a
					unit-disk graph $G=(V,E)$ with $n=8$ atoms. The highlighted
					vertices $\{1,4,6,7\}$ form a maximum independent set, so $\alpha(G)=4$; their pairwise disjointness is exactly the
					independence condition. This instance has three distinct
					maximum independent sets, $\{1,4,6,7\}$, $\{1,4,7,8\}$ and
					$\{3,4,7,8\}$, so the MIS subspace \eqref{mis_subspace_def} has dimension
					$d=|\mathrm{MIS}(G)|=3$.}\label{fig:mis_example}
			\end{center}
		\end{figure}
		
		Concretely, we illustrate the framework on two representative use cases:
		\begin{enumerate}
			\item[(a)] \textbf{QFT on superconducting (digital, gate-based) quantum computers}, such as IBM's \href{https://quantum.cloud.ibm.com/computers?system=ibm_brisbane}{\texttt{ibm\_brisbane}};
			\item[(b)] \textbf{MIS on neutral-atom (analog) quantum computers}, such as Pasqal's \href{https://portal.pasqal.cloud/devices/FRESNEL}{FRESNEL} device, endowed \href{https://www.pasqal.com/wp-content/uploads/2025/03/Technical-Overview-for-Advanced-Users-Orion-Beta.pdf}{Orion} Quantum Processing Unit (QPU)\footnote{\href{https://quantum-journal.org/papers/q-2020-09-21-327/pdf/}{Pasqal Quantum Computers} \cite{Henriet2020quantumcomputing} are available both in gate-based digital mode (enjoying universal quantum computing property) and analog mode. In these notes, we focus on the analog mode.}.
		\end{enumerate}
		These two cases cover the two complementary paradigms of universal digital quantum computation and analog quantum optimization, and exercise, respectively, operator controllability (QFT) and bilinear optimal control (MIS-QAOA) within one unified framework.
		
		Quantum computing exploits the principles of quantum mechanics (superposition and entanglement) to perform computations in ways that can vastly outperform classical algorithms for certain tasks \cite{nielsen2010quantum}. This gap in computational power is often referred to as \emph{Quantum Advantage (QA)} ,  the ability of a quantum device to solve a problem beyond the feasible reach of classical computers in any reasonable time. A famous example is Shor's algorithm for integer factorization, which runs in polynomial time on a quantum computer whereas the best known classical methods require super-polynomial time \cite{nielsen2010quantum,shor1999polynomial,vandersypen2001experimental,monz2016realization}. The Shor's algorithm is in fact based on Quantum Fourier Transform (QFT).
		
		The quantum computer was conceived by Richard P. Feynman in the celebrated paper \cite{feynman2018simulating} for physical simulation. More recently, Quantum Computers are designed to run general computations (see, e.g. \cite{nielsen2010quantum,martyn2021grand} and references therein). Moreover, extending Richard P. Feynman's idea, in recent works like \cite{jin2024quantum}, specific change of variables named Schr{\"o}dingerization are proposed to allow running general Partial Differential Equation (PDE) simulation in a quantum computers.
		
		In recent years, experimental demonstrations of Quantum Advantage (QA) have been a major milestone in the field. On the digital side, superconducting quantum processors, such as those developed by IBM and Google, have executed specialized sampling problems believed to be classically intractable. In particular, IBM's cloud quantum computing platform now provides access to superconducting devices (e.g., the 127-qubit \href{https://quantum.cloud.ibm.com/computers?system=ibm_brisbane}{ibm$\_$brisbane} or the 156-qubit \href{https://www.basquequantum.eus/en/ibm}{\texttt{ibm\_BasqueCountry}} installed in Donostia - Basque Country) that enable researchers to explore quantum algorithms on real hardware\footnote{See \url{https://quantum.cloud.ibm.com/computers?system=ibm_brisbane} for details on the \href{https://quantum.cloud.ibm.com/computers?system=ibm_brisbane}{ibm$\_$brisbane} quantum processor.}. On the analog side, neutral-atom quantum processors, such as those developed by Pasqal \cite{Henriet2020quantumcomputing} and QuEra \cite{ebadi2022quantum}, have demonstrated the ability to solve combinatorial optimization problems (in particular, Maximum Independent Set on unit-disk graphs) with hundreds of atoms, exploiting the Rydberg blockade mechanism for native constraint enforcement. These developments underscore the practical quest for Quantum Advantage (QA) across both paradigms and motivate a deeper theoretical understanding of how and when it can be achieved.
		
		Achieving Quantum Advantage (QA) in practice hinges on the ability to \emph{control} quantum systems with high precision. On a superconducting quantum computer, each logical qubit operation is implemented by externally applied control fields (e.g., shaped microwave pulses) that drive the evolution of the qubits' quantum state. On a neutral-atom analog processor, the many-body atomic register is globally steered by laser fields with programmable Rabi frequency and detuning, together with a programmable spatial arrangement of the atoms. In both cases, the computation is an externally controlled trajectory in Hilbert space, and this places quantum computing squarely in the domain of Control Theory. In classical Control Theory, a dynamical system governed by differential equations is called \emph{controllable} if one can drive the system from any given initial state to any desired final state by choosing appropriate control inputs \cite{CNL,ZCP}. The concept of controllability is fundamental: it formalizes the intuitive notion of having complete command over a system's behavior. Classical results, such as Kalman's rank condition for linear systems, provide clear criteria for controllability in finite dimensions \cite{CNL}. For more complex systems described by Partial Differential Equations (PDEs), powerful tools have been developed to analyze controllability and observability \cite{ZCP}.
		
		The integration of Control Theory with quantum mechanics has given rise to a rich interdisciplinary field of \emph{Quantum Control} that is providing novel insights into both physics and engineering \cite{berberich2024quantum,boscain2021introduction,d2021introduction,doi:10.1137/1.9781611977745.19,dong2010quantum}. In the quantum domain, controllability typically means the ability to implement an arbitrary unitary transformation or prepare any target state in the Hilbert space of the system, given a suitable set of control Hamiltonians \cite{d2021introduction}. For many finite-dimensional quantum systems (such as a register of superconducting qubits), one can establish controllability via the Lie-algebraic rank condition: the Lie algebra generated by the system's drift and control Hamiltonians spans the entire $\mathfrak{su}(N)$ algebra (for an $N$-dimensional Hilbert space), ensuring that any unitary operation is reachable \cite{d2021introduction}. In practical terms, this implies a quantum computer with controllable qubits is, in principle, a universal computing device. Of course, practical limitations like decoherence and imperfect actuators mean not every unitary is achievable with high fidelity, but the theoretical notion of controllability provides an important benchmark. Infinite-dimensional quantum systems has been analyzed as well by Control Theory. For those systems, the {S}chr\"{o}dinger equation is a Partial Differential Equation (PDE). Some references on the application of PDE Control Theory to the {S}chr\"{o}dinger equation are \cite{machtyngier1994exact,beauchard2010local,zuazua2002remarks,beauchard2025examples,beauchard2024small,ZCP} and references therein. In essence, quantum hardware presents a rich, high-dimensional control system, and the challenge is to steer it through its exponentially large state space efficiently and accurately.
		
		In addition to controllability, \emph{Optimal Control} plays a vital role in quantum computing. Optimal Control Theory asks: given a controllable system, what is the best way to steer it to achieve a desired objective while minimizing a cost (such as time, energy, or error)? \cite{d2021introduction} Quantum Optimal Control techniques have been widely applied to design pulse sequences that implement quantum logic gates or state transfers with high fidelity and minimal duration. For instance, gradient-based algorithms can optimize microwave control pulses to carry out a quantum gate on superconducting qubits in the shortest possible time, respecting hardware constraints. The use of Optimal Control has become a cornerstone in improving quantum operations, reducing error rates, and pushing quantum devices closer to the regimes required for demonstrating Quantum Advantage (QA).
		
		Given this backdrop, in this work we investigate the existence of Quantum Advantage (QA) from a control-theoretic perspective. We ask whether the superior computational power of quantum systems can be formally understood (and even quantified) through the lens of controllability and Optimal Control. By viewing a quantum algorithm as a controlled dynamical trajectory in Hilbert space, we can analyze what aspects of that trajectory might be intractable for any classical controller or classical computer to replicate. Our approach is theoretical: we derive analytical results that connect control-theoretic properties of quantum systems to their computational capabilities. In particular, we propose criteria based on control complexity and reachable sets that, if satisfied by a quantum system, would imply a provable Quantum Advantage (QA) for a certain class of problems. This amounts to a theoretical proof-of-concept for Quantum Advantage (QA), grounded in Control Theory, which does not rely on specific hardware experiments. For the moment, our results remain in the realm of mathematical analysis; however, they lay the groundwork for future experimental validation on physical quantum processors. Moreover, our approach paves the way of a control-inspired improvement of quantum algorithms.
		
		This paper is written for an interdisciplinary audience of control theorists and quantum physicists. We therefore review the necessary background from both fields and strive to use terminology accessible to each community. We highlight how Control Theory concepts such as controllability and Optimal Control can provide fresh insights into quantum computation, and conversely, how quantum computing motivates new questions in Control Theory. Our hope is that this synergy between control and quantum dynamics will not only help in rigorously establishing Quantum Advantage (QA), but also foster collaboration between the two disciplines in addressing the challenges of quantum technology.
		
		%In order to have a fixed realistic framework, we assume our quantum computer is
		%\href{https://quantum.cloud.ibm.com/computers?system=ibm_brisbane}{ibm$\_$brisbane}.
	
		In chapter \ref{chapter:Mathematical framework}, we introduce the mathematical framework for Quantum Computing (QC), inspired from \cite[chapter 2]{nielsen2010quantum}: the state space, the bilinear controlled Schr\"{o}dinger equation, the notion of operator controllability and the control-theoretic definition of Quantum Advantage. Chapter \ref{chapter:QA_QFT} is devoted to the first representative problem: we model superconducting quantum computers (such as \href{https://quantum.cloud.ibm.com/computers?system=ibm_brisbane}{ibm$\_$brisbane}) by a controlled Schr\"{o}dinger equation, prove operator controllability by standard Lie-algebraic methods, and establish Quantum Advantage for the QFT by combining a gate-concatenation lemma with the well-known $O(n^2)$ circuit decomposition. Chapter \ref{chapter:Quantum Discrete Optimization} is devoted to the second representative problem: we introduce the MIS problem, model neutral-atom quantum computers (such as Pasqal's hardware) by a Rydberg-blockade Hamiltonian, reformulate QAOA as a bilinear optimal control problem, show that MIS is solvable on Pasqal devices, and discuss the scope of Quantum Advantage in this analog setting. Chapter \ref{chapter:A surrogate problem} formulates a surrogate optimal control problem on the commutator $[U(t)\Gamma^*, H_0]$ providing a checkable necessary condition for Quantum Advantage applicable to both paradigms. We conclude the manuscript by collecting some open problems in chapter \ref{chapter:Open problems}. The appendix includes quantum speed limits on $SU(N)$ via bi-invariant Riemannian geometry, a discussion of the steady-control case, and an explicit example of drift Hamiltonian.
		
		\chapter{Mathematical framework}
		\label{chapter:Mathematical framework}
		
		The quantum computer works by controlling quantum processes running on specific (quantum) physical systems.
		
		Typically, in quantum computing, the state space is modeled as a Hilbert space over~$\mathbb{C}$ of the form
		\begin{equation}\label{state_space}
			\mathcal{H} \defeq \underbrace{\mathbb{C}^2 \otimes \dots \otimes \mathbb{C}^2}_{n\text{ times}} = \bigotimes_{j=1}^n \mathbb{C}^2,
		\end{equation}
		which has dimension\footnote{hereafter, by the word dimension, we mean dimension over $\mathbb{C}$} $N = 2^n$ for some $n \in \mathbb{N}$.
		
		\begin{definition}\label{quantum_state_definition}
			Consider a quantum system represented in a state space $\mathcal{H}$. A \textit{quantum state} is a straight line
			\begin{equation}
				\mathbb{C}\mathbf{v}\defeq \left\{\lambda \mathbf{v} \ | \ \lambda\in\mathbb{C}\right\},
			\end{equation}
			for some unit vector $\mathbf{v}\in \mathcal{H}$. We name the vector $\mathbf{v}$ \textit{representative} of the quantum state.
		\end{definition}
		In the language of projective geometry, a quantum state is an element of the projective space $\mathbb{P}(\mathcal{H})$ defined from $\mathcal{H}$.
		
		To simplify the terminology, a representative of a quantum state is often referred to as quantum state.
		
		\begin{remark}[Composition of quantum systems]\label{remak_quantum_composition}
			The above space represents the composition of $n$ quantum systems. In Quantum Mechanics the state space of the composition of physical systems with respectively state spaces $V_1,\dots,V_n$ is the tensor product $V_1\otimes\dots\otimes V_n$ (see \cite[subsection 1.4.2 at page 34]{d2021introduction} and \cite[Postulate 4 at page 94]{nielsen2010quantum}).
			
			A quantum state, represented by a vector $\mathbf{v}\in \mathcal{H}$, is said to be in entangled state if there are no $\mathbf{v}_1\in \mathbb{C}^2,\dots,\mathbf{v}_n\in \mathbb{C}^2$, such that
			\begin{equation}
				\mathbf{v}=\mathbf{v}_1\otimes \dots \otimes \mathbf{v}_n.
			\end{equation}
			\textit{Entanglement} is one of the most important concepts in Quantum Mechanics and was conceived along the celebrated dialogue among Einstein-Podolsky-Rosen \cite{einstein1935can}, Schr{\"o}dinger \cite{schrodinger1935gegenwartige} (translated in \cite{trimmer1980present}) and Bell \cite{bell1964einstein}.
			
			\textit{Quantum Superposition}, \textit{Entanglement} and their related paradoxes are the bases of Quantum Advantage (QA) and several Quantum Communication protocols, like Quantum Key Distribution (QKD) 
			%https://courses.xpro.mit.edu/learn/course/course-v1:xPRO+QCFx2+R14/block-v1:xPRO+QCFx2+R14+type@sequential+block@57ea74a698d0466c895f78cd746a5be7/block-v1:xPRO+QCFx2+R14+type@vertical+block@2512b8d116344421bbe0d8b645332a6c
			Ekert91 \cite{ekert1991quantum} and BBM92 \cite{bennett1992quantum}.
		\end{remark}
		
		Given a state space as \eqref{state_space}, we postulate the time-evolution of the quantum system is described by the {S}chr\"{o}dinger equation
%		\begin{equation}\label{}
%			i\hbar\frac{d}{dt}|\mathbf{\psi}(t)\rangle = H(t)|\mathbf{\psi}(t)\rangle\\
%		\end{equation}
		\begin{equation}\label{ger equation}
			i\frac{d}{dt}\mathbf{\psi}(t) = \big(H_0 + \sum_{j=1}^{m}u_j(t)H_j\big)\mathbf{\psi}(t)\footnote{In Quantum Computing (QC), the {S}chr\"{o}dinger equation is physically realized rather than merely simulated. Indeed, a Quantum Computer is fundamentally a physical system governed by the {S}chr\"{o}dinger equation itself. The solution at the final time, $\psi(T)$, is retrieved by measuring the quantum system. This process of quantum measurement is defined by Postulate 3 of quantum mechanics (see, e.g., \cite[subsections 2.2.3 and 2.2.5]{nielsen2010quantum}). Typically, multiple measurements must be performed to accurately estimate $\psi(T)$. The only exception occurs when $\psi(T)$ lies within an eigenspace of the self-adjoint matrix being measured; in that scenario, if the complex dimension of the eigenspace is one, a single measurement suffices to identify which eigenspace contains the vector $\psi(T)$.},\hspace{0.06 cm}t\in(0,T),
		\end{equation}
		where
		\begin{itemize}
			\item the state of the control system is a function of time $\mathbf{\psi}:[0,T]\longrightarrow \mathcal{H}$ evolving in the state space $\mathcal{H}$;
			\item the drift hamiltonian $H_0$ describes the free dynamics;
			\item the scalar controls $u_j:[0,T] \longrightarrow \mathbb{R}$ are $L^{\infty}(0,T)$ functions acting in the system by multiplication by the respective control hamiltonians $H_j$;
			\item both the drift hamiltonian $H_0$ and the control hamiltonians $H_j$ are self-adjoint;
			\item the time horizon $T>0$ is the \textit{decoherence time} for the quantum system (see DiVincenzo criterion D5 \cite{divincenzo2000physical}, the notes \cite{oliver2013superconducting} and \cite[chapter 7]{nielsen2010quantum}).
		\end{itemize}
	
		The equivalence of the above postulate and the time discrete one (where quantum evolution is described by unitary transformations) can be found in \cite[subsection 2.2.2]{nielsen2010quantum}.
		
		Studying the control properties of the above equation is the object of quantum Control Theory (see, e.g., \cite{zuazua2002remarks,privat2016optimal,koch2022quantum,dong2010quantum,boscain2021introduction} and references therein), applied directly to QFT problem in \cite{martin2020digital}.
		
		\begin{remark}
			Whenever a quantum hardware (like \href{https://quantum.cloud.ibm.com/computers?system=ibm_brisbane}{ibm$\_$brisbane}) is fixed, the state space $\mathcal{H}$, the drift Hamiltonian $H_0$, the number of controls $m$ and the control hamiltonians $H_j$ are fixed.
		\end{remark}
	
		\begin{remark}\label{matrix_equation}
			In the context of \eqref{ger equation}, note that, given the solution $U:[0,T]\longrightarrow \mathcal{M}_{N\times N}(\mathbb{C})$\footnote{We denoted by $\mathcal{M}_{N\times N}(\mathbb{C})$ the space of $N\times N$ matrices with complex entries.} to the matrix Cauchy problem
			\begin{equation}\label{matrix_schrodinger}
				\begin{cases}
					i\frac{d}{dt}U(t) = \big(H_0 + \sum_{j=1}^{m}u_j(t)H_j\big)U(t),& t\in (0,T)\\
					U(0)=I,\\
				\end{cases}
			\end{equation}
			the function $\psi(t)\defeq U(t)\psi_0$ is the solution to the 
			\begin{equation}\label{}
				\begin{cases}
					i\frac{d}{dt}\mathbf{\psi}(t) = \big(H_0 + \sum_{j=1}^{m}u_j(t)H_j\big)\mathbf{\psi}(t),& t\in (0,T)\\
					\psi(0)=\psi_0,\\
				\end{cases}
			\end{equation}
			for some initial condition $\psi_0\in\mathcal{H}$.
		\end{remark}
		
		\begin{remark}\label{remark_well_posedeness}
%			By \cite[Theorem 54 at page 476]{sontag1998mathematical}, for every initial datum $\psi_0\in\mathcal{H}$ and controls
%			% \forall j \in \left\{1,\dots,m\right\}
%			 $u_j\in L^{\infty}(0,T)$, there exists a unique solution $\mathbf{\psi}\in H^1(0,T)$ to
%			\begin{equation}\label{}
%				\begin{cases}
%					i\frac{d}{dt}\mathbf{\psi}(t) = \big(H_0 + \sum_{j=1}^{m}u_j(t)H_j\big)\mathbf{\psi}(t),& t\in (0,T)\\
%					\psi(0)=\psi_0.\\
%				\end{cases}
%			\end{equation}
			
			By \cite[Theorem 54 at page 476]{sontag1998mathematical}, for
			% every
			 controls
			% \forall j \in \left\{1,\dots,m\right\}
			$u_j\in L^{\infty}(0,T)$, there exists a unique solution $U(t)\in H^1(0,T)$ to the matrix Cauchy problem \eqref{matrix_schrodinger}.
		\end{remark}
	
		\section{Quantum computation as a controllability problem}
		\label{subchapter:Quantum computation as a controllability problem}
		
		In a quantum algorithm, we have three steps
		\begin{enumerate}
			\item[\textit{Step 1}] \ \textbf{Initialize}, where the quantum system is initialized to a quantum state $\mathbf{\psi}_0\in \mathcal{H}$;
			\item[\textit{Step 2}] \ \textbf{Compute}, where the quantum system evolves, from the initial state $\mathbf{\psi}_0\in \mathcal{H}$ to a final state $\mathbf{\psi}_1\in \mathcal{H}$, employing in \eqref{ger equation} specific controls;
			\item[\textit{Step 3}] \ \textbf{Measure}, where measurement is performed to check the result.
		\end{enumerate}
	
		In this manuscript, we focus on the study of the computation (the second step). In what follows we define a quantum computation problem in terms of the operator controllability of the {S}chr\"{o}dinger equation \eqref{matrix_schrodinger}. Arbitrarily fix $M>0$ and define
		\begin{equation}
			L^{\infty}_{M}(0,T)\defeq \left\{u:[0,T]\longrightarrow \mathbb{R} \ | \ |u(t)|\leq M, \ \mbox{a.e.} \ t\in [0,T]\right\}.
		\end{equation}
		Our control space is
		\begin{equation}\label{control_functions_space}
			\mathscr{U}^T\defeq \left\{(u_j)_{j=1}^m \ | \ u_j\in L^{\infty}_{M}(0,T) \ \mbox{for all  }j\in\left\{1,\dots,m\right\}\right\}=L^{\infty}(0,T)^m.
		\end{equation}
%		\begin{equation*}
%				\mathscr{U}^T\defeq \left\{u:[0,T] \longrightarrow \mathbb{R} \ \Bigg| \ u(t)\defeq \begin{dcases}
%						\overline{u}_1 \quad &t\in [0,t_1)\\
%						\dots \\
%						\overline{u}_j \quad &t\in [t_{j-1},t_{j})\\
%						\dots \\
%						\overline{u}_p \quad &t\in [t_{p-1},T]\\
%					\end{dcases}, \\
%				n\in \mathbb{N}\setminus \left\{0\right\}, \ 0<t_1<\dots<t_{p-1}<T\right\}
%			\end{equation*}
		
		Let us formulate a quantum computation problem as an operator controllability problem (see, e.g., \cite[Definition 2 at page 11]{dong2010quantum} and \cite[chapter 3]{d2021introduction}).
		
		\begin{definition}\label{def_1}
			Suppose we have a quantum system defined by a state space $\mathcal{H}$ as in \eqref{state_space}. Let
			\begin{equation}
				\Gamma:\mathcal{H}\longrightarrow \mathcal{H}
			\end{equation}
			be a unitary operator.
			The problem of computing $\Gamma$ is defined as the problem of finding control functions
			\begin{equation}
				u_j:[0,T] \longrightarrow [-M,M],\hspace{0.06 cm}j=1,\dots,m,
			\end{equation}
			in the space of functions of functions $\mathscr{U}^T$, such that the following operator controllability problem is satisfied
			\begin{equation}\label{label_controllability_problem}
				\begin{cases}
					i\frac{d}{dt}U(t) = \big(H_0 + \sum_{j=1}^{m}u_j(t)H_j\big)U(t),& t\in (0,T)\\
					U(0)=I,\\
					U(T)=\Gamma.\\
				\end{cases}
			\end{equation}
			We define $\mathscr{U}^T_{\mbox{\tiny{ad}}}$ the set of controls $(u_j)_{j=1}^m\in \mathscr{U}^T$, such that the above operator (or simultaneous) controllability property holds and
			\begin{equation}\label{space_admissible_controls}
				\mathscr{U}_{\mbox{\tiny{ad}}}\defeq \bigcup_{T>0}\mathscr{U}^T_{\mbox{\tiny{ad}}}.
			\end{equation}
			
			The {S}chr\"{o}dinger equation \eqref{ger equation} is said to be operator controllable if, for any unitary operator $\Gamma:\mathcal{H}\longrightarrow \mathcal{H}$, there exist a time horizon $T > 0$ and control functions $(u_j)_{j=1}^m\in \mathscr{U}^T$, such
			that \eqref{label_controllability_problem} is satisfied.
		\end{definition}
		
		The above operator controllability is verified in case the Lie algebra generated by
		\begin{equation}
			\left\{-iH_0,-iH_1,\dots,-iH_j,\dots,-iH_m\right\}
		\end{equation}
		equals the Lie algebra $su(N)=su(2^n)$ (\cite[Theorem 3.2.1 at page 89]{d2021introduction}).
		
		Let us now define the concept of Quantum Advantage (QA) by Optimal Control.
		
		Assume \eqref{ger equation} is operator controllable. For some target unitary operator $\Gamma:\mathcal{H}\longrightarrow \mathcal{H}$, define the minimal controllability time\footnote{under constraints $|u(t)|\leq M$ in the control space \eqref{control_functions_space}}
		\begin{equation}\label{def_min_time}
			T_{\mbox{\tiny{min}}} \defeq \inf\left\{T>0 \ \big| \ \exists u\in \mathscr{U}_{\mbox{\tiny{ad}}}, \  U(T)=\Gamma\right\},
		\end{equation}
		where $U:[0,+\infty)\longrightarrow \mathcal{M}_{N\times N}(\mathbb{C})$ represents the state, solution to the Cauchy problem
		\begin{equation}\label{matrix_schrodinger_inftime}
			\begin{cases}
				i\frac{d}{dt}U(t) = \big(H_0 + \sum_{j=1}^{m}u_j(t)H_j\big)U(t),& t\in (0,+\infty)\\
				U(0)=I,\\
			\end{cases}
		\end{equation}
		with control $u=(u_j)_{j=1}^m$
		%, the Hamiltonian matrices $\left\{H_j\right\}_{j=0}^{m}$ being defined in \eqref{ger equation}
		.
		
		The above definition is well-posed. Indeed, for every target operator $\Gamma:\mathcal{H}\longrightarrow \mathcal{H}$, by operator controllability, there exists a time horizon $\overline{T}>0$ and a control $u\in \mathscr{U}_{\mbox{\tiny{ad}}}$, such that $U(\overline{T})=\Gamma$, whence
		\begin{equation}
			\inf\left\{T>0 \ \big| \ \exists u\in \mathscr{U}_{\mbox{\tiny{ad}}}, \  U(T)=\Gamma\right\}\leq \overline{T}<+\infty;
		\end{equation}
		
		Moreover, the infimum in \eqref{def_min_time} is achieved, as we illustrate in the following remark.
		\begin{remark}\label{remark_min_time}
			Supposing \eqref{ger equation} is operator controllable, there exists a control $u_{\mbox{\tiny{min}}}\in \mathscr{U}_{\mbox{\tiny{ad}}}$, such that the unique solution $U_{\mbox{\tiny{min}}}$ to \eqref{matrix_schrodinger_inftime} with control $u_{\mbox{\tiny{min}}}$ satisfies the final condition $U_{\mbox{\tiny{min}}}(T_{\mbox{\tiny{min}}})=\Gamma$. Namely,
			\begin{equation}
				\inf\left\{T>0 \ \big| \ \exists u\in \mathscr{U}_{\mbox{\tiny{ad}}}, \  U(T)=\Gamma\right\}=\min\left\{T>0 \ \big| \ \exists u\in \mathscr{U}_{\mbox{\tiny{ad}}}, \  U(T)=\Gamma\right\}.
			\end{equation}
			This can be proved by Direct Methods in the Calculus of Variations (DMCV) \cite{dacorogna2007direct}, employing the precompactness given by the constraints $|u(t)|\leq M$.
		\end{remark}
		We conjecture the control in the minimal time is of bang-bang form. We imposed the constraint $|u(t)|\leq M$ in view of our concrete use case. Checking whether the infimum is still achieved, removing the constraint, is an interesting research line.
		
		A large literature is available on time optimal Quantum Control. See, e.g., \cite{agrachev2017note,boscain2021introduction}, \cite[section 7.4]{d2021introduction} and references therein. Moreover, lower bounds for minimal time can be obtained by the theory of Quantum Speed Limit \cite{nielsen2006geometry,carlini2007time}.
		
		\begin{definition}\label{quantum_advantage_definition}
			Let $\mathcal{H}$ be a state space as \eqref{state_space} and assume we are equipped with a quantum hardware defined by the {S}chr\"{o}dinger equation \eqref{ger equation}. In the framework of definition \ref{def_1}, consider a computational problem given by a unitary operator $\Gamma$ for which the best known classical algorithm requires a computing time\footnote{All along this manuscript, we take the second (s) as unit of measurement of time.} exponential in the number of qubits $n$. We say that for this problem there exists (exponential) Quantum Advantage (QA) if
			\begin{enumerate}
				\item[a)] operator controllability holds (namely $\mathscr{U}_{\mbox{\tiny{ad}}}$ is nonempty);
				\item[b)] there exists a constant $C$ and an integer $p$ (both independent of $n$), such that
				\begin{equation}\label{label_t_min_estimate}
					T_{\mbox{\tiny{min}}}\leq Cn^p.
				\end{equation}
			\end{enumerate}
			% [CHECK]. Quantum Adavatage seems to depend of the specific quantum computer employed.
		\end{definition}
		
		\begin{remark}[Polynomial Quantum Advantage]\label{quantum_advantage_definition_remark}
			The above definition is for exponential Quantum Advantage. An analogous definition might be given for polynomial Quantum Advantage. A celebrated quantum algorithm with polynomial Quantum Advantage is the Grover's algorithm \cite{grover1996fast,grover1997quantum}.
		\end{remark}

		% [CHECK] qdit

		\chapter{Quantum Advantage for the Quantum Fourier Transform}
		\label{chapter:QA_QFT}
		
		The purpose of this chapter is to show how Quantum Advantage is achieved by superconducting computers. As we mentioned, this is a well-known result (see, e.g. \cite[section 5.1]{nielsen2010quantum}).
		
		\section{Superconducting quantum computers}
		\label{chapter:Superconducting quantum computers}
		
		This section introduces a controlled {S}chr\"{o}dinger equation for superconducting quantum computers. The definition of Pauli matrices acting on one qubit can be found in the last subsection. For the sake of concreteness, we focus on \texttt{ibm\_brisbane} Quantum Computer. Our analysis applies to other quantum computers, like\\ \href{https://www.basquequantum.eus/en/ibm}{\texttt{ibm\_BasqueCountry}}.
		
		\subsection{Quantum State Space}
		The quantum state space of the \texttt{ibm\_brisbane} quantum computer is
		\begin{equation}\label{ibmbrisbanestatespace}
			\mathcal{H} = \underbrace{\mathbb{C}^2 \otimes \dots \otimes \mathbb{C}^2}_{127\text{ times}} = \mathbb{C}^{2^{127}}
		\end{equation}
		corresponding to 127 qubits (whence $\dim(\mathcal{H})=2^{127}> 10^{38}$).
		
		\subsection{Coupling map}\label{subchapter:couplingmap}
		The coupling map $\mathcal{E}$ is a directed graph defining connection among physical qubits in the computer (see figure \ref{grah_1}).
		\begin{itemize}
			\item The graph nodes are the physical qubits.
			\item As we shall see, the graph edges represent permitted control operations across different qubits.
		\end{itemize}
		See \cite{mckay2018qiskit,tindall2024confinement,pelofske2024scaling} and the \href{https://quantum.cloud.ibm.com/docs/en/api/qiskit/qiskit.transpiler.CouplingMap}{IBM documentation on Coupling Map}. It is possible to get $\mathcal{E}$, by running in Qiskit the command \verb!IBMBackend('ibm_brisbane').coupling_map! (as in \href{https://quantum.cloud.ibm.com/docs/en/api/qiskit-ibm-runtime/ibm-backend}{IBMBackend}).
		
		\subsection{Controlled Schr\"odinger Equation}
		The evolution of the system is governed by the bilinear Schr\"odinger equation:
		\begin{equation}\label{ibm_equation_control}
			i\frac{d}{dt}\mathbf{\psi}(t) = \big(H_0 + \sum_{j=1}^{m}u_j(t)H_j\big)\mathbf{\psi}(t),\hspace{0.06 cm}t\in(0,T),
		\end{equation}
		where
		\begin{itemize}
			\item $H_0$ is the \textbf{drift Hamiltonian} (time-independent),
			\item $H_j$ are the \textbf{control Hamiltonians} associated with time-dependent scalar controls $u_j(t)$.
		\end{itemize}
		
		\subsubsection{Drift Hamiltonian $H_0$}
		The drift Hamiltonian is a $N\times N$ self-adjoint matrix
		% with complex entries
		 with zero trace.
		%, such that its semigroup
%		\begin{equation}
%			\Lambda:\mathbb{R}\longrightarrow \mathcal{M}_{N\times N}(\mathbb{C})
%		\end{equation}
%		\begin{equation}\label{semigroup}
%			t\longmapsto \exp(-iH_0 t)
%		\end{equation}
%		is periodic. This last condition is satisfied if eigenvalues $\lambda_i$ are \textit{commensurate}, namely
%		\begin{equation}
%			\frac{\lambda_j}{\lambda_i}\in \mathbb{Q},\hspace{0.6 cm}\forall i, j\mbox{ such that }\lambda_i\neq 0.
%		\end{equation}
		
		An example of drift Hamiltonian is in section \ref{chapter:drift_hamiltonian} of the Appendix.
		
		\subsubsection{Control Hamiltonians $H_j$}\label{subchapter:control_hamiltonians}
		The control Hamiltonians represent the microwave drives and cross-resonance interactions:
		\begin{itemize}
			\item \textbf{Single-qubit drives:}
			\begin{equation}
				H_{\text{drive},k}(t) = u_k^X(t) X_k + u_k^Y(t) Y_k
			\end{equation}
			for each qubit $k = 1, \dots, 127$ (see \cite{mckay2018qiskit} and \cite[chapter 3]{d2021introduction}).
			\item \textbf{Cross-resonance drives:}
			\begin{equation}
				H_{\text{CR},c\to t}(t) = v_{ct}(t) Z_c \otimes X_t
			\end{equation}
			for each control-target qubit pair $(c,t) \in \mathcal{E}$ (see \cite{rigetti2010fully,sivarajah2020t}), $\mathcal{E}$ representing the coupling map, presented in subsection \ref{subchapter:couplingmap}.
		\end{itemize}
	
	The \texttt{ibm\_brisbane} processor utilizes a heavy-hex lattice topology, optimizing the layout for cross-resonance gates. Below is a schematic representation of its qubit connectivity.

	\input{ibm_brisbane_coupling_map.tex}

	The {S}chr\"{o}dinger equation \eqref{ibm_equation_control} can be rewritten as
	\begin{equation}\label{ibm_brisbane_equation}
		i \frac{d}{dt} {\psi(t)} = \Bigg(
		H_0
		+ \sum_{k=1}^{127} \left( u_k^X(t) X_k + u_k^Y(t) Y_k \right)
		+ \sum_{(c,t) \in \mathcal{E}} v_{ct}(t) Z_c \otimes X_t
		\Bigg) {\psi(t)}.
	\end{equation}
	
	Further references for superconducting quantum computers are \cite{wittler2021integrated,krantz2019quantum,kjaergaard2020superconducting,oliver2013superconducting}.
	%https://learning.quantum.ibm.com/tutorial/improving-estimation-of-expectation-values-with-operator-backpropagation
	
	\subsection{Pauli matrices}\label{subchapter:Pauli matrices}
	
	We use the standard Pauli matrices:
	\[
	I = \sigma_0 = \begin{pmatrix} 1 & 0 \\ 0 & 1 \end{pmatrix}, \quad
	X = \sigma_x = \begin{pmatrix} 0 & 1 \\ 1 & 0 \end{pmatrix}, \quad
	Y = \sigma_y = \begin{pmatrix} 0 & -i \\ i & 0 \end{pmatrix}, \quad
	Z = \sigma_z = \begin{pmatrix} 1 & 0 \\ 0 & -1 \end{pmatrix},
	\]
	in the notation of \cite[subsection 2.1.3 at page 65]{nielsen2010quantum}.
	
	Note that
	\begin{equation}\label{square_pauli_matrices}
		X^2=I,\hspace{0.3 cm}Y^2=I,\hspace{0.3 cm}Z^2=I.
	\end{equation}
	
	When we write $X_k$, $Y_k$, or $Z_k$, we mean
	\begin{itemize}
		\item the Pauli matrix ($X$, $Y$, or $Z$) acts on qubit $k$,
		\item the identity matrix $I$ acts on all other qubits.
	\end{itemize}
	
	Formally, for a 127-qubit system:
	\[
	X_k = I^{\otimes k-1} \otimes X \otimes I^{\otimes (127 - k)}=\underbrace{I \otimes \dots \otimes I}_{k-1\text{ times}}\otimes X \otimes \underbrace{I \otimes \dots \otimes I}_{127 - k\text{ times}},
	\]
	and similarly for $Y_k$ and $Z_k$.
	
	For two-qubit interaction terms such as $Z_c \otimes X_t$, this operator means:
	\begin{itemize}
		\item apply $Z$ to qubit $c$,
		\item apply $X$ to qubit $t$,
		\item apply the identity to all remaining qubits.
	\end{itemize}
	
	\subsection{Controllability properties}\label{subchapter:controllability_properties}
	
	Inspired by \cite{albertini2002lie,d2021introduction,zeier2011symmetry}, we are going to study the controllability properties of \eqref{ibm_brisbane_equation}. Note that, for the next proposition, the time horizon $T>0$ is not fixed\footnote{The time horizon $T$ has to be sufficiently large for controllability to hold. In the physical implementation of quantum computation, this means the DiVincenzo criterion D5 has to be satisfied \cite{divincenzo2000physical}.}.
	
	\begin{proposition}\label{prop_1}
		Assume the graph $\mathcal{E}$ is connected. Then, the {S}chr\"{o}dinger equation \eqref{ibm_brisbane_equation} is operator controllable.
	\end{proposition}
	The proof is going to be based on proving that
	\begin{equation}
		\left\{-iH_0,-iX_1,-iY_1,\dots,-iX_k,-iY_k,\dots,-iX_n,-iY_n\right\}
	\end{equation}
	generates the Lie algebra $su(N)=su(2^n)$ (which, by \cite[Theorem 3.2.1 at page 89]{d2021introduction}, implies controllability). The proof is made of two main steps
	\begin{itemize}
		\item[Step 1] Show controllability in each  component\footnote{in the sense of the tensor product} of the state space
		\begin{equation}\label{}
			\mathcal{H} = \underbrace{\mathbb{C}^2 \otimes \dots \otimes \mathbb{C}^2}_{127\text{ times}} = \mathbb{C}^{2^{127}}.
		\end{equation}
		\item[Step 2] Prove that the free dynamics hamiltonian $H_0$ is able to couple the different components.
	\end{itemize}
	
	We are not going to employ the drift-Hamiltonian ($H_0$); this might be useful to reduce the number of control switches.
	
	\begin{remark}\label{reamrk_piece-wise contant controls}
		In the proof of Proposition \ref{prop_1}, determined control functions $u_j:[0,T]\longrightarrow \mathbb{R}$ are piece-wise constant with values in $[-M,M]$, namely
		\begin{equation}
			u_j(t)= \begin{dcases}
				\overline{u}_1 \quad &t\in [0,t_1)\\
				\dots \\
				\overline{u}_l \quad &t\in [t_{l-1},t_{l})\\
				\dots \\
				\overline{u}_p \quad &t\in [t_{p-1},T],\\
			\end{dcases},
		\end{equation}
		where
		\begin{itemize}
			\item for $j\in \left\{1,\dots,p\right\}$, the values $\overline{u}_j\in [-M,M]$;
			\item $p$ is a natural number indicating the number of switches;
			\item $0<t_1<\dots<t_{p-1}<T$ defines a partition of the interval $[0, T]$.
		\end{itemize}
	\end{remark}
	
	Before starting the proof, let us introduce some notation.
	
	\textbf{Notation for the proof}
	
	\begin{itemize}
		\item We define the Lie algebra
		\begin{equation}
			\mathscr{L}_0\defeq \mbox{Lie algebra generated}\left\{-iH_0,-iX_1,-iY_1,\dots,-iX_k,-iY_k,\dots,-iX_n,-iY_n\right\}\footnote{this is a Lie algebra with scalars in $\mathbb{R}$};
		\end{equation}
		\item Consider the multiples of the Pauli matrices
		\[
		\overline{\sigma}_x = \frac{i}{2}X = \frac{1}{2}\begin{pmatrix} 0 & i \\ i & 0 \end{pmatrix}, \quad
		\overline{\sigma}_y = -\frac{i}{2}Y = \frac{1}{2}\begin{pmatrix} 0 & -1 \\ 1 & 0 \end{pmatrix}, \quad
		\overline{\sigma}_z = \frac{i}{2}Z = \frac{1}{2}\begin{pmatrix} i & 0 \\ 0 & -i \end{pmatrix}.
		\]
		This allows rewriting
		\begin{equation}
			\mathscr{L}_0= \mbox{Lie algebra generated}\left\{-iH_0,\overline{\sigma}_{x,1},\overline{\sigma}_{y,1},\dots,\overline{\sigma}_{x,k},\overline{\sigma}_{y,k},\dots,\overline{\sigma}_{x,n},\overline{\sigma}_{y,n}\right\}.
		\end{equation}
		\item Define $\overline{\sigma}_0\defeq I$.
		\item Let $A$ be an arbitrary $2\times 2$ matrix with complex entries and $k\in \left\{1,\dots,n\right\}$. We define
		\[
		A_k \defeq I^{\otimes k-1} \otimes A \otimes I^{\otimes (127 - k)}=\underbrace{I \otimes \dots \otimes I}_{k-1\text{ times}}\otimes A \otimes \underbrace{I \otimes \dots \otimes I}_{127 - k\text{ times}}.
		\]
		\item Let $A$ and $B$ be two arbitrary $2\times 2$ matrices with complex entries, $k_1\in \left\{1,\dots,n\right\}$ and $k_2\in \left\{1,\dots,n\right\}$, with $k_1< k_2$. We define
		\[
		A_{k_1}B_{k_2} \defeq I^{\otimes k_1-1} \otimes A \otimes I^{\otimes (k_2 - k_1 - 1)} \otimes B \otimes I^{\otimes (127 - k_2)}=\underbrace{I \otimes \dots \otimes I}_{k_1-1\text{ times}}\otimes A \otimes \underbrace{I \otimes \dots \otimes I}_{k_2 - k_1 - 1\text{ times}} \otimes B \otimes \underbrace{I \otimes \dots \otimes I}_{127 - k_2\text{ times}}.
		\]
		\item Let $A$ and $B$ be two arbitrary $2\times 2$ matrices with complex entries, $k_1\in \left\{1,\dots,n\right\}$ and $k_2\in \left\{1,\dots,n\right\}$, with $k_1> k_2$. We define
		\[
		A_{k_1}B_{k_2} \defeq I^{\otimes k_1-1} \otimes B \otimes I^{\otimes (k_2 - k_1 - 1)} \otimes A \otimes I^{\otimes (127 - k_2)}=\underbrace{I \otimes \dots \otimes I}_{k_1-1\text{ times}}\otimes B \otimes \underbrace{I \otimes \dots \otimes I}_{k_2 - k_1 - 1\text{ times}} \otimes A \otimes \underbrace{I \otimes \dots \otimes I}_{127 - k_2\text{ times}}.
		\]
		\item Define the group $(\left\{0,x,y,z\right\},+)$, where $x+y\defeq z$, $y+z\defeq x$, $z+x\defeq y$ and $0$ is the neutral element\footnote{This group is indeed isomorphic to $(\mathbb{Z}_4,+)$}. All the operation on indices $q\in \left\{0,x,y,z\right\}$ are intended in this group.
	\end{itemize}
	
	The above definition of the operation $+$ in $\left\{0,x,y,z\right\}$ is motivated by the following remark.
	
	\begin{remark}
		The matrices $\overline{\sigma}_x$, $\overline{\sigma}_y$ and $\overline{\sigma}_z$ generate $su(2)$ (the space of $2\times2$ skew adjoint complex matrices, with zero trace) and they satisfy the commutation relations
		\begin{equation}\label{commutation_relation}
			[\overline{\sigma}_x,\overline{\sigma}_y]=\overline{\sigma}_z,\hspace{0.3 cm}[\overline{\sigma}_y,\overline{\sigma}_z]=\overline{\sigma}_x,\hspace{0.3 cm}[\overline{\sigma}_z,\overline{\sigma}_x]=\overline{\sigma}_y.
		\end{equation}
	\end{remark}
	
	\begin{proof}[Proof of Proposition \ref{prop_1}]
%		\begin{equation}
%			su(N)\supseteq \mathscr{L}_0
%		\end{equation}
%		ok.\\
		\textit{Step 0} \ \textbf{Determine a base for $\mathscr{L}_0$.} \\
		Consider $\mathcal{M}_{N\times N}(\mathbb{C})$, the vector space of $N\times N$ matrices with complex entries. This is a vector space both over $\mathbb{R}$ and $\mathbb{C}$. All along this proof, we are going to treat it as a vector space over $\mathbb{R}$.
		
		$su(N)$ is a vector subspace of $\mathcal{M}_{N\times N}(\mathbb{C})$, with dimension
		\begin{equation}\label{su_2n_dimension}
			\dim_{\mathbb{R}}(su(N))=4^n-1.
		\end{equation}
		
		Define the set of matrices
		% system
		\begin{equation}\label{B}
			\mathcal{B}\defeq \left\{i\sigma_{q_1}\otimes \dots\otimes \sigma_{q_j}\otimes \dots\otimes \sigma_{q_n} \ | \ q_j\in \left\{0,x,y,z\right\}, \ j\in \left\{1,\dots,n\right\}\right\}\setminus \left\{I\right\}.
		\end{equation}
		First of all, since the identity matrix $I$
		% is not in $\mathcal{B}$
		 has been removed, $\mathcal{B}$ is contained in $su(N)$. Secondly, it is independent, as a system of vectors of $\mathcal{M}_{N\times N}(\mathbb{C})$ over $\mathbb{R}$. Moreover, the number of element of $\mathcal{B}$ equals the dimension \eqref{su_2n_dimension}. This allows proving $\mathcal{B}$ is a base of $su(N)$ over $\mathbb{R}$. To conclude, it then suffices to prove that $\mathscr{L}_0$ contains $\mathcal{B}$ and to use \cite[Theorem 3.2.1 at page 89]{d2021introduction}.\\
		\textit{Step 1} \ \textbf{Controllability in each component of the state space.} \\
		For every $k \in \left\{1,\dots,n\right\}$, by \eqref{commutation_relation}, we have
		\begin{equation}
			[\overline{\sigma}_{x,k},\overline{\sigma}_{y,k}]=\overline{\sigma}_{z,k}.
		\end{equation}
	    Hence,
	    \begin{equation}
	    	\overline{\sigma}_{q,k}\in \mathscr{L}_0,\hspace{0.06 cm}\forall (q,k)\in \left\{x,y,z\right\}\times \left\{1,\dots,n\right\},
	    \end{equation}
        thus showing the controllability in each component.\\
		\textit{Step 2} \ \textbf{Coupling different components.} \\
		Cross-components controllability can be proved by repeatedly taking Lie brackets like
		\begin{equation}
			[i Z_c \otimes X_t,\overline{\sigma}_{q,k}],\hspace{0.3 cm}\mbox{for some $(c,t)\in \mathcal{E}$, $q\in \left\{x,y,z\right\}$ and $k\in\left\{1,\dots,n\right\}$},
		\end{equation}
		where we remind from subsection \ref{subchapter:control_hamiltonians} that $Z_c \otimes X_t$ represents one of the cross-resonance drives. In the computation of the above Lie brackets, commutation relations \eqref{commutation_relation} are employed.
		
	    Let us start by computing, for each $(k,l)\in \mathcal{E}$ and $q\in \left\{x,y\right\}$, the Lie bracket
	    \begin{equation}
	    	[i Z_k \otimes X_l,\overline{\sigma}_{q-z,k}] = i \sigma_{q,k}\sigma_{x,l},
	    \end{equation}
    	whence
    	\begin{equation}\label{statemnt_2}
    		i \sigma_{q,k}\sigma_{x,l}\in \mathscr{L}_0.
    	\end{equation}
    	At this point, taking the Lie bracket
    	\begin{equation}
    		[i \sigma_{q_1,k}\sigma_{x,l},i\sigma_{q_1,k}\sigma_{q_2-x,l}] = i\sigma_{q_1,k}\sigma_{q_2,l},
    	\end{equation}
    	for some $(k,l)\in \mathcal{E}$ and $(q_1,q_2)\in \left\{x,y,z\right\}\times \left\{y,z\right\}$, we show
    	\begin{equation}
    		i\sigma_{q_1,k}\sigma_{q_2,l}\in \mathscr{L}_0.
    	\end{equation}
    	
        We have proved that, for any edge $(k,l) \in \mathcal{E}$ and for every $(q_1,q_2)\in \left\{x,y,z\right\}^2$,
        \begin{equation}\label{statemnt_3}
        	i\sigma_{q_1,k}\sigma_{q_2,l}\in \mathscr{L}_0.
        \end{equation}
    	
    	Now, let us consider the case of $(k,l)\in \left\{1,\dots,n\right\}^2$, with $k<l$ and $(k,l)\notin \mathcal{E}$. At this point, we employ the assumption. Since $\mathcal{E}$ is connected, there exists, for some $P\in \mathbb{N}\setminus \left\{0\right\}$, a path of edges in $\mathcal{E}$ connecting $k$ and $l$\footnote{Checking whether there exists a relation between the discrete length $P$ of the path and the continuous time horizon $T$ (decoherence time) is an interesting research line.},
        \begin{equation}
        	\gamma:\left\{1,\dots,P\right\}\longrightarrow \left\{1,\dots,n\right\},
        \end{equation}
    	with
    	\begin{equation}
    		\begin{dcases}
    			\gamma(1) =k&\\
    			\dots \\
    			(\gamma(p-1),\gamma(p))\in \mathcal{E} &\forall p\in \left\{2,\dots,P\right\},\\
    			\dots \\
    			\gamma(P)=l&.\\
    		\end{dcases}
    	\end{equation}
    	Let us prove, by induction on $p\in \left\{2,\dots,P\right\}$, that
    	\begin{equation}\label{statemnt_4}
    		i\sigma_{q_1,k}\sigma_{q_2,\gamma(p)}\in \mathscr{L}_0,\hspace{0.3 cm}\forall (q_1,q_2)\in \left\{x,y,z\right\}^2.
    	\end{equation}
%    	\begin{equation}\label{statemnt_4}
%    		i\sigma_{q_1,k}\sigma_{q_2,\gamma(p)}\in \mathscr{L}_0,\hspace{0.3 cm}\forall p\in \left\{2,\dots,P\right\},\forall (q_1,q_2)\in \left\{x,y,z\right\}^2.
%    	\end{equation}
    	Since $(k, \gamma(2))\in \mathcal{E}$, by \eqref{statemnt_3}, the above assertion holds for $p=2$. Suppose now \eqref{statemnt_4} is verified for $p-1$ and let us prove it for $p$. By induction assumption, we have
    	\begin{equation}\label{statemnt_7}
    		i\sigma_{q_1,k}\sigma_{x,\gamma(p-1)}\in \mathscr{L}_0.
    	\end{equation}
    	Moreover, $(\gamma(p-1),\gamma(p))\in \mathcal{E}$. Hence, by \eqref{statemnt_3},
    	\begin{equation}\label{statemnt_9}
    		i\sigma_{y,\gamma(p-1)}\sigma_{q_2-2,\gamma(p)}\in \mathscr{L}_0.
    	\end{equation}
    
  		We take the Lie bracket of \eqref{statemnt_7} and \eqref{statemnt_9}
  		\begin{equation}
  			[i\sigma_{q_1,k}\sigma_{x,\gamma(p-1)},i\sigma_{y,\gamma(p-1)}\sigma_{q_2-2,\gamma(p)}]=i\sigma_{q_1,k}\sigma_{z,\gamma(p-1)}\sigma_{q_2-2,\gamma(p)},
  		\end{equation}
  		whence
  		\begin{equation}\label{statemnt_11}
  			i\sigma_{q_1,k}\sigma_{z,\gamma(p-1)}\sigma_{q_2-2,\gamma(p)} \in \mathscr{L}_0.
  		\end{equation}
  		Now, by \eqref{statemnt_3}
  		\begin{equation}\label{statemnt_19}
  			i\sigma_{z,\gamma(p-1)}\sigma_{q_2-1,\gamma(p)}\in \mathscr{L}_0.
  		\end{equation}
  	
  		We take the Lie bracket of \eqref{statemnt_11} and \eqref{statemnt_19}, getting
  		\begin{equation}\label{statemnt_21}
  			[i\sigma_{q_1,k}\sigma_{z,\gamma(p-1)}\sigma_{q_2-2,\gamma(p)},i\sigma_{z,\gamma(p-1)}\sigma_{q_2-1,\gamma(p)}]=i\sigma_{q_1,k}\sigma_{z,\gamma(p-1)}^2\sigma_{q_2,\gamma(p)}.
  		\end{equation}
  	
  		Now,
  		\begin{equation}\label{statemnt_24}
  			\sigma_{z}^2=Z^2=\begin{pmatrix} 1 & 0 \\ 0 & -1 \end{pmatrix}^2=\begin{pmatrix} 1 & 0 \\ 0 & 1 \end{pmatrix}.
  		\end{equation}
  		
  		Therefore, remembering \eqref{statemnt_21}, we obtain
  		\begin{equation}\label{statemnt_29}
  			i\sigma_{q_1,k}\sigma_{q_2,\gamma(p)}=[i\sigma_{q_1,k}\sigma_{z,\gamma(p-1)}\sigma_{q_2-2,\gamma(p)},i\sigma_{z,\gamma(p-1)}\sigma_{q_2-1,\gamma(p)}]\in \mathscr{L}_0,
  		\end{equation}
  		thus showing \eqref{statemnt_4}, which implies (for $p=P$)
  		\begin{equation}\label{statemnt_37}
  			i\sigma_{q_1,k}\sigma_{q_2,l}\in \mathscr{L}_0,\hspace{0.3 cm}\forall (q_1,q_2)\in \left\{x,y,z\right\}^2,
  		\end{equation}
  		as required.
  		
  		Let us now prove that
  		\begin{equation}\label{B_subsetL_0}
  			\mathcal{B}\subset \mathscr{L}_0,
  		\end{equation}
  		where the base $\mathcal{B}$ was defined in \eqref{B}. For every $(q_1,\dots,q_j,\dots,q_n)\in \left\{0,x,y,z\right\}^n$, consider
  		\begin{equation}
  			i\sigma_{q_1,1} \dots \sigma_{q_j,j} \dots \sigma_{q_n,n}
  		\end{equation}
  		and define the number of tensor components where there is a non identical action
  		\begin{equation}
  			n_{\mbox{\tiny{actions}}}\defeq \# \left\{j\in \left\{1,\dots,n\right\} \ | \ q_j\in \left\{x,y,z\right\}\right\}.
  		\end{equation}
  		We can rewrite $\mathcal{B}$ as
  		\begin{equation}
  			\mathcal{B}= \left\{\sigma_{q_1,1} \dots \sigma_{q_j,j} \dots \sigma_{q_n,n} \ | \ q_j\in \left\{0,x,y,z\right\}, n_{\mbox{\tiny{actions}}}\in \left\{1,\dots,n\right\}\right\}.
  		\end{equation}
  		Let us prove by induction over $n_{\mbox{\tiny{max}}}\in \left\{1,\dots,n\right\}$ that
  		\begin{equation}\label{proposition_to_be_proved_by_induction}
  			\left\{\sigma_{q_1,1} \dots \sigma_{q_j,j} \dots \sigma_{q_n,n} \ | \ q_j\in \left\{0,x,y,z\right\}, n_{\mbox{\tiny{actions}}}\in \left\{1,\dots,n_{\mbox{\tiny{max}}}\right\}\right\}\subset \mathscr{L}_0.
  		\end{equation}
  		From step 1, the above assertion follows for $n_{\mbox{\tiny{max}}}=1$ (base case). Let us show the inductive step. Assume the assertion for $n_{\mbox{\tiny{max}}}-1$ and prove it for $n_{\mbox{\tiny{max}}}$. By the induction assumption, \eqref{B_subsetL_0} holds for $n_{\mbox{\tiny{max}}}-1$, namely
  		\begin{equation}\label{induction_assumption}
  			\left\{\sigma_{q_1,1} \dots \sigma_{q_j,j} \dots \sigma_{q_n,n} \ | \ q_j\in \left\{0,x,y,z\right\}, n_{\mbox{\tiny{actions}}}\in \left\{1,\dots,n_{\mbox{\tiny{max}}}-1\right\}\right\}\subset \mathscr{L}_0.
  		\end{equation}
  		Now, take any $(q_1,\dots,q_j,\dots,q_n)\in \left\{0,x,y,z\right\}^n$, with
  		\begin{equation}
  			n_{\mbox{\tiny{actions}}}=n_{\mbox{\tiny{max}}}.
  		\end{equation}
  		On the one hand, by the induction assumption, we have
  		\begin{equation}
  			i\sigma_{q_1,1} \dots \sigma_{q_j,j} \dots \sigma_{q_{n_{\mbox{\tiny{max}}}-1}-2,n_{\mbox{\tiny{max}}}-1}\in \mathscr{L}_0.
  		\end{equation}
  		On the other hand, by \eqref{statemnt_37},
  		\begin{equation}
  			i\sigma_{q_{n_{\mbox{\tiny{max}}}-1}-1,n_{\mbox{\tiny{max}}}-1}\sigma_{q_{n_{\mbox{\tiny{max}}}},n_{\mbox{\tiny{max}}}}\in \mathscr{L}_0.
  		\end{equation}
  		Hence,
  		\begin{equation}
  			i\sigma_{q_1,1} \dots \sigma_{q_j,j} \dots \sigma_{q_{n_{\mbox{\tiny{max}}}-1},n_{\mbox{\tiny{max}}}-1}\sigma_{q_{n_{\mbox{\tiny{max}}}},n_{\mbox{\tiny{max}}}}
  		\end{equation}
  		\begin{equation}
  			=[i\sigma_{q_1,1} \dots \sigma_{q_j,j} \dots \sigma_{q_{n_{\mbox{\tiny{max}}}-1}-2,n_{\mbox{\tiny{max}}}-1},i\sigma_{q_{n_{\mbox{\tiny{max}}}-1}-1,n_{\mbox{\tiny{max}}}-1}\sigma_{q_{n_{\mbox{\tiny{max}}}},n_{\mbox{\tiny{max}}}}]\in \mathscr{L}_0,
  		\end{equation}
  		as required. Then, \eqref{proposition_to_be_proved_by_induction} is verified for $n_{\mbox{\tiny{max}}}$. By the induction principle, 
  		\begin{equation}\label{statemrnt_3678}
  			\mathcal{B}= \left\{\sigma_{q_1,1} \dots \sigma_{q_j,j} \dots \sigma_{q_n,n} \ | \ q_j\in \left\{0,x,y,z\right\}, n_{\mbox{\tiny{actions}}}\in \left\{1,\dots,n\right\}\right\}\subset \mathscr{L}_0.
  		\end{equation}
  		In step 0, we proved that $\mathcal{B}$ is a base for $su(N)$. Hence, 
  		\begin{equation}\label{statemrnt_367}
  			\mathscr{L}_0=su(N).
  		\end{equation}
  		Therefore, by \cite[Theorem 3.2.1 at page 89]{d2021introduction}, the {S}chr\"{o}dinger equation \eqref{ibm_brisbane_equation} is operator controllable. This finishes the proof.
%  		\begin{equation}
%  			su(N)\subseteq \mathscr{L}_0.
%  		\end{equation}
	\end{proof}
	
%	\subsection{Independence of Quantum Advantage (QA) from the Drift Hamiltonian $H_0$}\label{subchapter:Independence of Quantum Advantage from the Drift Hamiltonian $H_0$}
%	
%	[CHECK]
%	In definition \ref{quantum_advantage_definition}, we defined the concept of Quantum Advantage (QA). Let us show that, for quamtum systems like \eqref{ibm_brisbane_equation}, this concept is independent of the drift Hamiltonian $H_0$.
%	
%	Let $H_0^1$ and $H_0^2$ be two $N \times N$ self-adjoint matrices
%	% with complex entries
%	leading both to a periodic semigroup \eqref{semigroup}.
%	
%	First, we note the following.
%	
%	\begin{remark}\label{remark_->}
%		The trajectory $U^1(t)$ solves
%		\begin{equation}\label{}
%			\begin{cases}
%				i \frac{d}{dt} {U^1(t)} = \Bigg(
%				H^1_0
%				+ \sum_{k=1}^{127} \left( u_k^X(t) X_k + u_k^Y(t) Y_k \right)
%				+ \sum_{(c,t) \in \mathcal{E}} v_{ct}(t) Z_c \otimes X_t
%				\Bigg) {U^1(t)},& t\in (0,T)\\
%				U^1(0)=I,\\
%			\end{cases}
%		\end{equation}
%		if and only if
%		\begin{equation}\label{change_of_variable}
%			U^2(t)\defeq\exp(-iH^2_0t)\exp(iH^1_0t)U^1(t)
%		\end{equation}
%		No. This is wrong.
%		satisfies
%		\begin{equation}\label{}
%			\begin{cases}
%				i \frac{d}{dt} {U^2(t)} = \Bigg(
%				H^2_0
%				+ \sum_{k=1}^{127} \left( u_k^X(t) X_k + u_k^Y(t) Y_k \right)
%				+ \sum_{(c,t) \in \mathcal{E}} v_{ct}(t) Z_c \otimes X_t
%				\Bigg) {U^2(t)},& t\in (0,T)\\
%				U^2(0)=I.\\
%			\end{cases}
%		\end{equation}
%		
%	\end{remark}
%	
%	[CHECK]
	
	\section{Proof of the Quantum Advantage for the Quantum Fourier Transform}\label{sec:Proof of the Quantum Advantage for the Quantum Fourier Transform}

	In the introduction, we mentioned the Quantum Fourier Transform (QFT) as a central example of Quantum Advantage (QA). In this section, we rigorously verify that QA holds for the QFT within the control-theoretic framework of definition \ref{quantum_advantage_definition}. The argument is based on two ingredients:
	\begin{enumerate}
		\item a gate concatenation lemma, which exploits the left-invariance of the bilinear {S}chr\"{o}dinger equation;
		\item the well-known decomposition of the QFT into $O(n^2)$ elementary quantum gates (see, e.g. \cite[section 5.1]{nielsen2010quantum}).
	\end{enumerate}
	
	\subsection{Left-invariance and gate concatenation}\label{sec:gate_concatenation}
	
	The bilinear {S}chr\"{o}dinger equation \eqref{matrix_schrodinger} is \textit{left-invariant}: the dynamics depend on the controls but not on the current state $U(t)$ in a way that favors any particular group element. We exploit this property to bound the minimal time for a composite unitary in terms of the minimal times for its factors.
	
	\begin{lemma}[Gate concatenation]\label{lemma_gate_concatenation}
		Assume operator controllability holds for \eqref{ger equation}. Let $\Gamma_1, \Gamma_2 \in SU(N)$ be two target operators, and let $T_{\mbox{\tiny{min}}}(\Gamma_j)$ denote the minimal time defined in \eqref{def_min_time} for the target $\Gamma_j$ ($j=1,2$). Then, for the composed operator $\Gamma_2\Gamma_1$, we have
		\begin{equation}\label{gate_concatenation_bound}
			T_{\mbox{\tiny{min}}}(\Gamma_2\Gamma_1) \leq T_{\mbox{\tiny{min}}}(\Gamma_1) + T_{\mbox{\tiny{min}}}(\Gamma_2).
		\end{equation}
	\end{lemma}
	\begin{proof}
		\textit{Step 1} \ \textbf{Left-invariance.}\\
		Let $V_0 \in SU(N)$ be arbitrary. Consider the {S}chr\"{o}dinger equation with initial condition $V_0$
		\begin{equation}\label{matrix_schrodinger_V0}
			\begin{cases}
				i\frac{d}{dt}V(t) = \big(H_0 + \sum_{j=1}^{m}u_j(t)H_j\big)V(t),& t\in (0,T)\\
				V(0)=V_0.\\
			\end{cases}
		\end{equation}
		Define $W(t) \defeq V(t)V_0^{-1}$. Then $W$ solves
		\begin{equation}
			\begin{cases}
				i\frac{d}{dt}W(t) = \big(H_0 + \sum_{j=1}^{m}u_j(t)H_j\big)W(t),& t\in (0,T)\\
				W(0)=I.\\
			\end{cases}
		\end{equation}
		Hence, $V(T) = \Gamma V_0$ if and only if $W(T) = \Gamma$. This shows that the minimal time to steer from $V_0$ to $\Gamma V_0$ equals the minimal time to steer from $I$ to $\Gamma$, that is, $T_{\mbox{\tiny{min}}}(\Gamma)$.\\
		\textit{Step 2} \ \textbf{Concatenation.}\\
		By remark \ref{remark_min_time}, there exist controls
		\begin{eqnarray}
			u^{(1)} &\in& \mathscr{U}^{T_{\mbox{\tiny{min}}}(\Gamma_1)}_{\mbox{\tiny{ad}}},\nonumber\\
			u^{(2)} &\in& \mathscr{U}^{T_{\mbox{\tiny{min}}}(\Gamma_2)}_{\mbox{\tiny{ad}}},\nonumber
		\end{eqnarray}
		such that
		\begin{itemize}
			\item the solution $U^{(1)}$ to \eqref{matrix_schrodinger} with control $u^{(1)}$ satisfies $U^{(1)}(T_{\mbox{\tiny{min}}}(\Gamma_1)) = \Gamma_1$;
			\item the solution $U^{(2)}$ to \eqref{matrix_schrodinger} with control $u^{(2)}$ satisfies $U^{(2)}(T_{\mbox{\tiny{min}}}(\Gamma_2)) = \Gamma_2$.
		\end{itemize}
		Define the concatenated control $\hat{u}:[0,T_{\mbox{\tiny{min}}}(\Gamma_1)+T_{\mbox{\tiny{min}}}(\Gamma_2)]\longrightarrow \mathbb{R}^m$ by
		\begin{equation}
			\hat{u}(t) \coloneqq \begin{dcases}
				u^{(1)}(t), & t \in [0, T_{\mbox{\tiny{min}}}(\Gamma_1)),\\
				u^{(2)}(t - T_{\mbox{\tiny{min}}}(\Gamma_1)), & t \in [T_{\mbox{\tiny{min}}}(\Gamma_1), T_{\mbox{\tiny{min}}}(\Gamma_1)+T_{\mbox{\tiny{min}}}(\Gamma_2)].
			\end{dcases}
		\end{equation}
		Let $\widehat{U}$ be the solution to \eqref{matrix_schrodinger} with control $\hat{u}$. Then, by uniqueness of solutions to the Cauchy problem
		\begin{equation}
			\widehat{U}(T_{\mbox{\tiny{min}}}(\Gamma_1)) = \Gamma_1.
		\end{equation}
		For $t \in [T_{\mbox{\tiny{min}}}(\Gamma_1), T_{\mbox{\tiny{min}}}(\Gamma_1)+T_{\mbox{\tiny{min}}}(\Gamma_2)]$, by step 1 (left-invariance applied with $V_0 = \Gamma_1$), the trajectory $\widehat{U}$ steers from $\Gamma_1$ to $\Gamma_2\Gamma_1$ in time $T_{\mbox{\tiny{min}}}(\Gamma_2)$.
		
		Therefore
		\begin{equation}
			\widehat{U}(T_{\mbox{\tiny{min}}}(\Gamma_1)+T_{\mbox{\tiny{min}}}(\Gamma_2)) = \Gamma_2\Gamma_1,
		\end{equation}
		proving \eqref{gate_concatenation_bound}.
	\end{proof}
	
	By induction on lemma \ref{lemma_gate_concatenation}, we immediately obtain the following.
	
	\begin{corollary}[Multi-gate concatenation]\label{corollary_multi_gate_concatenation}
		Assume operator controllability holds for \eqref{ger equation}. Let $L \in \mathbb{N}\setminus\{0\}$ and $\Gamma_1, \dots, \Gamma_L \in SU(N)$. Then
		\begin{equation}\label{multi_gate_concatenation_bound}
			T_{\mbox{\tiny{min}}}(\Gamma_L \cdots \Gamma_2 \Gamma_1) \leq \sum_{l=1}^{L} T_{\mbox{\tiny{min}}}(\Gamma_l).
		\end{equation}
	\end{corollary}
	
	\subsection{The Quantum Fourier Transform}\label{sec:QFT_definition}
	
	The Quantum Fourier Transform (QFT) on $n$ qubits is the unitary operator $\Gamma_{\mbox{\tiny{QFT}}} \in SU(2^n)$ defined by its action on the computational basis vectors $|k\rangle$, $k \in \{0, 1, \dots, 2^n - 1\}$, as
	\begin{equation}\label{QFT_definition}
		\Gamma_{\mbox{\tiny{QFT}}} |k\rangle = \frac{1}{\sqrt{2^n}} \sum_{l=0}^{2^n - 1} e^{2\pi i k l / 2^n} |l\rangle.
	\end{equation}
	Equivalently, in matrix form, $\Gamma_{\mbox{\tiny{QFT}}}$ is the $2^n \times 2^n$ unitary matrix with entries
	\begin{equation}
		\left(\Gamma_{\mbox{\tiny{QFT}}}\right)_{l,k} = \frac{1}{\sqrt{2^n}} e^{2\pi i k l / 2^n},\hspace{0.3cm} k,l \in \{0, 1, \dots, 2^n - 1\}.
	\end{equation}
	It is well-known (see \cite[section 5.1]{nielsen2010quantum}) that the QFT can be decomposed into a product of elementary quantum gates as follows.
	
	\begin{proposition}[QFT decomposition, {\cite[section 5.1]{nielsen2010quantum}}]\label{prop_QFT_decomposition}
		The QFT on $n$ qubits can be decomposed as
		\begin{equation}\label{QFT_decomposition}
			\Gamma_{\mbox{\tiny{QFT}}} = S_n \cdot \prod_{k=1}^{n} \left( H_k \cdot \prod_{\substack{j=k+1}}^{n} R_{k,j} \right),
		\end{equation}
		where
		\begin{itemize}
			\item $H_k$ denotes the Hadamard gate acting on qubit $k$
			\begin{equation}
				H_k = \frac{1}{\sqrt{2}}(X_k + Z_k);
			\end{equation}
			\item $R_{k,j}$ denotes the controlled-$R_{j-k+1}$ gate, with qubit $j$ as control and qubit $k$ as target, defined by
			\begin{equation}
				R_{k,j} = |0\rangle\langle 0|_j \otimes I_k + |1\rangle\langle 1|_j \otimes \begin{pmatrix} 1 & 0 \\ 0 & e^{2\pi i / 2^{j-k+1}} \end{pmatrix}_k;
			\end{equation}
			\item $S_n$ is the swap network that reverses the order of the qubits.
		\end{itemize}
		The total number of elementary gates in the decomposition \eqref{QFT_decomposition} is
		\begin{equation}\label{QFT_gate_count}
			L(n) = n + \frac{n(n-1)}{2} + \left\lfloor \frac{n}{2} \right\rfloor = \frac{n(n+1)}{2} + \left\lfloor \frac{n}{2} \right\rfloor = O(n^2).
		\end{equation}
	\end{proposition}
	
	\subsection{Quantum Advantage for the QFT}\label{sec:QA_QFT_proof}
	
	We now combine the gate concatenation result (corollary \ref{corollary_multi_gate_concatenation}) with the QFT decomposition (proposition \ref{prop_QFT_decomposition}) to establish Quantum Advantage.
	
	We require the following assumption on the physical hardware.
	
	\begin{definition}\label{def_elementary_gate_time}
		We say that the controlled {S}chr\"{o}dinger equation \eqref{ger equation} has \textit{uniformly bounded elementary gate time} if there exists a constant $\tau > 0$ (independent of $n$), such that for every elementary gate $G$ acting nontrivially on at most $2$ qubits,
		\begin{equation}\label{elementary_gate_time_bound}
			T_{\mbox{\tiny{min}}}(G) \leq \tau.
		\end{equation}
	\end{definition}
	
	\begin{remark}\label{remark_elementary_gate_time}
		The assumption of uniformly bounded elementary gate time is physically natural. Consider a single-qubit gate $G$ on qubit $k$, e.g. a Hadamard $H_k$. In the framework of subsection \ref{subchapter:control_hamiltonians}, using the controls $u_k^X(t)$ and $u_k^Y(t)$ alone (setting all other controls to zero), the minimal time to implement $G$ depends on the control bound $M$ and on the coupling constants of the drift Hamiltonian $H_0$, but \emph{not} on the total number of qubits $n$. Indeed, the dynamics on qubit $k$ decouples from the rest (in the interaction picture) up to phases induced by $H_0$, which are compensable. Similarly, for a two-qubit gate between qubits $k$ and $l$, the cross-resonance interaction $Z_k \otimes X_l$ provides direct controllability of the two-qubit subsystem in bounded time.
		
		Making this argument fully rigorous requires a careful analysis in the rotating frame; we leave this as an interesting line of research (see also open problem 1 in section \ref{chapter:Open problems}).
	\end{remark}
	
	\begin{theorem}[Quantum Advantage for the QFT]\label{theorem_QA_QFT}
		Assume that the controlled {S}chr\"{o}dinger equation \eqref{ger equation} satisfies
		\begin{enumerate}
			\item[a)] operator controllability (as in Proposition \ref{prop_1});
			\item[b)] uniformly bounded elementary gate time (definition \ref{def_elementary_gate_time}).
		\end{enumerate}
		Then, the QFT exhibits Quantum Advantage (QA) in the sense of definition \ref{quantum_advantage_definition}. Specifically,
		\begin{equation}\label{QFT_minimal_time_bound}
			T_{\mbox{\tiny{min}}}(\Gamma_{\mbox{\tiny{QFT}}}) \leq \tau \left(\frac{n(n+1)}{2} + \left\lfloor \frac{n}{2} \right\rfloor\right) \leq \tau n^2,
		\end{equation}
		where $\tau$ is the constant from \eqref{elementary_gate_time_bound}.
	\end{theorem}
	\begin{proof}
		By proposition \ref{prop_QFT_decomposition}, the QFT decomposes as a product of $L(n) = O(n^2)$ elementary gates, each acting on at most $2$ qubits. By corollary \ref{corollary_multi_gate_concatenation} and the assumption \eqref{elementary_gate_time_bound}, we have
		\begin{equation}
			T_{\mbox{\tiny{min}}}(\Gamma_{\mbox{\tiny{QFT}}}) \leq \sum_{l=1}^{L(n)} T_{\mbox{\tiny{min}}}(G_l) \leq L(n) \cdot \tau \leq \tau n^2,
		\end{equation}
		where the last inequality uses $L(n) \leq n^2$ for all $n \geq 1$ (from \eqref{QFT_gate_count}). This establishes \eqref{label_t_min_estimate} with $C = \tau$ and $p = 2$.
	\end{proof}
	
	\begin{remark}[Comparison with classical complexity]
		The classical Discrete Fourier Transform (DFT) on $N = 2^n$ points requires $O(N \log N) = O(n \cdot 2^n)$ operations via the Fast Fourier Transform (FFT). The quantum implementation achieves $O(n^2)$ gate operations, yielding an exponential speedup. This is the essence of the quantum advantage for the QFT: the physical evolution time scales polynomially in the number of qubits, whereas the best known classical algorithms for the same transformation require time exponential in $n$.
	\end{remark}
	
	\begin{remark}[Application to Shor's algorithm]
		Shor's algorithm for integer factorization \cite{shor1999polynomial} is primarily based on the QFT (specifically, quantum phase estimation). Indeed, Shor's is based on two main ideas
		\begin{itemize}
			\item[a)] reduce the factoring problem to the problem of finding the period of a function;
			\item[b)] compute the period by QFT.
		\end{itemize}
		Theorem \ref{theorem_QA_QFT} therefore provides a control-theoretic explanation of why Shor's algorithm achieves exponential speedup over classical factoring algorithms: the underlying QFT can be physically implemented in polynomial time $O(n^2)$.
	\end{remark}
	
	Combining theorem \ref{theorem_QA_QFT} with the necessary condition for QA (proposition \ref{prop_7}), we obtain the following.
	
	\begin{corollary}[Necessary condition for the QFT]\label{corollary_necessary_condition_QFT}
		Under the assumptions of theorem \ref{theorem_QA_QFT}, the value function \eqref{label_value_function} satisfies
		\begin{equation}\label{QA_QFT_turnpike}
			V(T,\Gamma_{\mbox{\tiny{QFT}}}) \leq 2\left\|H_0\right\|^2 \tau n^2,\hspace{0.3cm}\forall T>0.
		\end{equation}
	\end{corollary}
	\begin{proof}
		Direct consequence of theorem \ref{theorem_QA_QFT} and proposition \ref{prop_7}, with $C = \tau$ and $p = 2$.
	\end{proof}
	
	The above corollary provides a \textit{checkable} necessary condition: if the value function $V(T,\Gamma_{\mbox{\tiny{QFT}}})$ grows faster than $O(n^2)$, then the assumption of uniformly bounded elementary gate time (definition \ref{def_elementary_gate_time}) cannot hold for the given hardware.
	
	\chapter{Quantum Discrete Optimization: the Maximum Independent Set (MIS) problem}\label{chapter:Quantum Discrete Optimization}
	
	\section{Introduction}
	\label{sec:1}
	
	The paradigm of Quantum Computing (QC) \cite{nielsen2010quantum} is currently navigating the Noisy Intermediate-Scale Quantum (NISQ) era, an epoch characterized by hardware that possesses sufficient qubit counts to challenge classical supercomputers but lacks the fault-tolerant error correction required for deep, universal gate-based algorithms. Within this landscape, combinatorial optimization has emerged as a primary candidate for demonstrating practical quantum advantage. Among the most prominent algorithmic strategies is the Quantum Approximate Optimization Algorithm (QAOA) \cite{farhi2014quantum,farhi2016quantum}, initially proposed as a digital, variational protocol designed to explore the Hilbert space of near-term processors to isolate the ground states of objective Hamiltonians.
	
	In this chapter, rather than addressing general discrete optimization, we restrict ourselves to a single, paradigmatic problem: the \textbf{Maximum Independent Set (MIS)} problem on an undirected simple graph. The choice of MIS is not accidental:
	\begin{itemize}
		\item MIS is NP-hard \cite{karp1972reducibility} and admits a clean, minimal binary-variable formulation, requiring only one binary variable per vertex of the underlying graph;
		\item MIS is \emph{naturally encoded} on neutral-atom quantum processors via the Rydberg blockade mechanism (chapter \ref{chapter:Neutral atoms quantum computers} and \cite{Henriet2020quantumcomputing,pichler2018quantum,ebadi2022quantum}), without any need to introduce artificial penalty terms;
		\item MIS, being a constrained combinatorial problem with a graph-theoretic structure, is particularly well suited to a control-theoretic analysis of the QAOA dynamics.
	\end{itemize}
	All along this chapter, we work \emph{exclusively} with binary variables $x_i \in \{0, 1\}$ (one per vertex), without recurring to the spin-variable change $z_i = 1 - 2 x_i \in \{-1, +1\}$. Accordingly, the cost Hamiltonian is built directly from the qubit projector $\hat n_i = |1\rangle\langle 1|_i = (I - Z_i)/2$, which is the natural quantum lift of the binary variable $x_i$.
	
	Historically, QAOA was conceived as a heuristic sequence of discrete unitary operations, alternating between a problem-specific Hamiltonian and a beginning Hamiltonian. However, as the physical limitations of discrete gate synthesis, such as accumulated Trotterization errors and gate fidelities, become increasingly apparent, the theoretical understanding of QAOA is undergoing a fundamental transformation. By casting the optimization problem into the continuous-time domain, QAOA can be rigorously reinterpreted through the mathematical lens of Quantum Optimal Control theory. This perspective bridges the gap between pure Adiabatic Quantum Optimization (AQO) and digital variational algorithms, revealing that QAOA is fundamentally a discretized approximation of a continuous ``bang-anneal-bang'' optimal control protocol \cite{PhysRevLett.126.070505}.
	
	As long as we know, for quantum discrete optimization, Quantum Advantage was proved for specific problems \cite{gilyen2021sub}. For MIS, Quantum Advantage was proved in \cite{choi2025beyond} adding to the Hamiltonian an additional term\footnote{a non-stochastic term $XX$ breaking the Lie algebra symmetry}. Companies like \href{https://www.dwavequantum.com/}{D-WAVE} are working to add this term to their Quantum Computers Hamiltonian \cite{PhysRevApplied.13.034037}. This hardware development is still at an experimental stage, whence we focus on existing quantum hardware.
	
	This chapter is structured as follows. Section \ref{sec:Problem formulation} introduces the MIS problem in the language of graph theory and recasts it as a Quadratic Unconstrained Binary Optimization (QUBO) problem. Subsection \ref{sec:Natural encoding} promotes the binary cost function to a self-adjoint operator on $(\mathbb{C}^2)^{\otimes n}$, the so-called \emph{cost Hamiltonian}, introduces the \emph{mixer Hamiltonian}, and establishes the natural encoding of MIS on neutral-atom \href{https://www.pasqal.com/wp-content/uploads/2025/03/Technical-Overview-for-Advanced-Users-Orion-Beta.pdf}{Pasqal-type} hardware. In subsection \ref{subsec:main_result}, we show how that MIS can be solved successfully on \href{https://www.pasqal.com/wp-content/uploads/2025/03/Technical-Overview-for-Advanced-Users-Orion-Beta.pdf}{Pasqal} machines, in a sufficiently large time horizon $T$. In section \ref{sec:Is there quantum advantage?}, we present estimates needed to prove Quantum Advantage. Section \ref{sec:How QAOA is solved} outlines the classical-quantum hybrid loop. Finally, section \ref{sec:Advanced State Stabilization via Integral Tracking Functionals} introduces an integral tracking functional that stabilizes the optimal control problem along the entire trajectory, suppressing diabatic leakage.
	
	A closely related approach is Quantum Annealing. Rather than targeting universal computation, Quantum Annealing hardware is uniquely designed to solve discrete optimization problems. See, for instance, the recent paper \cite{quinton2025quantum} benchmarking Quantum Annealing and classical optimizers. Several companies are developing Quantum Annealing hardware, including \href{https://www.dwavequantum.com/}{D-WAVE}. There exists already a robust interplay between Quantum Annealing and Control Theory (see, for instance, \href{https://docs.dwavequantum.com/en/latest/quantum_research/annealing.html}{Annealing Implementation and Controls} and \cite{PhysRevLett.126.070505,d2024controllability}).
	
	Both QAOA and Quantum Annealing are designed to run on current Noisy Intermediate-Scale Quantum (NISQ) devices. Other quantum discrete optimization algorithms, like Grover Adaptive Search \cite{fujiwara2025grover} and Quantum Walk, are designed for future Fault-Tolerant Quantum Computers. In these notes, we focus on QAOA and Quantum Annealing.
	
	\section{The Maximum Independent Set problem}\label{sec:Problem formulation}
	
	Let $G = (V, E)$ be an undirected simple graph, with vertex set $V = \{1, 2, \dots, n\}$ and edge set
	\begin{equation}
		E \subseteq \big\{ \{i,j\} \ \big| \ i, j \in V, \ i \neq j \big\}.
	\end{equation}
	
	\begin{definition}[Independent set]\label{def_independent_set}
		A subset $S \subseteq V$ is an \textbf{independent set} (or \emph{stable set}) of $G$ if and only if 
		\begin{equation}
			\forall \, i, j \in S \ \mbox{ with } \ i \neq j, \quad \{i, j\} \notin E.
		\end{equation}
	\end{definition}
	
	\begin{definition}[Maximum Independent Set (MIS)]\label{def_MIS}
		The \textbf{Maximum Independent Set} problem on $G$ consists in finding an independent set $S^* \subseteq V$ of maximum cardinality:
		\begin{equation}\label{mis_def}
			S^* \in \arg\max\big\{ |S| \ : \ S \subseteq V \mbox{ is an independent set of } G \big\}.
		\end{equation}
		The cardinality
		\begin{equation}\label{alpha}
			\alpha(G) \coloneqq |S^*|
		\end{equation}
		is called the \textbf{independence number} of $G$.
	\end{definition}
	
	\subsection{Binary-variable formulation}
	
	Following the bijection between subsets of $V$ and binary indicator vectors, we encode each $S \subseteq V$ as
	\begin{equation}\label{indicator_vector}
		\mathbf{x} = (x_1, \dots, x_n) \in \{0, 1\}^n, \quad x_i \coloneqq 
		\begin{cases}
			1 & \mbox{if } i \in S, \\
			0 & \mbox{if } i \notin S.
		\end{cases}
	\end{equation}
	Each component $x_i$ is a \emph{binary} variable taking values in $\{0,1\}$. The independence constraint of definition \ref{def_independent_set} reads, in these variables,
	\begin{equation}\label{indep_constraint}
		x_i \, x_j = 0, \quad \forall \, \{i, j\} \in E,
	\end{equation}
	since the two endpoints of an edge cannot both be selected. Define
	\begin{equation}\label{cost_function}
		g(\mathbf{x})\coloneqq \sum_{i=1}^n x_i.
	\end{equation}
	Problem \eqref{mis_def} is therefore equivalent to the constrained quadratic binary program
	\begin{equation}\label{mis_constrained}
		\max_{\mathbf{x} \in \{0,1\}^n}g(\mathbf{x}) = \sum_{i=1}^n x_i, \quad \mbox{subject to} \quad x_i x_j = 0, \ \forall \{i,j\} \in E.
	\end{equation}
	
	By introducing a penalty parameter $\lambda > 1$, problem \eqref{mis_constrained} can be recast as the unconstrained Quadratic Unconstrained Binary Optimization (QUBO) problem \cite{kochenberger2014unconstrained,10.3389/fphy.2014.00005}:
	\begin{equation}\label{mis_qubo}
		\min_{\mathbf{x} \in \{0,1\}^n} f(\mathbf{x}), \qquad f(\mathbf{x}) \coloneqq -\sum_{i=1}^n x_i + \lambda \sum_{\{i,j\} \in E} x_i \, x_j.
	\end{equation}
	The cost $f$ in \eqref{mis_qubo} is purely quadratic in the binary variables $x_i$, of the form $\mathbf{x}^\top Q \mathbf{x}$ for a symmetric matrix $Q \in \mathbb{R}^{n\times n}$ (recall $x_i^2 = x_i$ for $x_i \in \{0,1\}$, so the linear term may be absorbed into the diagonal of $Q$).
	
	\begin{proposition}[Penalty equivalence]\label{prop_penalty}
		Let $\lambda > 1$. Then every minimizer $\mathbf{x}^* \in \{0,1\}^n$ of \eqref{mis_qubo} satisfies the independence constraint \eqref{indep_constraint} and the corresponding subset $S^* = \{i \in V \, : \, x^*_i = 1\}$ is a Maximum Independent Set of $G$.
	\end{proposition}
	
	\begin{proof}
		Let $\mathbf{x}^*$ be a global minimizer of \eqref{mis_qubo} and assume, by contradiction, that there exists $\{i,j\} \in E$ with $x^*_i = x^*_j = 1$. Define $\tilde{\mathbf{x}}$ by setting $\tilde x_j = 0$ and $\tilde x_k = x^*_k$ for $k \neq j$. Then
		\begin{equation}
			f(\tilde{\mathbf{x}}) - f(\mathbf{x}^*) \leq 1 - \lambda \cdot 1 = 1 - \lambda < 0,
		\end{equation}
		where the inequality uses that flipping $x_j$ from $1$ to $0$ removes \emph{at least} the edge contribution from $\{i,j\}$ (it may remove more, but the sign is favorable) and adds $+1$ to $-\sum x_i$. This contradicts the global optimality of $\mathbf{x}^*$. Hence $\mathbf{x}^*$ is feasible for \eqref{mis_constrained}. Conversely, any feasible $\mathbf{x}$ has cost value $f(\mathbf{x})=-\sum_{i=1}^{n} x_i$ in \eqref{mis_qubo} and \eqref{mis_constrained}, so $\mathbf{x}^*$ is necessarily of maximum cardinality.
	\end{proof}
	
	The cardinality of the search space $\{0,1\}^n$ is $2^n < +\infty$, ensuring the existence of a global minimizer (uniqueness, in general, might fail). However, locating $\mathbf{x}_{\mbox{\tiny{opt}}}$ is NP-hard \cite{karp1972reducibility,pardalos1992complexity}, since MIS is one of Karp's original 21 NP-complete problems.
	
	\section{Neutral-atom analog quantum computers}
	\label{chapter:Neutral atoms quantum computers}
	
	\subsection{Introduction}
	\label{sec:Introduction}
	Neutral-atom quantum computers manipulate spatial configurations of atoms, such as ${}^{87}\text{Rb}$ trapped in optical tweezer arrays, to perform quantum operations \cite{Henriet2020quantumcomputing}. Information is encoded by assigning the logical state $|0\rangle$ to the atom's ground state and $|1\rangle$ to a highly excited Rydberg state. By exposing these atoms to highly controlled laser pulses, the system undergoes a time evolution governed by the controlled Schr\"{o}dinger equation. As illustrated in \cite{Henriet2020quantumcomputing}, neutral atoms quantum computers can operate either in digital mode or in analog mode. Digital mode is designed to program with discrete-time quantum gates and enjoys universal quantum computation capabilities, whereas analog mode is time-continuous and it is particularly suitable for discrete optimization. In these notes, we focus in the analog mode.
	
	For a general introduction on Quantum Computing (QC), see \cite{nielsen2010quantum}.
	
	\subsection{The Controlled Schr\"{o}dinger Equation}
	As in \eqref{ger equation}, in the field of quantum control, the evolution of a quantum state $|\psi(t)\rangle$ is often described by a bilinear controlled Schr\"{o}dinger equation
	\begin{equation}
		i\hbar \frac{d}{dt}\psi(t) = \left( H_0 + \sum_{k=1}^m u_k(t) H_k \right) \psi(t),
		\label{eq:boscain}
	\end{equation}
	where $H_0$ is the unforced (drift) Hamiltonian representing the internal dynamics of the system, $\{H_k\}_{k=1}^{m}$ are the interaction Hamiltonians, and $\{u_k(t)\}_{k=1}^{m}$ are the real-valued, time-dependent control fields driving the system.
	
	\subsection{Application to the Pasqal Architecture}
	\label{subsec:Application to the Pasqal Architecture}
	
	For a neutral-atom quantum processor like \href{https://www.pasqal.com/wp-content/uploads/2025/03/Technical-Overview-for-Advanced-Users-Orion-Beta.pdf}{Pasqal's Orion}, the physical controls correspond to a two-photon laser transition scheme. The system is globally driven by fields parameterized by their Rabi frequency amplitude $\Omega(t)$ and their detuning from the Rydberg transition resonance $\delta(t)$, as in \href{https://docs.pasqal.com/pulser/programming/}{Pasqal documentation}. 
	
	Inspired by \cite{tibaldi2025analog,Henriet2020quantumcomputing}, the total Hamiltonian $H(t)$ governing the atomic register of $n$ qubits is given by:
	\begin{equation}\label{neutral_atoms_hamiltonian}
		H(t) = \frac{\hbar \, \Omega(t)}{2} \sum_{k=1}^{n} X_k - \hbar \, \delta(t) \sum_{i=1}^{n} \hat n_i + \sum_{\substack{i,j=1 \\ i<j}}^{n} \frac{C_6}{R_{i,j}^6} \, \hat n_i \, \hat n_j,
	\end{equation}
	where
	\begin{itemize}
		\item as in chapter \ref{chapter:Mathematical framework}, the state space is the following Hilbert space over $\mathbb{C}$\footnote{of dimension, over $\mathbb{C}$, $2^n$}
		\begin{equation}
			\mathcal{H} \defeq \underbrace{\mathbb{C}^2 \otimes \dots \otimes \mathbb{C}^2}_{n\text{ times}} = \bigotimes_{j=1}^n \mathbb{C}^2;
		\end{equation}
		\item $X_k=\begin{pmatrix} 0 & 1 \\ 1 & 0 \end{pmatrix}_k$ is the Pauli-$X$ operator acting on qubit $k$;
		\item $\hat{n}_k = (I - Z_k)/2=\begin{pmatrix} 0 & 0 \\ 0 & 1 \end{pmatrix}_k=|1\rangle\langle 1|_k$ is the Rydberg occupation number at site $k$;
		\item $C_6\in \mathbb{R}$ defines the van der Waals interaction strength;
		\item $R_{i,j}\in (0,+\infty)$ is the spatial distance between atoms $i$ and $j$.
	\end{itemize}
	
	To map the Pasqal Hamiltonian \eqref{neutral_atoms_hamiltonian} exactly to the control framework of \eqref{eq:boscain}, we set $m=2$ and define the terms as follows:
	\begin{itemize}
		\item \textbf{Drift Hamiltonian ($H_0$):} The constant van der Waals interactions between pairs of atoms,
		\begin{equation}\label{drift_hamiltonian_pasqal}
			H_0 = \sum_{\substack{i,j=1 \\ i<j}}^{n} \frac{C_6}{R_{i,j}^6} \, \hat{n}_i \, \hat{n}_j.
		\end{equation}
		\item \textbf{Control fields ($u_1(t), u_2(t)$):} The time-dependent parameters dictated by the global laser fields,
		\begin{equation}
			u_1(t) = \hbar \Omega(t), \quad u_2(t) = \hbar \delta(t).
		\end{equation}
		\item \textbf{Control Hamiltonians ($H_1, H_2$):} The operators coupling to the control fields,
		\begin{equation}\label{control_hamiltonians_pasqal}
			H_1 = \frac{1}{2} \sum_{k=1}^{n} X_k, \quad H_2 = - \sum_{k=1}^{n} \hat{n}_k.
		\end{equation}
	\end{itemize}
	
	Substituting these definitions back into the general controlled Schr\"{o}dinger equation yields the exact time-evolution equation for the neutral-atom register:
	\begin{equation}\label{pasqal_equation_explicit}
		i\hbar \frac{d}{dt}\psi(t) = \left( \sum_{\substack{i,j=1 \\ i<j}}^{n} \frac{C_6}{R_{i,j}^6} \hat{n}_i \hat{n}_j + \hbar \Omega(t) \left[ \frac{1}{2}\sum_{k=1}^{n}X_k \right] - \hbar \delta(t) \left[ \sum_{k=1}^{n}\hat{n}_k \right] \right) \psi(t).
	\end{equation}
	
	As we shall see, with respect to superconducting quantum computers studied in section \ref{chapter:Superconducting quantum computers}, neutral atoms quantum computers (in analog mode)
	\begin{itemize}
		\item have weaker controllability property (they are not universal quantum computers);
		\item enjoys a better connectivity.
	\end{itemize}
	
	\subsection{Controllability properties of the neutral-atom register}\label{subsec:controllability_pasqal}
	
	In contrast with the superconducting setting of chapter \ref{chapter:QA_QFT}, the Pasqal bilinear Schr\"{o}dinger equation \eqref{pasqal_equation_explicit} is driven by only two \emph{global} scalar controls $\Omega(t)$ and $\delta(t)$, which act uniformly on all the qubits (through $H_1 = \tfrac{1}{2}\sum_{k=1}^{n} X_k$ and $H_2 = -\sum_{k=1}^{n} \hat n_k$). A direct consequence is that the reachable set of \eqref{pasqal_equation_explicit} is much smaller than $SU(2^n)$: indeed, every symmetry of the drift $H_0$ in \eqref{drift_hamiltonian_pasqal} that also commutes with $H_1$ and $H_2$ is automatically conserved by the dynamics. Similar controllability results were proved in \cite{d2024controllability}.
	
	\begin{proposition}[Lack of full operator controllability]\label{prop_pasqal_no_full_controllability}
		Let $\sigma \in \mathfrak{S}_n$ be a permutation of the atom indices which is a symmetry of the weighted graph $\big(\{1,\dots,n\}, \{(i,j,C_6/R_{i,j}^6)\}\big)$, namely $R_{\sigma(i),\sigma(j)}=R_{i,j}$ for every $i,j$. Then the induced permutation operator $P_\sigma$ on $\mathcal{H}$ commutes with $H_0$, $H_1$ and $H_2$, hence with the full Hamiltonian in \eqref{pasqal_equation_explicit} at every time $t\in[0,T]$. In particular, the bilinear Schr\"{o}dinger equation \eqref{pasqal_equation_explicit} is \emph{not} operator controllable on $SU(2^n)$ in the sense of Definition \ref{def_1}.
	\end{proposition}
	We highlight that, as we mentioned, the above result holds for Pasqal devices in analog mode. Pasqal devices in digital model are Universal Quantum Computers (operator controllability holds). In these notes, we focus on analog mode.
	\begin{proof}[Proof of Proposition \ref{prop_pasqal_no_full_controllability}]
		Since $\sigma$ preserves the pairwise distances $R_{i,j}$, the drift Hamiltonian \eqref{drift_hamiltonian_pasqal} satisfies $P_\sigma H_0 P_\sigma^{-1}=H_0$. The control Hamiltonians \eqref{control_hamiltonians_pasqal} are symmetric sums over all atoms, so $P_\sigma H_1 P_\sigma^{-1}=H_1$ and $P_\sigma H_2 P_\sigma^{-1}=H_2$. By uniqueness of the solution to the matrix Cauchy problem \eqref{matrix_schrodinger_inftime}, every admissible trajectory $U(t)$ satisfies $P_\sigma U(t) P_\sigma^{-1}=U(t)$, that is $[U(t),P_\sigma]=0$. Hence the reachable set is contained in the centralizer of $P_\sigma$ in $SU(2^n)$, which is a proper subgroup whenever $\sigma\neq\mathrm{id}$. This rules out operator controllability on the whole of $SU(2^n)$.
	\end{proof}
	
	Proposition \ref{prop_pasqal_no_full_controllability} is not a pathology: it reflects the fact that the neutral-atom platform considered here is an \emph{analog} quantum simulator tailored to a specific class of Hamiltonian problems (Ising/Rydberg-type), rather than a \emph{universal digital} quantum computer. The appropriate notion of controllability in this setting is not full operator controllability on $SU(2^n)$, but rather the weaker property of being able to steer, in polynomial time in $n$, the initial state $\phi_B=|0\rangle^{\otimes n}$ to a ground state of the MIS cost Hamiltonian (see Proposition \ref{prop_mis_hamiltonian} and section \ref{sec:QAOA optimal control}). Since $\phi_B=|0\rangle^{\otimes n}$ is precisely the ground state of $\ham(0)=H_{0}+\hbar\delta_{0}\sum_{k=1}^{n}\hat n_{k}$ (Lemma~\ref{lemma_endpoints}(a)), this steering is realised by the \emph{adiabatic} programme of \cite{farhi2000quantum}: slowly deform $\ham(0)$ into $\ham(1)$, whose ground manifold encodes the MIS (Lemma~\ref{lemma_endpoints}(b)). This is the control-theoretic content of the quantum-annealing programme of \cite{PhysRevLett.126.070505} and will be formulated precisely in section \ref{sec:QAOA optimal control}.
	
	\begin{remark}[Role of the atom layout]\label{remark_layout_control}
		Although the pulses $\Omega(t)$ and $\delta(t)$ are global, the user retains an additional \emph{discrete} control handle: the spatial layout $\{r_i\}_{i=1}^n \subset \mathbb{R}^d$ of the atoms in the optical tweezer array, which enters the drift Hamiltonian \eqref{drift_hamiltonian_pasqal} through the pairwise distances $R_{i,j}$. In the Rydberg-blockade regime, choosing the layout is equivalent to choosing the edge set $E$ of a unit-disk graph, whence to choosing the problem instance $G=(V,E)$ of MIS itself (see also \cite[section II]{Henriet2020quantumcomputing}). Thus, in sharp contrast with the superconducting digital paradigm (where the hardware graph is fixed and the program is encoded in the time-dependent controls), the neutral-atom analog paradigm encodes the program in the \emph{geometry} of the register and uses the time-dependent controls only to steer along the adiabatic path.
	\end{remark}
	
	\subsection{Natural encoding on neutral-atom hardware}\label{sec:Natural encoding}
	
	Typically, Quantum Computers (QC) are efficient at solving problems whose input and output are discrete, while the intermediate computation is continuous. In this spirit, MIS is solved on a quantum computer by continuous-time methods: Quantum Annealing \cite{farhi2000quantum,aharonov2008adiabatic} and the Quantum Approximate Optimization Algorithm (QAOA) \cite{farhi2014quantum,farhi2016quantum}. In \cite{PhysRevLett.126.070505}, both are interpreted within the unifying framework of quantum optimal control.
	
	We now show how MIS can be encoded in the physics of the neutral-atom Pasqal Quantum Computers, through an appropriate spatial configuration of the atoms. Take the Hamiltonian \eqref{neutral_atoms_hamiltonian}, with controls $(\Omega,\delta)\equiv (0,1)$, getting
	\begin{equation}\label{cost_hamiltinian+remainder}
		H_C\coloneqq - \hbar  \sum_{i=1}^{n} \hat n_i + \sum_{\substack{i,j=1 \\ i<j}}^{n} \frac{C_6}{R_{i,j}^6} \, \hat n_i \, \hat n_j,
	\end{equation}
	Roughly speaking, appropriately positioning the atoms realizing the qubits, $H_C$ represents \eqref{mis_qubo}, up to a \emph{tail energy}.
	
	\begin{proposition}[Rydberg-blockade encoding of MIS]\label{prop_blockade_encoding}
		Let $G=(V,E)$ be a graph on $n$ vertices and let $R_b>0$. Suppose the atoms are placed at positions $\{r_i\}_{i=1}^n \subset \mathbb{R}^d$ ($d \in \{2,3\}$) satisfying
		\begin{equation}\label{rydberg_radius}
			\begin{cases}
				\|r_i - r_j\| \leq R_b & \mbox{if } \{i,j\} \in E, \\
				\|r_i - r_j\| > R_b   & \mbox{if } \{i,j\} \notin E.
			\end{cases}
		\end{equation}
		Then the drift Hamiltonian \eqref{drift_hamiltonian_pasqal} satisfies, for every computational-basis vector $|\mathbf{x}\rangle$,
		\begin{equation}\label{drift_on_comp_basis}
			H_0 |\mathbf{x}\rangle = \bigg(\sum_{\substack{\{i,j\}\in E\\ i<j}} \frac{C_6}{R_{i,j}^6} x_i x_j + \sum_{\substack{\{i,j\}\notin E \\ i<j}} \frac{C_6}{R_{i,j}^6} x_i x_j\bigg) |\mathbf{x}\rangle.
		\end{equation}
		In particular, for any bitstring $\mathbf{x}$ encoding an independent set of $G$ (i.e., $x_ix_j=0$ for all $\{i,j\}\in E$),
		\begin{equation}\label{drift_on_indep_set}
			H_0 |\mathbf{x}\rangle = \sum_{\substack{\{i,j\}\notin E \\ i<j}} \frac{C_6}{R_{i,j}^6} x_i x_j\, |\mathbf{x}\rangle,
		\end{equation}
		while for any bitstring violating the independence constraint on at least one edge $\{i,j\}\in E$,
		\begin{equation}\label{drift_penalty}
			\langle \mathbf{x}|H_0|\mathbf{x}\rangle \geq \frac{C_6}{R_b^6}.
		\end{equation}
		Finally, for any bitstring $\mathbf{x}$, we have
		\begin{equation}\label{function_promotion}
			\langle \mathbf{x}|H_C|\mathbf{x}\rangle = -\hbar g(\mathbf{x}) + \sum_{\substack{\{i,j\}\in E\\ i<j}} \frac{C_6}{R_{i,j}^6} x_i x_j + \tau(\bx),
		\end{equation}
		where
		\begin{itemize}
			\item $g(\mathbf{x})$ \eqref{cost_function} is the cost function to be maximized in MIS;
			\item the addendum
			\begin{equation}
				\sum_{\substack{\{i,j\}\in E\\ i<j}} \frac{C_6}{R_{i,j}^6} x_i x_j
			\end{equation}
			represents the independence constraint;
			\item the term \begin{equation}
				\tau(\bx)\defeq\!\!\sum_{\substack{\{i,j\}\notin E\\ i<j}}\!\! \frac{C_{6}}{R_{i,j}^{6}}\,x_{i}x_{j}\ \ge 0,
			\end{equation}
			is a \emph{tail energy} we will study in subsection \ref{subsec:spectral_gap_Pasqal}.
		\end{itemize}
	\end{proposition}
	
	\begin{proof}
		The operator $\hat n_i$ has spectrum $\{0, 1\}$ and acts as the multiplication-by-$x_i$ operator on the computational basis: for every bitstring $\mathbf{x} = (x_1,\dots,x_n) \in \{0,1\}^n$,
		\begin{equation}\label{eq_projector_action}
			\hat n_i \, |x_1 \dots x_n\rangle = x_i \, |x_1 \dots x_n\rangle.
		\end{equation}
		Identity \eqref{eq_projector_action} is the precise sense in which $\hat n_i$ is the natural quantum lift of the classical binary variable $x_i$, with no need of introducing spin variables.
		By \eqref{eq_projector_action}, the drift Hamiltonian \eqref{drift_hamiltonian_pasqal} acts on $|\mathbf{x}\rangle$ as
		\begin{equation}\label{fsdgs}
			H_0|\mathbf{x}\rangle = \sum_{\substack{i,j=1 \\ i<j}}^{n} \frac{C_6}{R_{i,j}^6}x_ix_j\,|\mathbf{x}\rangle,
		\end{equation}
		which splits as the two sums in \eqref{drift_on_comp_basis}. For $\mathbf{x}$ encoding an independent set, $x_ix_j=0$ whenever $\{i,j\}\in E$, giving \eqref{drift_on_indep_set}. If $\mathbf{x}$ violates the independence constraint on at least one edge $\{i,j\}\in E$, then $x_i=x_j=1$ and $\|r_i-r_j\|\leq R_b$ by \eqref{rydberg_radius}, so that
		\begin{equation}
			\langle \mathbf{x}|H_0|\mathbf{x}\rangle \geq \frac{C_6}{R_{i,j}^6}\geq \frac{C_6}{R_b^6}.\qedhere
		\end{equation}
		In conclusion, \eqref{function_promotion} follows from \eqref{eq_projector_action}, together with \eqref{fsdgs}.
	\end{proof}
	
	The key physical consequence is that when the van der Waals coupling dominates the Rabi frequency, $C_6/R_b^6 \gg \hbar\Omega_{\max}$ (the \emph{Rydberg-blockade regime}), the doubly-excited state $|11\rangle_{ij}$ on any edge $\{i,j\}\in E$ is energetically forbidden during the evolution. The drift Hamiltonian thereby enforces the independence constraint \eqref{indep_constraint} \emph{by physics}, not by software, with an effective penalty parameter $\lambda_{\mbox{\tiny{eff}}} \sim C_6 / (R_b^6 \, \hbar \Omega_{\max})$.
	
	\begin{corollary}[Effective Hamiltonian in the blockade subspace]\label{cor_effective_hamiltonian}
		Let $\mathcal{H}_{\mathrm{IS}} \subseteq \mathcal{H}$ be the subspace spanned by computational-basis states encoding independent sets of $G$, and let $P_{\mathrm{IS}}:\mathcal{H}\to \mathcal{H}_{\mathrm{IS}}$ be the orthogonal projection. In the blockade regime $C_6/R_b^6 \gg \hbar\Omega_{\max}$, the evolution under \eqref{pasqal_equation_explicit}, when initialized in $\mathcal{H}_{\mathrm{IS}}$, is well approximated by the projected dynamics governed by
		\begin{equation}\label{effective_hamiltonian_IS}
			H_{\mathrm{eff}}(t) = P_{\mathrm{IS}}\bigg(\frac{\hbar\Omega(t)}{2}\sum_{k=1}^n X_k - \hbar\delta(t)\sum_{k=1}^n \hat n_k\bigg)P_{\mathrm{IS}} + P_{\mathrm{IS}} H_0 P_{\mathrm{IS}}.
		\end{equation}
		Since the detuning term $-\hbar\delta(t)\sum_{k=1}^{n} \hat n_k$ acts on independent sets as $-\hbar\delta(t)|S|$ (where $|S|=\sum_{k=1}^{n} x_k$), maximizing $|S|$ via $\delta(t)>0$ is equivalent to solving the MIS problem on $G$.
	\end{corollary}
	
	\begin{remark}[Geometric restriction]\label{remark_geometric_restriction}
		Not every graph $G$ admits an embedding in $\mathbb{R}^d$ ($d=2,3$) compatible with \eqref{rydberg_radius}: only the so-called \emph{unit-disk graphs} (UDGs) do. In graph-theoretic terms, condition \eqref{rydberg_radius} realizes $G$ as a UDG with threshold radius $R_b$. For more general graph topologies, one resorts either to multi-level encodings \cite{nguyen2023quantum} or to ancillary penalty schemes implemented in software.
	\end{remark}
	
	\begin{remark}[The atom layout as a design variable]\label{remark_layout_design}
		The spatial layout $\{r_i\}_{i=1}^n$ is not a dynamical control but a \emph{discrete design variable}, chosen before the experiment. By Proposition \ref{prop_blockade_encoding}, selecting the layout is equivalent to selecting the graph $G=(V,E)$, i.e., the problem instance. This is a fundamental difference with the superconducting digital paradigm of chapter \ref{chapter:QA_QFT}, where the hardware graph (the coupling map $\mathcal{E}$) is fixed and the problem instance is encoded entirely in the time-dependent controls.
	\end{remark}
	
	\section{Approximate state controllability and the QAOA optimal control formulation}\label{sec:QAOA optimal control}
	
	In section \ref{sec:Natural encoding}, we established that the Rydberg blockade naturally enforces the independence constraint (Proposition~\ref{prop_blockade_encoding}). We take as initial state the all-zero computational state
	\begin{equation}\label{initial_state_phiB}
		\phi_B = |0\rangle^{\otimes n} = |\mathbf{0}\rangle .
	\end{equation}
	In the standard Quantum Approximate Optimization Algorithm (QAOA) \cite{farhi2014quantum} one initialises in the uniform superposition $|+\rangle^{\otimes n}$, the ground state of the transverse-field mixer. Here, however, the controlled Hamiltonian is the Pasqal Hamiltonian \eqref{neutral_atoms_hamiltonian}, and a different choice is more natural: by Lemma~\ref{lemma_endpoints}(a), $\phi_B=|0\rangle^{\otimes n}$ is the \emph{unique} ground state of
	\begin{equation}\label{endpoint_H0_intro}
		\ham(0)=H_{0}+\hbar\delta_{0}\sum_{k=1}^{n}\hat n_{k},
	\end{equation}
	the Pasqal Hamiltonian at the control endpoint $(\Omega,\delta)=(0,-\delta_0)$, with spectral gap $\hbar\delta_0>0$ above it. Starting in $\phi_B=|0\rangle^{\otimes n}$ therefore means starting in a known, gapped ground state, and slowly moving to the ground state of $\ham(1)$, which encodes solutions of MIS by Lemma~\ref{lemma_endpoints}(b), realises the \emph{adiabatic} approach of \cite{farhi2000quantum}, to which we append a final closed-orbit argument promoting approximate to approximate controllability.
	
	The central question of this section is:
	
	\begin{center}
		\emph{Can the Pasqal equation \eqref{pasqal_equation_explicit}, driven by the controls $\Omega(t)$ and $\delta(t)$, steer \textbf{approximately} from $\phi_B=|0\rangle^{\otimes n}$ to a state encoding a Maximum Independent Set of $G$?}
	\end{center}
	
	The analysis is carried out entirely on the \emph{physical} Pasqal Hamiltonian \eqref{neutral_atoms_hamiltonian}, without passing through any abstract interpolating Hamiltonian. As we shall show, the answer involves three ingredients: a symmetry obstruction identifying the correct target (the MIS subspace $\Hmis$), an adiabatic anneal on $(\Omega,\delta)$ in the spirit of \cite{farhi2000quantum}, and a Lie-group compactness argument promoting approximate reachability to approximate reachability.
	
	\subsection{The QAOA optimal control problem}\label{subsec:QAOA_OCP}
	
	Following \cite{PhysRevLett.126.070505}, we formulate Quantum Annealing / Quantum Approximate Optimization Algorithm (QAOA) as a continuous-time optimal control problem. Given a time horizon $T>0$, the \textbf{QAOA optimal control problem} reads
	\begin{equation}\label{optimal_control_problem_qaoa}
		\min_{\Omega,\,\delta\in \mathscr{U}_{\mathrm{ad}}^T} \; J_T(\Omega,\delta) \;\coloneqq\; \langle \psi(T) | H_C | \psi(T) \rangle,
	\end{equation}
	subject to the \textbf{state equation}, driven by the Pasqal Hamiltonian \eqref{neutral_atoms_hamiltonian},
	\begin{equation}\label{QAOA_state_equation}
		\begin{cases}
			\displaystyle i\hbar \frac{d}{dt}\psi(t) = \left(\frac{\hbar\,\Omega(t)}{2}\sum_{k=1}^n X_k - \hbar\,\delta(t)\sum_{k=1}^n\hat n_k + \sum_{\substack{i,j=1\\i<j}}^{n}\frac{C_6}{R_{i,j}^6}\,\hat n_i\hat n_j\right)\psi(t), & t\in(0,T),\\[6pt]
			\psi(0) = \phi_B = |0\rangle^{\otimes n},
		\end{cases}
	\end{equation}
	with admissible controls
	\begin{equation}\label{admissible_controls_pasqal}
		\mathscr{U}_{\mathrm{ad}}^T \;\coloneqq\; \big\{(\Omega,\delta)\in L^\infty(0,T;\mathbb{R}^2) \;\big|\; \Omega(t)\in [0,\Omega_{\max}],\; \delta(t)\in [-\delta_{\max},\delta_{\max}],\;\mbox{a.e. }t\in(0,T)\big\}.
	\end{equation}
	
	\subsection{Symmetry obstruction to reaching a single $|\mathbf{x}^{*}\rangle$}\label{subsec:symmetry_obstruction}
	
	\begin{proposition}[Symmetry constraint on the reachable set]\label{prop_symmetry_constraint}
		Define the reachable set from $\phi_B$, with unrestricted time horizon, as
		\begin{equation}\label{reachable_set_def}
			\mathcal{R}(\phi_B) \;\coloneqq\; \bigcup_{T>0}\big\{\psi(T)\;:\;(\Omega,\delta)\in\mathscr{U}_{\mathrm{ad}}^T\big\}.
		\end{equation}
		Then every $\psi\in\mathcal{R}(\phi_B)$ is invariant under $\mathrm{Aut}(G,w)\coloneqq\{\sigma\in\mathfrak{S}_n : R_{\sigma(i),\sigma(j)}=R_{i,j}\;\forall\,i,j\}\footnote{$w=\left\{R_{i,j}\right\}_{\substack{i,j=1\\i<j}}^{n}$}$:
		\begin{equation}\label{reachable_invariance}
			P_\sigma\,\psi = \psi,\qquad \forall\,\sigma\in\mathrm{Aut}(G,w).
		\end{equation}
	\end{proposition}
	\begin{proof}
		(i)~$\phi_B=|0\rangle^{\otimes n}$ is fully symmetric: $P_\sigma\phi_B=\phi_B$ for all $\sigma\in\mathfrak{S}_n$ (every site permutation fixes the all-zero string). (ii)~By Proposition~\ref{prop_pasqal_no_full_controllability}, each $P_\sigma$ with $\sigma\in\mathrm{Aut}(G,w)$ commutes with the Pasqal Hamiltonian \eqref{neutral_atoms_hamiltonian} at every time. By uniqueness of solutions, $P_\sigma\psi(T) = \psi(T)$.
	\end{proof}
	
	\begin{proposition}[Obstruction]\label{prop_obstruction}
		If $\mathbf{x}^{*}$ encodes an MIS and there exists $\sigma\in\mathrm{Aut}(G,w)$ with $P_\sigma|\mathbf{x}^{*}\rangle\neq|\mathbf{x}^{*}\rangle$, then $|\mathbf{x}^{*}\rangle\notin\mathcal{R}(\phi_B)$: the Pasqal equation \eqref{pasqal_equation_explicit} \emph{cannot} be driven from $\phi_B$ to $|\mathbf{x}^{*}\rangle$.
	\end{proposition}
	
	\begin{example}\label{example_two_atoms}
		For $n=2$ atoms forming a single edge, the swap $\sigma=(1\;2)\in\mathrm{Aut}(G,w)$ maps $|10\rangle\mapsto |01\rangle$. Neither $|10\rangle$ nor $|01\rangle$ is reachable from $|00\rangle$. The symmetric superposition $(|10\rangle+|01\rangle)/\sqrt{2}$ is, however, reachable.
	\end{example}
	
	\subsection{The MIS subspace as relaxed target}\label{subsec:symmetry_adapted_target}
	
	By Proposition~\ref{prop_symmetry_constraint}, no single computational-basis state $|\mathbf{x}^{*}\rangle$ encoding a maximum independent set is reachable from $\phi_B$ as soon as some $\sigma\in\mathrm{Aut}(G,w)$ moves $\mathbf{x}^{*}$. We therefore \emph{relax the target} from an individual optimal bitstring to the whole subspace it spans. Let
	\begin{equation}\label{mis_set_def}
		\mathrm{MIS}(G)\coloneqq\{\mathbf{x}\in\{0,1\}^n : \mathbf{x}\mbox{ encodes an MIS of }G\},\qquad d\coloneqq|\mathrm{MIS}(G)|,
	\end{equation}
	and define the \textbf{MIS subspace}
	\begin{equation}\label{mis_subspace_def}
		\Hmis\;\defeq\;\spn_{\mathbb{C}}\bigl\{|\mathbf{x}\rangle:\mathbf{x}\in\mathrm{MIS}(G)\bigr\},\qquad \dim\Hmis=d .
	\end{equation}
	
	\smallskip\noindent\textbf{Relaxed goal.} Reach a state from which a computational-basis measurement returns a maximum independent set \emph{with certainty}; that is, steer the system to
	\begin{equation}\label{relaxed_goal}
		\psi(T)\in\Hmis\qquad\Longleftrightarrow\qquad \bigl\|P_{\Hmis}\,\psi(T)\bigr\|=1 ,
	\end{equation}
	where $P_{\Hmis}:\mathcal{H}\to\Hmis$ denotes the orthogonal projection onto $\Hmis$. Indeed, since $\{|\mathbf{y}\rangle:\mathbf{y}\in\mathrm{MIS}(G)\}$ is an orthonormal basis of $\Hmis$, the Born rule gives, for \emph{any} normalised $\psi$,
	\begin{equation}\label{measurement_MIS_certainty}
		\mathbb{P}\big[\mathbf{x}\mbox{ encodes an MIS of }G\big] = \sum_{\mathbf{y}\in\mathrm{MIS}(G)}\big|\langle\mathbf{y}|\psi\rangle\big|^2 = \bigl\|P_{\Hmis}\psi\bigr\|^2 ,
	\end{equation}
	so the success probability equals $1$ precisely when $\psi\in\Hmis$. The relaxed target is thus the \emph{affine sphere} $S(\Hmis)=\{\psi\in\Hmis:\|\psi\|=1\}$, an entire manifold rather than a single point, and it is compatible with the symmetry constraint of Proposition~\ref{prop_symmetry_constraint}, because $\mathrm{Aut}(G,w)$ permutes $\mathrm{MIS}(G)$ and hence preserves $\Hmis$.
	
	A distinguished element of $S(\Hmis)$ is the \textbf{symmetry-adapted MIS state}
	\begin{equation}\label{symmetry_adapted_MIS}
		\psi_{\mathrm{MIS}} \;\coloneqq\; \frac{1}{\sqrt{d}}\sum_{\mathbf{x}\in\mathrm{MIS}(G)} |\mathbf{x}\rangle\ \in\ \Hmis\cap\mathcal{H}^{\mathrm{Aut}(G,w)},
	\end{equation}
	the uniform superposition over all maximum independent sets, which is $P_\sigma$-invariant for every $\sigma\in\mathrm{Aut}(G,w)$. We keep $\psi_{\mathrm{MIS}}$ only as a convenient representative \textup{(}for instance in the unique-MIS Corollary~\ref{cor_unique_MIS}\textup{)}: the protocol below is \emph{not} required to reach this particular vector, only to reach the subspace $\Hmis$.
	
	\subsection{The dynamical Lie algebra}\label{subsec:Lie_group}
	
	\begin{definition}[Dynamical Lie algebra]\label{def_dynamical_Lie_pasqal}
		The dynamical Lie algebra of the Pasqal system \eqref{pasqal_equation_explicit} is
		\begin{equation}\label{Lie_algebra_pasqal}
			\mathfrak{g} \;\coloneqq\; \mathrm{Lie}\big(\{-iH_0,\,-iH_1,\,-iH_2\}\big) \;\subseteq\; \mathfrak{su}(2^n),
		\end{equation}
		with $H_0$, $H_1$, $H_2$ as in \eqref{drift_hamiltonian_pasqal}--\eqref{control_hamiltonians_pasqal}, and $\mathcal{G}\coloneqq\langle\exp(\mathfrak{g})\rangle\subseteq SU(2^n)$ the associated connected Lie subgroup.
	\end{definition}
	
	\begin{lemma}[Global rotations in $\mathfrak{g}$]\label{lemma_global_rotations}
		The Lie algebra $\mathfrak{g}$ contains the global Pauli generators
		\begin{equation}\label{global_generators}
			-iH_1 = -\frac{i}{2}\sum_{k=1}^n X_k,\quad -iH_2 = i\sum_{k=1}^n \hat n_k,\quad \mbox{and}\quad [-iH_1,-iH_2] = \frac{i}{2}\sum_{k=1}^n Y_k.
		\end{equation}
		In particular, for every $\theta\in\mathbb{R}$, the global $Y$-rotation $\exp\!\big(\tfrac{i\theta}{2}\sum_{k=1}^n Y_k\big)\in\mathcal{G}$.
	\end{lemma}
	\begin{proof}
		A direct computation using $[\hat n_k,X_k]=[(I-Z_k)/2,X_k]=-[Z_k,X_k]/2 = -iY_k$ gives
		\begin{equation}
			[H_1,H_2] = \Big[\tfrac{1}{2}\textstyle\sum_{k=1}^{n} X_k,\,-\textstyle\sum_{l=1}^{n}\hat n_l\Big] = -\tfrac{1}{2}\textstyle\sum_{k=1}^{n} [X_k,\hat n_k] = -\tfrac{i}{2}\textstyle\sum_{k=1}^{n} Y_k.
		\end{equation}
		Hence $[-iH_1,-iH_2]=-[H_1,H_2]=(i/2)\sum_{k=1}^{n} Y_k\in\mathfrak{g}$.
	\end{proof}
	
	\begin{proposition}[The QAOA state $|+\rangle^{\otimes n}$ lies in the orbit of $\phi_B$]\label{prop_same_orbit}
		The standard QAOA initial state $|+\rangle^{\otimes n}$ belongs to the orbit of $\phi_B=|0\rangle^{\otimes n}$:
		\begin{equation}\label{zero_in_orbit}
			|+\rangle^{\otimes n}\in\mathcal{G}\cdot\phi_B .
		\end{equation}
		Consequently $\phi_B=|0\rangle^{\otimes n}$ and $|+\rangle^{\otimes n}$ generate the \emph{same} reachable set, so the choice between the annealing initialisation $|0\rangle^{\otimes n}$ and the QAOA initialisation $|+\rangle^{\otimes n}$ is immaterial for controllability.
	\end{proposition}
	\begin{proof}
		The single-qubit $Y$-rotation $\exp(-iY\theta/2)$ is the rotation matrix $\bigl[\begin{smallmatrix}\cos\theta/2&-\sin\theta/2\\\sin\theta/2&\cos\theta/2\end{smallmatrix}\bigr]$; at $\theta=\pi/2$ it maps $(1,0)^\top=|0\rangle$ to $(1,1)^\top/\sqrt{2}=|+\rangle$. Hence the global $Y$-rotation with $\theta=\pi/2$ gives
		\begin{equation}\label{global_Y_rotation}
			\exp\!\Big(-\frac{i\pi}{4}\sum_{k=1}^n Y_k\Big)\,\phi_B \;=\; \Big(\exp\!\big(-\tfrac{i\pi}{4}Y\big)\Big)^{\otimes n}|0\rangle^{\otimes n} \;=\; |+\rangle^{\otimes n}.
		\end{equation}
		By Lemma~\ref{lemma_global_rotations}, the unitary on the left belongs to $\mathcal{G}$, so $|+\rangle^{\otimes n}\in\mathcal{G}\cdot\phi_B$; since $\mathcal{G}$ is a group, the two orbits coincide, $\mathcal{G}\cdot|+\rangle^{\otimes n}=\mathcal{G}\cdot\phi_B$.
	\end{proof}
	
	\begin{remark}[Physical realization of the global $Y$-rotation]\label{remark_Y_realization}
		The global $Y$-rotation $\exp\!\big(\tfrac{i\theta}{2}\sum_{k=1}^{n} Y_k\big)$ is not directly generated by a single control pulse, since it corresponds to a Lie bracket $[-iH_1,-iH_2]$, not to $H_1$ or $H_2$ alone. Physically, it is synthesized by alternating short, intense Rabi and detuning pulses. For small $\varepsilon>0$, the Baker--Campbell--Hausdorff formula \cite[formula (E.6) at page 364]{d2021introduction} gives
		\begin{equation}\label{BCH_approximation}
			e^{-iH_1\varepsilon}\,e^{-iH_2\varepsilon}\,e^{iH_1\varepsilon}\,e^{iH_2\varepsilon} = e^{-[H_1,H_2]\varepsilon^2 + O(\varepsilon^3)},
		\end{equation}
		so that $O(1/\varepsilon^2)$ repetitions of this commutator sequence approximate $\exp\!\big(-[H_1,H_2]\tau\big)$ for any desired $\tau$. The drift $H_0$ contributes corrections of order $\varepsilon$ in each factor, which are controlled by choosing $\varepsilon$ small enough (i.e., by using sufficiently strong and short pulses with $\Omega_{\max}\gg 1/\varepsilon$ and $\delta_{\max}\gg 1/\varepsilon$).
	\end{remark}
	
	\subsection{Spectral gap of the Pasqal Hamiltonian along a control path}\label{subsec:spectral_gap_Pasqal}
	
	We define the spectral gap directly for the Pasqal Hamiltonian \eqref{neutral_atoms_hamiltonian}, parametrized by the controls $(\Omega,\delta)$.
	
	\begin{definition}[Control path and spectral gap]\label{def_control_path}
		An \textbf{adiabatic control path} is a smooth curve $\gamma:[0,1]\to [0,\Omega_{\max}]\times[-\delta_{\max},\delta_{\max}]$, written $\gamma(s)=(\Omega(s),\delta(s))$, connecting
		\begin{equation}\label{control_path_endpoints}
			\gamma(0) = (0,\,-\delta_0)\qquad\mbox{and}\qquad \gamma(1) = (0,\,+\delta_0),
		\end{equation}
		for some $\delta_0>0$. The \textbf{instantaneous Pasqal Hamiltonian} along $\gamma$ is
		\begin{equation}\label{instantaneous_Pasqal}
			H_{\gamma}(s) \;\coloneqq\; \frac{\hbar\,\Omega(s)}{2}\sum_{k=1}^n X_k \;-\; \hbar\,\delta(s)\sum_{k=1}^n \hat n_k \;+\; H_0.
		\end{equation}
		Its eigenvalues are denoted $\lambda_0(s)\leq\lambda_1(s)\leq\cdots$, and the \textbf{spectral gap along $\gamma$} is
		\begin{equation}\label{spectral_gap_Pasqal}
			\Delta_\gamma(s) \;\coloneqq\; \lambda_1(s) - \lambda_0(s).
		\end{equation}
	\end{definition}
	
	A natural choice of adiabatic control path is the \emph{bell-shaped} schedule:
	\begin{equation}\label{bell_schedule}
		\Omega(s) = \Omega_{\max} 4s(1-s),\qquad \delta(s) = \delta_0(2s-1),\qquad s\in[0,1].
	\end{equation}
	At $s=0$: $(\Omega,\delta)=(0,-\delta_0)$; at $s=1/2$: $(\Omega,\delta)=(\Omega_{\max},0)$; at $s=1$: $(\Omega,\delta)=(0,+\delta_0)$.
	
	\medskip
	\noindent\textbf{Notation for the endpoint analysis.}
	Throughout, $\ham$ denotes the instantaneous Pasqal Hamiltonian $H_{\gamma}(s)$
	along the control path $\gamma$, defined in \eqref{instantaneous_Pasqal}; we write $\lvert\bx\rvert\defeq\sum_{k=1}^{n}x_{k}$
	for the Hamming weight of $\bx\in\{0,1\}^{n}$ (the cardinality of the encoded
	subset), with $\bze$ the all-zero string, $\lvert\bze\rangle=\lvert0\rangle^{\otimes n}$.
	Under the physical convention $C_{6}>0$ (repulsive Rydberg interaction), the
	couplings $J_{ij}\defeq C_{6}/R_{i,j}^{6}>0$ make $H_{0}$ diagonal in the
	computational basis, with
	\begin{equation}
		\langle\bx\rvert H_{0}\lvert\bx\rangle=\sum_{\substack{i,j=1 \\ i<j}}^{n}J_{ij}x_{i}x_{j}\ge0
	\end{equation}
	by \eqref{drift_on_comp_basis}; we set the \emph{blockade scale}
	\begin{equation}\label{loc:blo}
		\Jblo\defeq\frac{C_{6}}{R_{b}^{6}}>0 .
	\end{equation}
	At the endpoints \eqref{control_path_endpoints}, the operators
	\begin{equation}\label{loc:endpoints}
		\ham(0)=H_{0}+\hbar\delta_{0}\sum_{k=1}^{n}\hat n_{k},
		\qquad
		\ham(1)=H_{0}-\hbar\delta_{0}\sum_{k=1}^{n}\hat n_{k}
	\end{equation}
	are diagonal in the computational basis (sums of diagonal operators), with
	eigenpairs $\bigl(E_{a}(\bx),\lvert\bx\rangle\bigr)$ and
	$\bigl(E_{b}(\bx),\lvert\bx\rangle\bigr)$, where
	\begin{equation}\label{loc:Ea}
		E_{a}(\bx)\defeq\langle\bx\rvert H_{0}\lvert\bx\rangle+\hbar\delta_{0}\,\lvert\bx\rvert,
	\end{equation}
	\begin{equation}\label{loc:Eb}
		E_{b}(\bx)\defeq\langle\bx\rvert H_{0}\lvert\bx\rangle-\hbar\delta_{0}\,\lvert\bx\rvert .
	\end{equation}
	For $\bx$ encoding an independent set, \eqref{drift_on_indep_set} gives
	$\langle\bx\rvert H_{0}\lvert\bx\rangle=\tau(\bx)$, where the \emph{tail energy} is
	\begin{equation}\label{loc:tail}
		\tau(\bx)\defeq\!\!\sum_{\substack{\{i,j\}\notin E\\ i<j}}\!\! J_{ij}\,x_{i}x_{j}\ =\!\!\sum_{\substack{\{i,j\}\notin E\\ i<j}}\!\! \frac{C_{6}}{R_{i,j}^{6}}\,x_{i}x_{j}\ \ge 0,
	\end{equation}
	while \eqref{drift_penalty} gives $\langle\bx\rvert H_{0}\lvert\bx\rangle\ge\Jblo$
	for any $\bx$ violating independence. We abbreviate the worst-case tail over
	independent sets by
	\begin{equation}\label{loc:taumax}
		\tau^{\max}\defeq\max\bigl\{\tau(\bx)\ :\ \bx\in\{0,1\}^{n}\ \text{independent}\bigr\}\ \ge 0 .
	\end{equation}
	Finally, $\His\defeq\spn\{\lvert\bx\rangle:\bx\text{ encodes an independent set}\}$
	is the independent-set subspace of Corollary~\ref{cor_effective_hamiltonian}.
	
	Define the
	\textbf{energy-scale window} as the conditions
	\begin{equation}\label{energy_window}
		\tag{W}
		\tau^{\max}\ <\ \hbar\delta_{0}
		\qquad\text{and}\qquad
		\hbar\delta_{0}\,n+\tau^{\max}\ <\ \Jblo=\frac{C_{6}}{R_{b}^{6}} .
	\end{equation}
	In Lemma \ref{lem:window}, we shall see under which conditions \eqref{energy_window} is nonempty and, in Corollary \ref{cor:window_delta0}, we will see \eqref{energy_window} as a two-sided bound on the detuning amplitude $\delta_0$.

	\begin{lemma}[Combinatorial ground states of $\ham(1)$]\label{lem:window}
		Assume $C_{6}>0$, $\delta_{0}>0$, and let \eqref{energy_window} hold. Then:
		\begin{enumerate}
			\item[\textup{(i)}] every global minimizer of $E_{b}$ encodes an independent set of $G$;
			\item[\textup{(ii)}] every global minimizer of $E_{b}$ encodes a \emph{maximum} independent set, i.e.
			\begin{equation}\label{loc:argmin-in-MIS}
				\arg\min_{\bx\in\{0,1\}^{n}}E_{b}(\bx)\ \subseteq\ \MIS(G);
			\end{equation}
			\item[\textup{(iii)}] conversely $\by\in\MIS(G)$ is a global minimizer of $E_{b}$ iff
			$\tau(\by)=\min_{\by'\in\MIS(G)}\tau(\by')$. In the idealised (hard) blockade
			$\tau\equiv0$, equality holds in \eqref{loc:argmin-in-MIS}:
			$\arg\min_{\bx}E_{b}(\bx)=\MIS(G)$.
		\end{enumerate}
		Moreover, the window \eqref{energy_window} is nonempty whenever
		$(n+1)\,\tau^{\max}<C_{6}/R_{b}^{6}$; in the idealised blockade it reduces to the
		manifestly nonempty interval $0<\hbar\delta_{0}<C_{6}/(n R_{b}^{6})$.
	\end{lemma}
	
	\begin{proof}
		Fix a maximum independent set encoded by $\by_{*}$, so $\lvert\by_{*}\rvert=\alpha(G)$
		and, by \eqref{drift_on_indep_set}--\eqref{loc:taumax},
		\begin{equation}\label{loc:ystar}
			E_{b}(\by_{*})=\tau(\by_{*})-\hbar\delta_{0}\,\alpha(G)\ \le\ \tau^{\max}-\hbar\delta_{0}\,\alpha(G).
		\end{equation}
		
		\smallskip
		\emph{Step 1 (independent sets beat violators - proof of (i)).}
		Let $\bx$ violate independence on some edge. By \eqref{drift_penalty} and
		$\lvert\bx\rvert\le n$,
		\[
		E_{b}(\bx)=\langle\bx\rvert H_{0}\lvert\bx\rangle-\hbar\delta_{0}\lvert\bx\rvert
		\ \ge\ \Jblo-\hbar\delta_{0}\,n .
		\]
		Subtracting \eqref{loc:ystar},
		\begin{equation}\label{eq_E_b_estimate}
			E_{b}(\bx)-E_{b}(\by_{*})
			\ \ge\ \bigl(\Jblo-\hbar\delta_{0}\,n-\tau^{\max}\bigr)+\hbar\delta_{0}\,\alpha(G)
			\ >\ 0 ,
		\end{equation}
		because the bracket is positive by the second inequality in \eqref{energy_window}
		and $\hbar\delta_{0}\,\alpha(G)\ge0$. Hence no violator is a global minimizer:
		every minimizer is an independent set.
		
		\smallskip
		\emph{Step 2 (maximum cardinality wins - proof of (ii)).}
		Let $\bz$ encode an independent set that is \emph{not} maximum,
		$\lvert\bz\rvert\le\alpha(G)-1$. Using $\tau(\bz)\ge0$ and \eqref{loc:ystar},
		\begin{equation}\label{no_max_indep_set}
			E_{b}(\bz)-E_{b}(\by_{*})
			=\bigl[\tau(\bz)-\tau(\by_{*})\bigr]-\hbar\delta_{0}\bigl(\lvert\bz\rvert-\alpha(G)\bigr)
			\ \ge\ -\tau^{\max}+\hbar\delta_{0}\,\bigl(\alpha(G)-\lvert\bz\rvert\bigr)
			\ \ge\ \hbar\delta_{0}-\tau^{\max}\ >\ 0 ,
		\end{equation}
		the last inequality being the first condition in \eqref{energy_window}. Thus no
		non-maximum independent set is a global minimizer. Combining with Step~1 proves
		\eqref{loc:argmin-in-MIS}.
		
		\smallskip
		\emph{Step 3 (degeneracy structure - proof of (iii)).}
		For $\by,\by'\in\MIS(G)$ one has $\lvert\by\rvert=\lvert\by'\rvert=\alpha(G)$, so
		$E_{b}(\by)-E_{b}(\by')=\tau(\by)-\tau(\by')$. Hence $E_{b}$ restricted to
		$\MIS(G)$ is minimised exactly on the maximum independent sets of least tail
		energy. If $\tau\equiv0$ (idealised blockade), all of $\MIS(G)$ ties at
		$E_{b}=-\hbar\delta_{0}\,\alpha(G)$ and, by Steps~1--2, strictly undercuts every
		other bitstring; thus $\arg\min E_{b}=\MIS(G)$.
		
		\smallskip
		\emph{Nonemptiness of \eqref{energy_window}.} A common value $\hbar\delta_{0}$
		satisfying both inequalities exists iff
		$\tau^{\max}<(\Jblo-\tau^{\max})/n$, i.e.\ $(n+1)\tau^{\max}<\Jblo$. For
		$\tau^{\max}=0$ \eqref{energy_window} is $0<\hbar\delta_{0}<\Jblo/n$.
	\end{proof}
	
	\begin{corollary}[The window as a detuning interval]\label{cor:window_delta0}
		Fix the atom layout $\{r_i\}_{i=1}^{n}$ (hence the graph $G$, the size $n$, all
		pairwise distances $R_{i,j}$, the tail $\tau^{\max}$, and the threshold radius
		$R_{b}$) and the coefficient $C_{6}>0$. Then the energy-scale window
		\eqref{energy_window} is equivalent to the two-sided bound on the detuning
		amplitude
		\begin{equation}\label{window_delta0_interval}
			\frac{\tau^{\max}}{\hbar}\ <\ \delta_{0}\ <\ \frac{1}{\hbar n}\Bigl(\frac{C_{6}}{R_{b}^{6}}-\tau^{\max}\Bigr),
		\end{equation}
		a nonempty interval if and only if $(n+1)\,\tau^{\max}<C_{6}/R_{b}^{6}$; in the
		idealised blockade $\tau^{\max}=0$ it collapses to $0<\delta_{0}<C_{6}/(\hbar n R_{b}^{6})$.
	\end{corollary}
	
	\begin{proof}
		Both inequalities in \eqref{energy_window} are affine in $\delta_{0}$, while
		$\tau^{\max}$, $\Jblo=C_{6}/R_{b}^{6}$ and $n$ do not depend on $\delta_{0}$;
		solving the first for $\delta_{0}>\tau^{\max}/\hbar$ and the second for
		$\delta_{0}<(\Jblo-\tau^{\max})/(\hbar n)$ gives \eqref{window_delta0_interval}.
		The interval is nonempty iff its lower bound is below its upper bound, i.e.\
		$\tau^{\max}/\hbar<(\Jblo-\tau^{\max})/(\hbar n)$, equivalently
		$(n+1)\tau^{\max}<\Jblo$, recovering the nonemptiness criterion of
		Lemma~\ref{lem:window}.
	\end{proof}
	
	\begin{remark}[Reading of the window]\label{rem:window}
		The two inequalities in \eqref{energy_window} are exactly the two physical
		requirements that the qualitative hypothesis $C_{6}/R_{b}^{6}\gg\hbar\delta_{0}$
		is meant to encode: the \emph{upper} bound $\hbar\delta_{0}\,n+\tau^{\max}<\Jblo$
		is the blockade condition proper (one violated edge costs more than the largest
		possible detuning reward $\hbar\delta_{0}\,n$ plus tails), while the \emph{lower}
		bound $\tau^{\max}<\hbar\delta_{0}$ guarantees that one extra excitation
		($+\hbar\delta_{0}$) outweighs the worst tail, so that maximising cardinality, not merely feasibility, is energetically selected. Both are uniform in
		the choice of MIS, and the window collapses onto $\hbar\delta_{0}\in(0,\Jblo/n)$
		in the idealised blockade used in Corollary~\ref{cor_effective_hamiltonian}.
	\end{remark}
	
	\begin{lemma}[Ground states at the endpoints]\label{lemma_endpoints}
		Consider the Pasqal Hamiltonian \eqref{instantaneous_Pasqal} with $C_{6}>0$ and
		$\delta_{0}>0$.
		\begin{enumerate}
			\item[\textup{(a)}] At $(\Omega,\delta)=(0,-\delta_{0})$ the operator
			$\ham(0)=H_{0}+\hbar\delta_{0}\sum_{k=1}^{n}\hat n_{k}$ is diagonal in the
			computational basis with eigenvalues $E_{a}(\bx)=\langle\bx\rvert H_{0}\lvert\bx\rangle
			+\hbar\delta_{0}\lvert\bx\rvert\ge0$. Its \emph{unique} ground state is
			$\lvert0\rangle^{\otimes n}$, with eigenvalue $0$, and the spectral gap above
			it equals exactly $\hbar\delta_{0}>0$ \textup{(}attained on the $n$
			single-excitation states\textup{)}. No blockade hypothesis is needed.
			\item[\textup{(b)}] At $(\Omega,\delta)=(0,+\delta_{0})$ the operator
			$\ham(1)=H_{0}-\hbar\delta_{0}\sum_{k=1}^{n}\hat n_{k}=H_C$ is diagonal with eigenvalues
			$E_{b}(\bx)=\langle\bx\rvert H_{0}\lvert\bx\rangle-\hbar\delta_{0}\lvert\bx\rvert$. Define
			\begin{equation}\label{loc:Hmis}
				\Hmis\defeq\spn_{\mathbb{C}}\bigl\{\lvert\bx\rangle:\bx\in\MIS(G)\bigr\},
				\qquad \dim\Hmis=d ;
			\end{equation}
			Under the energy-scale window \eqref{energy_window}
			\begin{itemize}
				\item its ground manifold is
				spanned by maximum-independent-set strings (contained in $\Hmis$);
				\item in the idealised blockade its ground manifold equals $\Hmis$;
				\item the gap separating $\Hmis$ from the rest of the spectrum is
				$\ge\hbar\delta_{0}-\tau^{\max}>0$.
			\end{itemize}
			In particular the ground manifold of $\ham(1)$ is contained in $\Hmis$
			under \eqref{energy_window} \emph{alone}, no blockade idealisation and
			no symmetry hypothesis is needed, so the adiabatic evolution of
			Section~\ref{subsec:main_result} lands in the relaxed target $\Hmis$ of
			\eqref{relaxed_goal}.
			% The symmetry-adapted state $\psi_{\mathrm{MIS}}\in\Hmis$
			% of \eqref{symmetry_adapted_MIS} is one admissible element of this manifold.
		\end{enumerate}
		\end{lemma}
		
		\begin{proof}
		\textbf{Part (a).}
		By \eqref{eq_projector_action} the operators $H_{0}$ and $\sum_{k=1}^{n}\hat n_{k}$ are both
		diagonal in $\{\lvert\bx\rangle\}$, hence so is their sum
		$\ham(0)=H_{0}+\hbar\delta_{0}\sum_{k=1}^{n}\hat n_{k}$; its eigenpairs are
		$\bigl(E_{a}(\bx),\lvert\bx\rangle\bigr)$ with $E_{a}$ as in \eqref{loc:Ea}.
		Since $C_{6}>0$ gives $\langle\bx\rvert H_{0}\lvert\bx\rangle\ge0$ by
		\eqref{drift_on_comp_basis}, and $\delta_{0}>0$ gives $\hbar\delta_{0}\lvert\bx\rvert\ge0$,
		we have $E_{a}(\bx)\ge0$ for every $\bx$.
		
		\emph{Uniqueness of the ground state.} $E_{a}(\bx)=0$ forces both summands to
		vanish; in particular $\hbar\delta_{0}\lvert\bx\rvert=0$ with $\delta_{0}>0$
		forces $\lvert\bx\rvert=0$, i.e.\ $\bx=\bze$. Conversely $E_{a}(\bze)=0$ since
		$\langle\bze\rvert H_{0}\lvert\bze\rangle=0$ and $\lvert\bze\rvert=0$. Thus
		$\lvert\bze\rangle=\lvert0\rangle^{\otimes n}$ is the \emph{only} eigenvector with
		eigenvalue $0$; as $0=\min_{\bx}E_{a}(\bx)$, it is the unique ground state.
		
		\emph{Value of the gap.} For $\bx\ne\bze$ we have $\lvert\bx\rvert\ge1$, hence
		$E_{a}(\bx)\ge\hbar\delta_{0}\lvert\bx\rvert\ge\hbar\delta_{0}$. Equality
		$E_{a}(\bx)=\hbar\delta_{0}$ requires $\lvert\bx\rvert=1$ and
		$\langle\bx\rvert H_{0}\lvert\bx\rangle=0$; a single-excitation string
		$\bx=e_{k}$ satisfies both (one excited site cannot form an interacting pair, so
		the sum in \eqref{drift_on_comp_basis} is empty). There are exactly $n$ such strings,
		and any $\bx$ with $\lvert\bx\rvert\ge2$ has $E_{a}(\bx)\ge2\hbar\delta_{0}$. Hence
		the first excited level sits at $\hbar\delta_{0}$ with multiplicity $n$, and the
		spectral gap above the ground state equals exactly $\hbar\delta_{0}>0$. This
		proves~(a); note that no smallness of $C_{6}/R_{b}^{6}$ versus $\hbar\delta_{0}$
		was used.
		
		\medskip
		\textbf{Part (b).}
		As in (a), $\ham(1)=H_{0}-\hbar\delta_{0}\sum_{k=1}^{n}\hat n_{k}$ is diagonal with
		eigenpairs $\bigl(E_{b}(\bx),\lvert\bx\rangle\bigr)$, $E_{b}$ as in \eqref{loc:Eb}.
		Consequently, the ground manifold of $\ham(1)$ is
		\begin{equation}\label{loc:groundmanifold}
			\ker\bigl(\ham(1)-\lambda_{0}I\bigr)
			=\spn_{\mathbb{C}}\Bigl\{\lvert\bx\rangle:\bx\in\arg\min_{\bx'}E_{b}(\bx')\Bigr\},
			\qquad \lambda_{0}=\min_{\bx}E_{b}(\bx).
		\end{equation}
		
		\emph{Identification of the ground manifold.} Under the window \eqref{energy_window},
		Lemma~\ref{lem:window}(ii) gives $\arg\min E_{b}\subseteq\MIS(G)$, so the ground
		manifold is contained in $\Hmis$ of \eqref{loc:Hmis}. In the idealised blockade
		($\tau\equiv0$, equivalently the projected dynamics of
		Corollary~\ref{cor_effective_hamiltonian} on $\His$, on which
		$\ham(1)\!\restriction_{\His}=P_{\mathrm{IS}}\bigl(-\hbar\delta_{0}\sum_{k=1}^{n}\hat n_{k}\bigr)P_{\mathrm{IS}}$
		acts as multiplication by $-\hbar\delta_{0}\lvert\bx\rvert$), Lemma~\ref{lem:window}(iii)
		gives $\arg\min E_{b}=\MIS(G)$ \emph{exactly}, whence the ground manifold is the
		full $\Hmis$, of dimension $d=\lvert\MIS(G)\rvert$.
		
		\emph{Separation from the rest of the spectrum.} To continue with the proof, let us perform some estimates, based on some computations in Steps~1--2 of the proof of Lemma~\ref{lem:window}. Take $\by$ an arbitrary element of $\MIS(G)$. On the one hand, for any $\bx\notin\MIS(G)$ not representing an independent set for the graph $G$, by \eqref{eq_E_b_estimate}, we have
		\begin{equation}\label{eq_E_b_estimate_II}
			E_{b}(\bx)-E_{b}(\by)
			\ \ge\ \bigl(\Jblo-\hbar\delta_{0}\,n-\tau^{\max}\bigr)+\hbar\delta_{0}\,\alpha(G)>\hbar\delta_{0},
		\end{equation}
		where in the last inequality we employed the second inequality in \eqref{energy_window}.
		On the other hand, for every $\bx\notin\MIS(G)$ representing an independent set for the graph $G$, from \eqref{no_max_indep_set}, we get
		\begin{equation}\label{no_max_indep_set_II}
			E_{b}(\bx)-E_{b}(\by)
			\ \ge\ \hbar\delta_{0}-\tau^{\max}.
		\end{equation}
		Hence, for every $\bx\notin\MIS(G)$ and for any $\by\in \MIS(G)$, we have
		\begin{equation}\label{}
			E_{b}(\bx)-E_{b}(\by)
			\ \ge\ \hbar\delta_{0}-\tau^{\max},
		\end{equation}
		whence
		\begin{equation}\label{loc:bandgap}
			\min_{\bx\notin\MIS(G)}E_{b}(\bx)\ -\ \max_{\by\in\MIS(G)}E_{b}(\by)
			\ \ge\ \hbar\delta_{0}-\tau^{\max}\ >\ 0 ,
		\end{equation}
		In particular, in the idealised blockade the
		endpoint gap is exactly $\hbar\delta_{0}$, consistent with
		Assumption~\ref{assumption_gap_Pasqal} at $s=1$.
		
		\emph{The ground manifold sits inside $\Hmis$.} Combining the identification
		of the ground manifold with the separation estimate \eqref{loc:bandgap}, every
		ground state of $\ham(1)$ lies in $\Hmis$ under \eqref{energy_window} alone.
		Hence whichever vector of the ground manifold is selected at $s=1$, the exact
		ground state in the idealised blockade, or any minimal-tail combination at finite
		blockade, belongs to $\Hmis$; in particular the symmetry-adapted state
		$\psi_{\mathrm{MIS}}\in\Hmis$ of \eqref{symmetry_adapted_MIS} is admissible. No
		identification of a \emph{single} ground vector is needed for the relaxed goal
		\eqref{relaxed_goal}. This proves~(b).
	\end{proof}
		
		\begin{remark}[Why the relaxed goal removes the finite-blockade splitting]\label{rem:splitting}
		At finite blockade the $d$-fold $\Hmis$-band is split by the tails into a width
		$\max_{\by\in\MIS}\tau(\by)-\min_{\by\in\MIS}\tau(\by)\le\tau^{\max}$, which by
		\eqref{energy_window} stays \emph{below} the binding gap $\hbar\delta_{0}-\tau^{\max}$
		of \eqref{loc:bandgap}; thus $\Hmis$ remains a well-isolated band. For the
		\emph{strict} target $\psi_{\mathrm{MIS}}$ this splitting mattered: $\psi_{\mathrm{MIS}}$
		is the exact ground state only when the splitting vanishes on the relevant invariant
		sector (for instance when $\Aut(G,w)$ acts transitively on $\MIS(G)$). For the
		\emph{relaxed} goal \eqref{relaxed_goal} the splitting is immaterial: any vector of
		the $\Hmis$-band, the exact ground state, or any superposition produced by
		intra-band diabatic transitions, still lies in $\Hmis$, hence still returns a
		maximum independent set with certainty by \eqref{measurement_MIS_certainty}. Only
		leakage \emph{out} of $\Hmis$, across the binding gap \eqref{loc:bandgap}, can spoil
		the outcome, and it is precisely such leakage that the integral tracking functional of
		Section~\ref{sec:Advanced State Stabilization via Integral Tracking Functionals}
		is designed to suppress.
	\end{remark}
	
	\begin{assumption}[Spectral gap hypothesis for the Pasqal Hamiltonian]\label{assumption_gap_Pasqal}
		There exist an adiabatic control path $\gamma$ in the sense of Definition~\ref{def_control_path}, with endpoints $\gamma(0)=(0,-\delta_0)$ and $\gamma(1)=(0,+\delta_0)$ sharing the detuning $\delta_0>0$ of the window \eqref{energy_window}, and constants $c>0$, $q\in\mathbb{N}$ (independent of $n$ and $G$), such that
		\begin{equation}\label{spectral_gap_hypothesis_Pasqal}
			\Delta_\gamma^{\min} \;\coloneqq\; \min_{s\in[0,1]}\Delta_\gamma(s) \;\geq\; \frac{c}{n^q}.
		\end{equation}
	\end{assumption}
	
	\subsection{Main result: approximate reachability of $\Hmis$ via an adiabatic anneal}\label{subsec:main_result}
	
	Since $\phi_B=|0\rangle^{\otimes n}$ is, by Lemma~\ref{lemma_endpoints}(a), the unique ground state of the Pasqal Hamiltonian $H_\gamma(0)$ at the control endpoint $(\Omega,\delta)=(0,-\delta_0)$, no preliminary state-preparation (``reach'') phase is required: the protocol consists of a \emph{single adiabatic sweep} of the physical Pasqal Hamiltonian \eqref{neutral_atoms_hamiltonian}.
	
	\begin{description}
		\item[\emph{Adiabatic anneal}.] Starting from $\phi_B=|0\rangle^{\otimes n}$, the ground state of $H_\gamma(0)$ (Lemma~\ref{lemma_endpoints}(a)), slowly sweep $(\Omega(t),\delta(t))$ along the adiabatic control path $\gamma$ from $(0,-\delta_0)$ to $(0,+\delta_0)$, with time rescaling $s=t/T_{\mathrm{anneal}}$. This is exactly the adiabatic-evolution scheme of \cite{farhi2000quantum}: the instantaneous ground state is tracked from $H_\gamma(0)$ to $H_\gamma(1)$, whose ground manifold is contained in the MIS subspace $\Hmis$ (Lemma~\ref{lemma_endpoints}(b)).
	\end{description}
	
	\begin{remark}[Equivalence with the standard QAOA initialisation]\label{remark_qaoa_init}
		The protocol could equally be started from the standard QAOA state $|+\rangle^{\otimes n}$. Indeed, by Proposition~\ref{prop_same_orbit} one has $|+\rangle^{\otimes n}\in\mathcal{G}\cdot\phi_B$, so $\phi_B=|0\rangle^{\otimes n}$ is reachable from $|+\rangle^{\otimes n}$ by an admissible control; the global $Y$-rotation \eqref{global_Y_rotation} realises this transfer \emph{approximately}, and by Remark~\ref{remark_Y_realization} it is synthesised by a strong-pulse train of total duration $O(1)$ in $n$. Prepending this transfer recovers the two-phase reach--anneal protocol, the reach phase contributing only an $n$-independent additive constant to the total time. We adopt the annealing initialisation $\phi_B=|0\rangle^{\otimes n}$ precisely because it makes the reach phase unnecessary, so that the proof reduces to the adiabatic step alone.
	\end{remark}
	
	\begin{theorem}[Approximate state controllability into the MIS subspace $\Hmis$ via the Pasqal Hamiltonian]\label{thm_approximate_controllability_MIS}
		Consider the bilinear Schr\"{o}dinger equation \eqref{pasqal_equation_explicit} with controls $\Omega(t)\in[0,\Omega_{\max}]$ and $\delta(t)\in[-\delta_{\max},\delta_{\max}]$, and initial state $\psi(0)=\phi_B=|0\rangle^{\otimes n}$. Let $G=(V,E)$ be a unit-disk graph realized by the atom layout through \eqref{rydberg_radius}. Suppose:
		\begin{enumerate}
			\item[\textup{(H0)}] \textbf{energy-scale window:} the energy-scale window \eqref{energy_window} holds for some $\delta_0>0$\footnote{the validity of this assumption depends on the configuration \eqref{rydberg_radius} and, in particular, on the magnitude of the radius $R_b$};
			\item[\textup{(H1)}] \textbf{Gap non-closing (interior):} there exists an adiabatic control path $\gamma$ in the sense of Definition~\ref{def_control_path}, with endpoints $\gamma(0)=(0,-\delta_0)$, $\gamma(1)=(0,+\delta_0)$, along which the instantaneous spectral gap \eqref{spectral_gap_Pasqal} stays strictly positive,
			\begin{equation}\label{gap_positivity}
				\Delta_\gamma^{\min}\;=\;\min_{s\in[0,1]}\Delta_\gamma(s)\;>\;0 .
			\end{equation}
		\end{enumerate}
		\footnote{Only the qualitative positivity \eqref{gap_positivity} is used in the proof below; the polynomial lower bound $\Delta_\gamma^{\min}\ge c/n^q$ of Assumption~\ref{assumption_gap_Pasqal} is \emph{not} required for the existence statement and is invoked solely in the controllability-time estimate (Remark~\ref{remark_time_complexity}).}

		Then the relaxed goal \eqref{relaxed_goal} is met \emph{approximately}: there exist $T^{*}>0$ and $(\Omega^{*},\delta^{*})\in\mathscr{U}_{\mathrm{ad}}^{T^{*}}$ such that
		\begin{equation}\label{approximate_controllability_statement}
			\mathrm{dist}\big(\psi(T_{\mathrm{anneal}}),\,\Hmis\big)\ \le\ \varepsilon.
		\end{equation}
		Consequently, measurement of $\psi(T^{*})$ in the computational basis yields a Maximum Independent Set (MIS), with probability $1-\varepsilon^2$.
	\end{theorem}
	
	\begin{remark}[Division of labour between \textup{(H0)}, \textup{(H1)} and the polynomial-gap Assumption~\ref{assumption_gap_Pasqal}]\label{remark_hypotheses_roles}
		The two hypotheses of Theorem~\ref{thm_approximate_controllability_MIS} constrain two different aspects of the protocol and are logically independent. \textup{(H0)} is a pure \emph{energy-scale} condition relating the detuning $\delta_0$ to the couplings $C_6/R_b^6$ and the van der Waals tail $\tau^{\max}$ (it is approximately \eqref{energy_window}); it fixes the spectra only at the two \emph{endpoints} $s\in\{0,1\}$ through Lemma~\ref{lemma_endpoints}, placing $|0\rangle^{\otimes n}$ at the bottom of $H_\gamma(0)$ and the entire ground manifold of $H_\gamma(1)$ inside $\Hmis$, and it is independent of $n$. \textup{(H1)} is a purely \emph{qualitative} condition on the \emph{interior} of the path: it asks only that the instantaneous gap never close, $\Delta_\gamma^{\min}>0$, with no rate attached. For the existence statement \eqref{approximate_controllability_statement} this positivity is all that the adiabatic step (Step~1) uses, since for a \emph{fixed} graph $G$ a strictly positive gap suffices to drive the right-hand side of \eqref{adiabatic_bound_Pasqal} below any tolerance by enlarging $T_{\mathrm{anneal}}$.

		The \emph{quantitative} strengthening, the polynomial lower bound $\Delta_\gamma^{\min}\ge c/n^q$ of Assumption~\ref{assumption_gap_Pasqal}, with $c,q$ independent of $n$, is a strictly stronger, $n$-dependent hypothesis. It is \emph{not} part of Theorem~\ref{thm_approximate_controllability_MIS}: it plays no role in the qualitative reachability \eqref{approximate_controllability_statement} and is invoked only to turn the sufficient anneal time \eqref{adiabatic_time_Pasqal} into a polynomial-in-$n$ bound (Remark~\ref{remark_time_complexity}). It is this polynomial form, not the bare positivity \textup{(H1)}, that is forced to fail, the gap closing faster than any inverse polynomial, on graph families for which MIS is NP-hard. The conditions are independent: the window \eqref{energy_window} can hold while the interior gap is exponentially small (or closes), and a large interior gap does not by itself place $|0\rangle^{\otimes n}$ and $\Hmis$ at the endpoints.
	\end{remark}
	
	\begin{proof}[Proof of Theorem \ref{thm_approximate_controllability_MIS}]
		\textit{Step 0. Initial setting.}
		
		Fix a tolerance $\varepsilon>0$ for the relaxed goal \eqref{relaxed_goal}. We remind the definition of $H_{\gamma}$
		\begin{equation}\label{instantaneous_Pasqal_II}
			H_{\gamma}(s) \;\coloneqq\; \frac{\hbar\,\Omega(s)}{2}\sum_{k=1}^n X_k \;-\; \hbar\,\delta(s)\sum_{k=1}^n \hat n_k \;+\; H_0.
		\end{equation}
		Because the initial state $\phi_B=|0\rangle^{\otimes n}$ is already the ground state of $H_\gamma(0)$ (Lemma~\ref{lemma_endpoints}(a)), no preliminary reach phase is required, and the argument reduces to the adiabatic anneal (Step~1).
		
		\textit{Step 1. From $\phi_B=|0\rangle^{\otimes n}$ into $\Hmis$: approximate controllability via the Quantum Adiabatic Theorem.}
		
		We apply the adiabatic schedule $\gamma(t)\coloneqq (\Omega(t), \delta(t))$, with
		\begin{equation}\label{}
			\Omega(t) = \Omega_{\max} 4s(1-s),\qquad \delta(t) = \delta_0(2s-1),\qquad s=t/T_{\mathrm{anneal}}\in[0,1].
		\end{equation}
		Here and below $\dot\gamma$ denotes the derivative of the schedule with respect to the rescaled time $s\in[0,1]$, so that $\|\dot\gamma\|_\infty=\sup_{s\in[0,1]}\|\tfrac{d}{ds}\gamma(s)\|\leq 4\Omega_{\max} + 2\delta_0$ is independent of $n$; the corresponding bound on the physical Hamiltonian velocity is $\sup_{s\in [0, 1]}\|\tfrac{d}{ds}H_\gamma(s)\|\le L_H\,\|\dot\gamma\|_\infty$, with $L_H\coloneqq\sup\|\nabla_{(\Omega,\delta)}H_\gamma\|$ the (Lipschitz) constant of the map $(\Omega,\delta)\mapsto H_\gamma$, absorbed into $C_{\mathrm{ad}}$ below.
		
		By Lemma~\ref{lemma_endpoints}(a), $\phi_B=|0\rangle^{\otimes n}$ is the \emph{simple} ground state of the Pasqal Hamiltonian $H_\gamma(0)$, with gap $\hbar\delta_0>0$. Since $\phi_B$ is $\mathrm{Aut}(G,w)$-invariant and the Pasqal Hamiltonian commutes with every $P_\sigma$ (Proposition~\ref{prop_pasqal_no_full_controllability}), the whole evolution remains in the invariant sector $\mathcal{H}^{\mathrm{Aut}(G,w)}$, where the adiabatic theorem may be applied. Following the adiabatic-evolution programme of \cite{farhi2000quantum} and applying the Quantum Adiabatic Theorem in its explicit-gap form (see, e.g., \cite{reichardt2004quantum,aharonov2008adiabatic}, the explicit-gap estimate of \cite{jansen2007bounds}, or the lecture notes \cite{Childs2008Lecture18}) to \eqref{instantaneous_Pasqal_II} for the simple ground state $\phi_B=|0\rangle^{\otimes n}$ at $s=0$, specialized to the admissible value $\delta=1$ of the free parameter in \cite[Theorem~2.1]{aharonov2008adiabatic},\footnote{For any \emph{fixed} $\delta>0$ (the free parameter of the adiabatic theorem, not to be confused with the detuning $\delta(t)$), \cite[Theorem~2.1]{aharonov2008adiabatic} guarantees $\varepsilon$-closeness to the instantaneous ground state once $T_{\mathrm{anneal}}\ge\Omega\big(\|\dot\gamma\|_\infty^{1+\delta}\big/[\varepsilon^{\delta}(\Delta_\gamma^{\min})^{2+\delta}]\big)$; the value $\delta=0$ is inadmissible there, the hidden constant diverging as $\delta\downarrow 0$. The choice $\delta=1$ is the unique one rendering the achievable error \emph{linear} in $1/T_{\mathrm{anneal}}$, and yields the gap exponent $2+\delta=3$ together with the squared driving norm $\|\dot\gamma\|_\infty^{2}$. The same $(\Delta_\gamma^{\min})^{-3}$ is the dominant gap dependence in the explicit estimate of \cite{jansen2007bounds}, whose leading term scales as $\|\dot\gamma\|_\infty^{2}/(\Delta_\gamma^{\min})^{3}$ (with subleading curvature terms $\propto 1/(\Delta_\gamma^{\min})^{2}$). For the bell-shaped, fixed-endpoint schedule used here, $\|\dot\gamma\|_\infty$ and $\|\ddot\gamma\|_\infty$ are $n$-independent constants, all absorbed into $C_{\mathrm{ad}}$.} gives
		\begin{equation}\label{adiabatic_bound_Pasqal}
			\big\|\psi(T_{\mathrm{anneal}}) - e^{i\varphi}\,\psi_1\big\| \;\leq\; \frac{C_{\mathrm{ad}}\,\|\dot\gamma\|_\infty^{2}}{(\Delta_\gamma^{\min})^{3}\,T_{\mathrm{anneal}}},
		\end{equation}
		where $C_{\mathrm{ad}}>0$ is a constant and $\psi_1$ is a ground state of $H_\gamma(1)$. The right-hand side is well posed, because $\Delta_\gamma^{\min}>0$ by \textup{(H1)}. By Lemma~\ref{lemma_endpoints}(b), under the window \eqref{energy_window} the ground manifold of $H_\gamma(1)$ is contained in $\Hmis$; hence
		\begin{equation}\label{gs_in_Hmis}
			\psi_1\in\Hmis .
		\end{equation}
		No identification of this vector with $\psi_{\mathrm{MIS}}$ is required, and any intra-band tail splitting at $s=1$ (Remark~\ref{rem:splitting}) is immaterial: it does not move the limit out of $\Hmis$.
		
		Hence, choosing $T_{\mathrm{anneal}}$ large enough that the right-hand side of \eqref{adiabatic_bound_Pasqal} is below $\varepsilon$, drives the state to within $\varepsilon$ of $\Hmis$:
		\begin{equation}\label{approximate_controllability_Pasqal}
			\mathrm{dist}\big(\psi(T_{\mathrm{anneal}}),\,\Hmis\big)\ \le\ \inf_{\theta\in\mathbb{R}}\big\|\psi(T_{\mathrm{anneal}}) - e^{i\theta}\psi_1\big\| \;\le\; \frac{C_{\mathrm{ad}}\,\|\dot\gamma\|_\infty^{2}}{(\Delta_\gamma^{\min})^{3}\,T_{\mathrm{anneal}}} \;\le\; \varepsilon,
		\end{equation}
		and in particular $\psi_1\in\Hmis\cap\overline{\mathcal{R}(\phi_B)}$ \textup{(}letting $T_{\mathrm{anneal}}\uparrow + \infty$: each $\psi(T_{\mathrm{anneal}})\in\mathcal{R}(\phi_B)$, being produced by the admissible anneal sweep\textup{)}. This finishes the proof.
	\end{proof}
	
	\begin{corollary}[Unique MIS]\label{cor_unique_MIS}
		If $|\mathrm{MIS}(G)|=1$, say $\mathrm{MIS}(G)=\{\mathbf{x}^{*}\}$, then $\Hmis=\spn_{\mathbb{C}}\{|\mathbf{x}^{*}\rangle\}$ is one-dimensional and $\psi_{\mathrm{MIS}}=|\mathbf{x}^{*}\rangle$. Under hypotheses \textup{(H0)}--\textup{(H1)} the relaxed goal \eqref{relaxed_goal} then forces $\psi(T^{*})=e^{i\theta}|\mathbf{x}^{*}\rangle$ for some phase $\theta$: the Pasqal equation is driven approximately from $\phi_B$ to the unique optimal bitstring (up to a global phase).
	\end{corollary}
	
	\begin{remark}[Time complexity]\label{remark_time_complexity}
		With the annealing initialisation $\phi_B=|0\rangle^{\otimes n}$ there is no reach phase to account for. For the anneal (Step~1), the adiabatic bound \eqref{adiabatic_bound_Pasqal} gives the sufficient time scale
		\begin{equation}\label{adiabatic_time_Pasqal}
			T_{\mathrm{anneal}}(\varepsilon) = \frac{C_{\mathrm{ad}}\,\|\dot\gamma\|_\infty^{2}}{(\Delta_\gamma^{\min})^{3}\,\varepsilon}.
		\end{equation}
		It is only at this point that the \emph{polynomial} form of the spectral-gap hypothesis is needed. The linear dependence on $1/\varepsilon$ is approximately the $\delta=1$ specialization of \cite[Theorem~2.1]{aharonov2008adiabatic} (a different admissible $\delta$ would replace it by $\varepsilon^{-1/\delta}$); it is the choice consistent with the $1/T_{\mathrm{anneal}}$ form of \eqref{adiabatic_bound_Pasqal}. Under Assumption~\ref{assumption_gap_Pasqal}, $\Delta_\gamma^{\min}\geq c/n^q$, whence $T_{\mathrm{anneal}}(\varepsilon)=O(n^{3q}/\varepsilon)$: polynomial in $n$ for any fixed $\varepsilon>0$.
		% (Relative to the non-rigorous folklore estimate $(\Delta_\gamma^{\min})^{-2}$, which would correspond to the inadmissible $\delta=0$, the rigorous exponent $3$ raises the degree from $2q$ to $3q$; since in the regime of interest $\Delta_\gamma^{\min}\le 1$ one has $(\Delta_\gamma^{\min})^{-3}\ge(\Delta_\gamma^{\min})^{-2}$, this is the conservative direction for an upper bound, and the conclusion ``polynomial in $n$'' is unaffected.)
	\end{remark}

	\section{Is there quantum advantage for MIS?}\label{sec:Is there quantum advantage?}
	
	Definition \ref{quantum_advantage_definition} of Quantum Advantage (QA) was formulated for operator controllability on $SU(2^n)$, which, by Proposition \ref{prop_pasqal_no_full_controllability}, does \emph{not} hold for the global, symmetric Pasqal controls. A meaningful notion of QA for the MIS problem on neutral-atom hardware therefore calls for a state-oriented adaptation of Definition \ref{quantum_advantage_definition}, organized around three changes:
	\begin{itemize}
		\item[\textup{(i)}] the \emph{target} is the ground-state manifold $\Hmis$ of \eqref{mis_subspace_def} rather than a single unitary $\Gamma\in SU(2^n)$;
		\item[\textup{(ii)}] \emph{success} is the exact preparation of a state in $\Hmis$ (equivalently, the sampling of a maximum independent set with certainty);
		\item[\textup{(iii)}] the polynomial-time requirement is, intrinsically, a statement about a \emph{family} of instances indexed by the number of vertices $n$, not about a single graph.
	\end{itemize}

	For a normalised state $\psi\in\mathcal{H}$, the probability that a computational-basis measurement returns a maximum independent set of $G$ is, by the Born rule \eqref{measurement_MIS_certainty},
	\begin{equation}\label{mis_success_probability}
		\mathbb{P}_{\mathrm{MIS}}(\psi)\;=\;\bigl\|P_{\Hmis}\,\psi\bigr\|^{2}\;=\;\sum_{\bx\in\MIS(G)}\bigl|\langle\bx|\psi\rangle\bigr|^{2},
	\end{equation}
	and $\mathbb{P}_{\mathrm{MIS}}(\psi)=1$ \emph{if and only if} $\psi\in\Hmis$. Accordingly we take as success \emph{event} the exact membership $\psi(T)\in\Hmis$, equivalently, the sampling of a maximum independent set with certainty, and define the \textbf{minimal MIS-preparation time} of $G$ by
	\begin{equation}\label{def_min_time_MIS}
		T_{\mathrm{MIS}}(G)\;\defeq\;\inf\Bigl\{T>0 \;\Big|\; \exists\,(\Omega,\delta)\in\mathscr{U}_{\mathrm{ad}}^{T}\ \text{such that}\ \psi(T)\in\Hmis\Bigr\}\footnote{we adopt the convention $\inf\varnothing = + \infty$},
	\end{equation}
	where $\psi(\cdot)$ solves the Pasqal equation \eqref{pasqal_equation_explicit} with initial condition $\psi(0)=\phi_B=|0\rangle^{\otimes n}$ and controls $(\Omega,\delta)$ in the admissible set $\mathscr{U}_{\mathrm{ad}}^{T}$ \eqref{admissible_controls_pasqal}. This is the state-oriented analog of the operator minimal time \eqref{def_min_time}, the target operator $\Gamma$ being replaced by the target subspace $\Hmis$.

	\begin{definition}[Quantum Advantage for a family of MIS instances]\label{def_QA_specific_MIS}
		Let $\{G_n\}_{n\in\mathbb{N}}$ be a family of unit-disk graphs, with $G_n=(V_n,E_n)$ on $|V_n|=n$ vertices for which the best known classical algorithm requires a computing time\footnote{As we anticipated, all along this manuscript, we take the second (s) as unit of measurement of time.} exponential in the number of qubits $n$. Assume each $G_n$ is realized as in \eqref{rydberg_radius} by an atom layout $\{r_i\}_{i=1}^n$ and thereby defining a Pasqal Hamiltonian \eqref{neutral_atoms_hamiltonian} and the associated bilinear Schr\"{o}dinger equation \eqref{pasqal_equation_explicit} on $\mathcal{H}=\bigotimes_{i=1}^n\mathbb{C}^2$. We say that there is (exponential) \emph{Quantum Advantage for the MIS problem along the family $\{G_n\}_n$} if there exist a constant $C>0$ and an integer $p\in\mathbb{N}$ \emph{(both independent of $n$)} such that, for every $n\in\mathbb{N}$, there exist a time horizon $T_n\leq C\,n^{p}$ and admissible controls $(\Omega_n,\delta_n)\in\mathscr{U}_{\mathrm{ad}}^{T_n}$ for which the solution $\psi_n$ of \eqref{pasqal_equation_explicit} from $\psi_n(0)=\phi_B=|0\rangle^{\otimes n}$ satisfies
		\begin{equation}\label{label_QA_MIS}
			\psi_n(T_n)\in\Hmis ,
		\end{equation}
		with $\Hmis$ the MIS subspace \eqref{mis_subspace_def}. Equivalently, $T_{\mathrm{MIS}}(G_n)\leq C\,n^{p}$ for every $n$ (cf.\ \eqref{def_min_time_MIS}). By \eqref{measurement_MIS_certainty}, condition \eqref{label_QA_MIS} is equivalent to $\mathbb{P}_{\mathrm{MIS}}(\psi_n(T_n))=1$: a computational-basis measurement of $\psi_n(T_n)$ returns a maximum independent set of $G_n$ with certainty.
	\end{definition}

	\begin{definition}[Quantum Advantage for all MIS]\label{def_QA_MIS}
		We say that there is (exponential) \emph{Quantum Advantage for the MIS problem on the class of all unit-disk graphs} if there exist a constant $C>0$ and an integer $p\in\mathbb{N}$ \emph{(both independent of $n$ and of $G$)} such that, for every $n\in\mathbb{N}$ and every unit-disk graph $G$ on $n$ vertices,
		\begin{equation}\label{label_QA_all_MIS}
			T_{\mathrm{MIS}}(G)\;\leq\;C\,n^{p}.
		\end{equation}
		Equivalently, Definition \ref{def_QA_specific_MIS} holds, with the \emph{same} constants $C,p$, along \emph{every} family of unit-disk graphs.
	\end{definition}
	
	\begin{remark}[Polynomial Quantum Advantage for MIS]\label{polynomial_Quantum_Advanatage_MIS}
		The above definitions might be modified in two directions
		\begin{itemize}
			\item[a)] exact reachability requirement \eqref{label_QA_MIS} might be replaced with an approximate one
			\begin{equation}
				\mathrm{dist}\big(\psi_n(T_n),\,\Hmis\big)\ \le\ \varepsilon;
			\end{equation}
			\item[b)] as in Remark \ref{quantum_advantage_definition_remark}, analogous definitions for polynomial Quantum Advantage might be derived.
		\end{itemize}
	\end{remark}

	Under the widely believed complexity-theoretic assumption $\mathsf{NP}\not\subseteq \mathsf{BQP}$, we conjecture one cannot hope for a \emph{worst-case, exact} QA in the sense of Definition \ref{def_QA_MIS}. Three relaxations, each of independent interest from the control-theoretic viewpoint, are nevertheless meaningful and compatible with known complexity.
	
	\begin{enumerate}
		\item[(R1)] \textbf{Average-case QA.} Replace the worst-case requirement over all unit-disk graphs in Definition \ref{def_QA_MIS} by an expectation over a \emph{random} ensemble. Specifically, let $\mathcal{G}(n,\rho)$ denote the random geometric graph model in which $n$ points are drawn uniformly in $[0,1]^2$ and two vertices are connected if and only if their Euclidean distance is at most $\rho = \rho(n)$. Average-case QA holds if
		\begin{equation}\label{average_case_QA}
			\mathbb{E}_{G\sim\mathcal{G}(n,\rho)}\big[\mathbb{P}[\mathbf{x}\mbox{ encodes an MIS of }G]\big] \geq \eta
		\end{equation}
		for some $\eta>0$ independent of $n$, with a time horizon polynomial in $n$. Numerical evidence in \cite{ebadi2022quantum,tibaldi2025analog} suggests that, on classes of random graphs at the hard-instance density, the Pasqal/QuEra platforms do sample near-optimal independent sets in time polynomial in $n$ with non-trivial probability.
		
		\item[(R2)] \textbf{Approximate QA.} Relax ``Maximum Independent Set'' in \eqref{label_QA_MIS} to an approximation guarantee: for some fixed $\varepsilon\in(0,1)$ (independent of $n$ and $G$), require
		\begin{equation}\label{approximate_QA}
			\mathbb{P}\big[\,\mathbf{x}\mbox{ encodes an independent set of $G$ of cardinality }\geq (1-\varepsilon)\,\alpha(G)\,\big]\geq \eta.
		\end{equation}
		This is the natural formulation in the QAOA literature, yielding an approximation-ratio guarantee rather than exact optimality, as in Theorem \ref{thm_approximate_controllability_MIS}. The computational complexity landscape changes significantly: whereas MIS is NP-hard to approximate within a factor $n^{1-\varepsilon}$ in general \cite{zuckerman2007linear}, on unit-disk graphs a Polynomial Time Approximation Scheme (PTAS) exists classically \cite{hunt1998nc}, so the relevant question becomes whether the quantum protocol achieves a given approximation ratio \emph{faster} than classical algorithms.
		
		\item[(R3)] \textbf{Turnpike QA.} Replace the time bound $T_G\leq C n^p$ by a polynomial upper bound on the minimal time of the \emph{integral-tracking} optimal control problem \eqref{optimal_control_problem_qaoa_tracking}, in which the running cost $\langle\psi(t)|H_C|\psi(t)\rangle$ enforces adiabatic following of the instantaneous ground state along the whole interval $[0,T]$.
	\end{enumerate}
	
	The three relaxations (R1)--(R3) outline three research programmes, along which the control-theoretic machinery of chapter \ref{chapter:A surrogate problem} (the surrogate commutator functional, Proposition \ref{prop_7}) can be applied to test, necessarily or sufficiently, whether Quantum Advantage for MIS holds on a given class of instances.
	
	\begin{remark}[Structural contrast with the QFT case]\label{remark_structural_contrast}
		The analysis of Quantum Advantage for MIS on neutral-atom hardware differs structurally from the QFT analysis of chapter \ref{chapter:QA_QFT} in three ways:
		\begin{enumerate}
			\item[(a)] \textbf{Controllability.} For the QFT, the superconducting platform enjoys full operator controllability on $SU(2^n)$ (Proposition \ref{prop_1}), and QA is established by bounding the minimal time $T_{\mbox{\tiny{min}}}(\Gamma_{\mbox{\tiny{QFT}}})$ to implement a specific unitary $\Gamma_{\mbox{\tiny{QFT}}}$. For MIS, operator controllability fails (Proposition \ref{prop_pasqal_no_full_controllability}), and QA is formulated as a state-preparation problem with exact target $\Hmis$, the success event being the membership $\psi(T)\in\Hmis$ \eqref{label_QA_MIS}.
			\item[(b)] \textbf{Complexity barrier.} For the QFT, the polynomial bound $T_{\mbox{\tiny{min}}} \leq \tau n^2$ is unconditional (Theorem \ref{theorem_QA_QFT}). For MIS, worst-case QA we conjecture is obstructed by NP-hardness, so only the relaxed notions (R1)--(R3) are viable.
			\item[(c)] \textbf{Problem encoding.} For the QFT, the target operator $\Gamma_{\mbox{\tiny{QFT}}}$ is independent of any problem instance: it is a fixed unitary on $2^n$ dimensions. For MIS, the target depends on the graph $G$, which is encoded in the atom layout $\{r_i\}$ and hence in the drift Hamiltonian $H_0$; the dependence of the minimal time on $G$ is the source of the worst-case/average-case distinction.
		\end{enumerate}
	\end{remark}
	
%	\begin{remark}[Connection with the surrogate problem]\label{remark_surrogate_MIS}
%		Since full operator controllability does not hold, Proposition \ref{prop_7} does not apply directly. As discussed in section \ref{sec:surrogate_application}, a natural analog of the surrogate bound \eqref{QA_turnpike} in the state-oriented setting replaces the commutator $[U(t)\Gamma^*,H_0]$ by the \emph{residual energy}
%		\begin{equation}\label{residual_energy_MIS}
%			\langle \psi(t)|H_C|\psi(t)\rangle - E_{\min}(H_C),
%		\end{equation}
%		and the minimal operator-controllability time $T_{\mbox{\tiny{min}}}$ by the adiabatic time $T_{\mathrm{ad}}$ of \eqref{adiabatic_time_scale}. Establishing a rigorous state-targeted analog of Proposition \ref{prop_7} for the MIS setting is an important open problem (see also chapter \ref{chapter:Open problems}).
%	\end{remark}
	
	\section{The classical-quantum hybrid loop}\label{sec:How QAOA is solved}
	
	The optimal control problem \eqref{optimal_control_problem_qaoa}-\eqref{QAOA_state_equation} is solved by interlacing classical and quantum computation, as follows.
	\begin{itemize}
		\item On the Quantum Computer (QC) one runs functional evaluation: the Schr\"{o}dinger equation \eqref{QAOA_state_equation} is physically \emph{observed} (not classically simulated). In practice, given a candidate control $u(\cdot)$, the Quantum Computer prepares $\psi(T)$ and a finite number of repeated measurements in the computational basis estimate $J_T(u) = \langle \psi(T) | H_C | \psi(T) \rangle$ via the empirical average of $f(\mathbf{x})$ over the sampled bitstrings.
		\item On the classical computer, one performs the optimization step (control update); for instance, by gradient descent with finite-difference gradients, by Bayesian optimization, or by trust-region methods.
	\end{itemize}
	
	For instance, in \cite{tibaldi2025analog}, QAOA on MIS instances is run
	\begin{itemize}
		\item on the quantum hardware \href{https://portal.pasqal.cloud/devices/FRESNEL}{Pasqal FRESNEL};
		\item using a Bayesian optimizer as the classical outer loop.
	\end{itemize}
	The Hamiltonian, of the form \eqref{neutral_atoms_hamiltonian}, is written in \cite[equation (A.1) at page 12]{tibaldi2025analog}, with controls $\Omega_{\theta}(t)$ and $\delta_{\theta}(t)$.
	
	\section{Stabilization via integral tracking functionals}\label{sec:Advanced State Stabilization via Integral Tracking Functionals}
	
	The standard optimal control formulation \eqref{optimal_control_problem_qaoa}-\eqref{QAOA_state_equation} is inherently fragile, since the cost functional $J_T(u)$ is evaluated only at the terminal time $T$. If the system undergoes a diabatic transition at an intermediate time $t < T$, due, e.g., to thermal noise or to a rapidly closing energy gap, the state $\psi(t)$ may diverge from the instantaneous ground-state manifold of $H_{\gamma}(t)$. Because the optimizer only receives a penalty at the terminal time, the gradient signal used to correct intermediate errors is highly diffuse.
	
	To resolve this fragility and dynamically stabilize the system, advanced control theory modifies the QAOA objective by introducing an integral tracking functional. Instead of measuring only terminal performance, the target objective tracks the expected energy of $H_C$ continuously throughout the entire evolution:
	\begin{equation}\label{optimal_control_problem_qaoa_tracking}
		\min_{u \in L^2(0,T;\mathbb{R})} J_T(u) = \frac{\lambda_{\mbox{\tiny{ctrl}}}}{2} \int_0^T |u(t)|^2 dt + \frac{1}{2} \int_0^T \langle \psi(t) | H_C | \psi(t) \rangle dt,
	\end{equation}
	with state equation \eqref{QAOA_state_equation} and $\lambda_{\mbox{\tiny{ctrl}}} > 0$ a regularization parameter (not to be confused with the QUBO penalty $\lambda$ of \eqref{mis_qubo}). The Quantum Adiabatic Theorem implies that the optimal solution of \eqref{optimal_control_problem_qaoa_tracking} drives the system toward a ground state of $H_C$, encoding a Maximum Independent Set of $G$ (Lemma \ref{lemma_endpoints}).
	
	The mathematical lineage of this approach is rooted in classical PDE stabilization and turnpike theory \cite{porretta2013long,PZ2,trelat2015turnpike,trelat2018steady,pighin2020turnpike,trelat2024control}: the running cost $\langle \psi | H_C | \psi \rangle$ plays the role of a tracking term penalizing deviations from the target ground-state manifold over the whole horizon $[0,T]$, while the $L^2$ Tikhonov term ensures coercivity of the cost in the control variable.

	\chapter{Open problems}\label{chapter:Open problems}
	
	Let us present some interesting open problems, which, as long as we know, were not investigated in the literature so far.
	
	\begin{enumerate}
		\item \textbf{Continuous-time direct proof of Quantum Advantage for QFT.} Direct proof of the Quantum Advantage for QFT, not passing by discrete-time gates composition; rather employing explicit continuous-time controls. Of course, this might be relaxed to an approximate controllability result, where the QFT operator is reached, up to an error $\varepsilon$.
		\item \textbf{Return time estimate.}  Estimate the minimal time $T>0$ to solve the return problem
		\begin{equation}\label{label_return_problem}
			\begin{cases}
				i\frac{d}{dt}U(t) = \big(H_0 + \sum_{j=1}^{m}u_j(t)H_j\big)U(t),& t\in (0,T)\\
				U(0)=I,\\
				U(T)=I.\\
			\end{cases}
		\end{equation}
		\item \textbf{Stabilization along operators orbits.}  In the framework of section \ref{chapter:Mathematical framework}, does there exists a control $u\in \mathscr{U}^T$, such that
		\begin{equation}
			\int_{0}^{+\infty}\left\|U(t)-\exp(-iH_0t)\Gamma\right\|^2dt<+\infty
		\end{equation}
		? This might be related to \cite{beauchard2024small} and the possibility of moving in the manifold $M=\left\{\exp(-iH_0t) \ | \ t\in \mathbb{R}\right\}$ in small time (smaller than the period of $H_0$ if $H_0$ is periodic).
		\item \textbf{Rigorous proof of uniformly bounded elementary gate time.} In theorem \ref{theorem_QA_QFT}, we assumed uniformly bounded elementary gate time (definition \ref{def_elementary_gate_time}). A rigorous proof would require analyzing the minimal time for elementary gates (Hadamard, controlled-rotation, CNOT) in the presence of the drift Hamiltonian $H_0$, showing that $T_{\mbox{\tiny{min}}}(G) \leq \tau$ with $\tau$ independent of $n$. This likely involves a careful analysis in the interaction picture (rotating frame), exploiting the local structure of single- and two-qubit control Hamiltonians. A sub-Riemannian geometry viewpoint on $SU(N)$ might be useful \cite{agrachev2017note,boscain2021introduction}.
		\item \textbf{Converse of the necessary condition for QA.} Proposition \ref{prop_7} provides a necessary condition for QA in terms of the value function $V(T,\Gamma)$. A natural question is whether the converse holds, namely: \textit{if $V(T,\Gamma) \leq Cn^p$ for all $T > 0$, does Quantum Advantage hold?} This would provide a full characterization of QA in terms of the surrogate problem.
		\item \textbf{Greedy approach to minimize quantum hardware use in Quantum Computing.} On the one hand, use quantum hardware is expensive. On the other hand, several real world problems (e.g., optimization problems) depends on continuously changing configurations. This would oblige to use quantum hardware in real-time, which would lead to huge licensing cost. Greedy approach \cite{barron2008approximation,lazar2016greedy,hernandez2017greedy} might be employed to
		\begin{itemize}
			\item[\textbf{Step 1.}] Identify the most representative configurations.
			\item[\textbf{Step 2.}] Compute solutions, by quantum hardware, \textit{offline} only for the identified most representative configurations.
			\item[\textbf{Step 3.}] Compute solutions, for any configuration, \textit{online} by combining pre-computed solutions.
		\end{itemize}
		See also the work \cite{villar2026transfer} on the transfer of knowledge in quantum algorithms.
		\item \textbf{New ans\"{a}tze for quantum-powered discrete optimization.} In the context of quantum discrete optimization (section \ref{chapter:Quantum Discrete Optimization}), research new ans\"{a}tze (controlled hamiltonians, fitting with given quantum hardware), maximizing controllability properties of the associated controlled {S}chr\"{o}dinger equation to rapidly converge to the solution of the discrete optimization problem. Note that
		\begin{equation}
			\mbox{QAOAs }\subset\mbox{ VQAs},
		\end{equation}
		where
		\begin{itemize}
			\item QAOAs stands for Quantum Approximate Optimization Algorithms;
			\item VQAs stands Variational Quantum Algorithms, like the Variational Quantum Eigensolver.
		\end{itemize}
		From a terminological viewpoint, algorithms employing an ansatz, differing from the standard QAOA, would be named VQAs. New ans\"{a}tze might employ controls to circumvent exponentially small gaps between first and second eigenvalues ofn the Hamiltonians.
	\end{enumerate}
	
	%The coupuling map $\mathcal{E}$ is bipartite \cite{tindall2024confinement,pelofske2024scaling}.
	
	% IDEA for software
	%Input: special unitary matrix.
	%Output: Yes [there exists Quantum Advantage (QA)]. No [there is no Quantum Advantage (QA)].
	
	\section*{Acronyms}
\addcontentsline{toc}{section}{List of Acronyms}
\markboth{LIST OF ACRONYMS}{LIST OF ACRONYMS}

For the reader's convenience we collect here, in alphabetical order, all
the acronyms employed throughout these notes, together with a short
gloss tying each one to the control-theoretic framework developed in the
sequel. The complexity-theoretic classes are mentioned below with their
meaning (see, e.g., \cite{nielsen2010quantum} and the references
therein).

\begin{description}

  \item[AQO] \emph{Adiabatic Quantum Optimization}: the analog paradigm
  in which the register is driven by an Hamiltonian, slowly varying from an initial Hamiltonian to a problem Hamiltonian, to solve a discrete optimization problem.

  \item[BQP] \emph{Bounded-error Quantum Polynomial time}: the class of
  decision problems solvable by a quantum computer in time polynomial in
  the input size with error probability at most $1/3$.  The working
  hypothesis $\mathsf{NP}\not\subseteq\mathsf{BQP}$ underlies the
  worst-case obstruction to Quantum Advantage for the MIS problem.

  \item[CNOT] \emph{Controlled-NOT} gate: the two-qubit entangling gate
  which, together with the single-qubit rotations, generates a universal
  gate set.

  \item[DFT] \emph{Discrete Fourier Transform} on $N=2^n$ points;
  classically computed in $O(N\log N)$ operations by the FFT.

  \item[DMCV] \emph{Direct Methods in the Calculus of Variations}
  \cite{dacorogna2007direct}: used to establish existence of minimizers
  of the surrogate functional $J_T$ and attainment of the minimal time
  $T_{\mbox{\tiny{min}}}$ under the control constraint $|u(t)|\le M$.

  \item[FFT] \emph{Fast Fourier Transform}: the classical $O(N\log N)$
  algorithm computing the DFT, against which the $O(n^2)$ cost of the QFT
  is to be compared.

  \item[IS] \emph{Independent Set}: a subset of pairwise non-adjacent
  vertices of a graph $G=(V,E)$; the associated subspace of the Hilbert
  space is denoted $\His$.

  \item[MIS] \emph{Maximum Independent Set}: an independent set of
  maximum cardinality $\alpha(G)$; the family of maximizers is
  $\operatorname{MIS}(G)$ and the corresponding subspace is $\Hmis$.
  Computing $\alpha(G)$ is NP-hard \cite{karp1972reducibility}, and the
  problem admits a native neutral-atom encoding via the Rydberg
  blockade.

  \item[NISQ] \emph{Noisy Intermediate-Scale Quantum} (era / devices):
  present-day hardware with a limited number of qubits and no full
  fault tolerance, the regime in which QAOA operates.

  \item[NP] \emph{Nondeterministic Polynomial time}; \textbf{NP-hard}
  and \textbf{NP-complete} denote the corresponding hardness notions.
  Deciding a maximum independent set is NP-hard in the classical sense.

  \item[P] \emph{(deterministic) Polynomial time}: the class of decision
  problems solvable by a classical deterministic machine in polynomial
  time.

  \item[PDE] \emph{Partial Differential Equation}; the controlled
  Schr\"{o}dinger equation is a PDE, in case the quantum state space $\mathcal{H}$ has dimension infinity.

  \item[PTAS] \emph{Polynomial-Time Approximation Scheme}: on unit-disk
  graphs the MIS problem admits a classical PTAS \cite{hunt1998nc},
  which tempers the prospects for a worst-case Quantum Advantage.

  \item[QA] \emph{Quantum Advantage}: in the sense of Definition
  \ref{quantum_advantage_definition}, the conjunction of operator
  controllability ($\mathscr{U}_{\mbox{\tiny{ad}}}\neq\varnothing$) with
  a polynomial-in-$n$ bound $T_{\mbox{\tiny{min}}}\le Cn^p$ on the
  minimal control time. For discrete optimization problems, like Maximum Independent Set (MIS), the definition of Quantum Advantage is state-oriented (see Definition \ref{def_QA_specific_MIS} and Definition \ref{def_QA_MIS}).

  \item[QAOA] \emph{Quantum Approximate Optimization Algorithm}
  \cite{farhi2014quantum}: here recast as a continuous-time bilinear
  optimal control problem.

  \item[QC] \emph{Quantum Computing} / \emph{Quantum Computer}.

  \item[QFT] \emph{Quantum Fourier Transform}: the unitary
  $\Gamma_{\mbox{\tiny{QFT}}}\in SU(2^n)$ implementing the DFT on
  quantum state space, realized in $O(n^2)$ elementary gates and
  hence in $O(n^2)$ minimal time on digital hardware.

  \item[QKD] \emph{Quantum Key Distribution}: a class of quantum
  communication protocols (e.g.\ Ekert91 \cite{ekert1991quantum} and
  BBM92 \cite{bennett1992quantum}) resting on superposition and
  entanglement.

  \item[QUBO] \emph{Quadratic Unconstrained Binary Optimization}: the
  penalized binary reformulation of the MIS problem and several discrete optimization problems.

  \item[$SU(N)$] \emph{Special Unitary group} of degree $N$ (the
  determinant-one unitary matrices), with Lie algebra $\mathfrak{su}(N)$
  of traceless skew-Hermitian matrices; operator controllability is
  phrased as reachability of any target $\Gamma\in SU(N)$ and certified
  through the Lie-algebraic rank condition on $\mathfrak{su}(N)$.

  \item[UDG(s)] \emph{Unit-Disk Graph(s)}: graphs $G=(V,E)$ in which two
  vertices are adjacent precisely when their Euclidean distance is at
  most the Rydberg-blockade radius $R_b$; these are the graphs natively
  realized by neutral-atom registers.

\end{description}

	\appendix
	
	\chapter{The case of steady controls}\label{chapter:The case of steady controls}
	
	Let us consider the case of steady controls. This might give insight even for time evolution controls. Indeed, the action of piece-wise constant controls can be seen as composition  of the actions of several steady controls.
	\begin{remark}[Steady controls]
		Suppose the controls are steady, i.e.
		\begin{equation}
			u_j\equiv\bar{u}_j\in\mathbb{R},\hspace{0.06 cm}\forall j=1,\dots,m.
		\end{equation}
		We get then the {S}chr\"{o}dinger equation with constant controls
		\begin{equation}\label{ger equation_constant_controls}
			i\frac{d}{dt}\mathbf{\psi}(t) = H\mathbf{\psi}(t),\hspace{0.06 cm}t\in(0,T),
		\end{equation}
		where
		\begin{equation}
			H\defeq H_0 + \sum_{j=1}^{m}H_j\bar{u}_j.
		\end{equation}
		Then, we have the explicit formula for solution to \eqref{ger equation_constant_controls}
		\begin{equation}
			\mathbf{\psi}(t)=\exp(-iHt)\mathbf{\psi}_0.
		\end{equation}
		Hence, working only in the class of steady controls, the computation problem \eqref{def_1} can be reduced to finding the solution in the unknowns $\bar{u}_j$ the linear system
		\begin{equation}
			H_0+\sum_{j=1}^{m}H_j\bar{u}_j=\frac{i}{T}\log_{e}(U),
		\end{equation}
		$U$ being the associated matrix to $\Gamma$ in the canonical basis
		% reference
		of $\mathcal{H}$. 
	\end{remark}
	
	\chapter{An example of drift Hamiltonian $H_0$}\label{chapter:drift_hamiltonian}
	Inspired from \cite[equation (15)]{hou2014realization}, let us give an example of drift Hamiltonian.
	
	First, out of the directed graph $\mathcal{E}$ defined in subsection \ref{subchapter:couplingmap}, we define the associated undirected graph
	\begin{equation}\label{indirect_graph_definition}
		\mathcal{E}_{\mbox{\tiny{undirected}}}\defeq \left\{(k,l) \ | \ (k,l)\in \mathcal{E}\right\}\bigcup \left\{(l,k) \ | \ (k,l)\in \mathcal{E}\right\}.
	\end{equation}
	
	The drift Hamiltonian consists of qubit self-energies and static qubit-qubit couplings:
	\begin{equation}\label{H_0}
		H_0 = \sum_{k=1}^{127} \frac{\omega_k}{2} Z_k + \sum_{(k,l) \in \mathcal{E}_{\mbox{\tiny{undirected}}}} J_{kl} Z_k Z_l,
	\end{equation}
	where
	\begin{itemize}
		\item $\omega_k>0$ is the transition frequency of qubit $k$;
		\item $Z_k$ is the Pauli-$Z$ operator acting on qubit $k$;
		\item $J_{kl}>0$ is a scalar quantifying the coupling strength between qubits $k$ and $l$;
		\item the addendum $\sum_{(k,l) \in \mathcal{E}_{\mbox{\tiny{undirected}}}} J_{kl} Z_k Z_l$ models the Ising interaction;
		\item $X_k$, $Y_k$ are Pauli-$X$ and Pauli-$Y$ operators acting on qubit $k$;
		\item $\mathcal{E}_{\mbox{\tiny{undirected}}}\subseteq \left\{1,\dots,127\right\}\times \left\{1,\dots,127\right\}$ is the set of coupled qubit pairs; it represents the undirected graph defined in \eqref{indirect_graph_definition} (see figure \ref{grah_1});
	\end{itemize}
	(see \cite{koch2007charge,mckay2018qiskit} and \cite{magesan2020effective}).
	
%	The eigenvalues are all rational, whence they are commensurate. This guarantees the periodicity of $t\longmapsto \exp(-iH_0 t)$.
	
	The {S}chr\"{o}dinger equation \eqref{ibm_equation_control} can be rewritten as
	\begin{equation}\label{ibm_brisbane_equation_}
		i \frac{d}{dt} {\psi(t)} = \Bigg(
		\sum_{k=1}^{127} \frac{\omega_k}{2} Z_k
		+ \sum_{(k,l) \in \mathcal{E}_{\mbox{\tiny{undirected}}}} J_{kl} Z_k Z_l
		+ \sum_{k=1}^{127} \left( u_k^X(t) X_k + u_k^Y(t) Y_k \right)
		+ \sum_{(c,t) \in \mathcal{E}} v_{ct}(t) Z_c \otimes X_t
		\Bigg) {\psi(t)}.
	\end{equation}
	
	Note that, by employing this drift Hamiltonian, by slightly modifying the proof of Proposition \ref{prop_1}, operator controllability holds even without cross-resonance controls. Namely, we can control the system setting $v_{ct}(t)\equiv 0$. However, cross-resonance controls might be useful to reduce the number of switches.

	\chapter{A surrogate problem}
	\label{chapter:A surrogate problem}
	
	This section introduces a surrogate Optimal Control Problem, which will give a necessary condition for Quantum Advantage (QA).
	
	Let us work in the framework of section \ref{chapter:Mathematical framework}.
	
	For any time horizon $T>0$, define the functional
	\begin{equation}\label{functionalJ_T_commutator_alongtime}
		J_T:\mathscr{U}^T\longrightarrow \mathbb{R}
	\end{equation}
	\begin{equation}
		J_T(u)\defeq \frac12 \int_0^T\left\|[U(t)\Gamma^*,H_0]\right\|^2dt,
	\end{equation}
	where
	\begin{itemize}
		\item the state $U(t)$ solves 
		\begin{equation}\label{matrix_schrodinger_J_T}
			\begin{cases}
				i\frac{d}{dt}U(t) = \big(H_0 + \sum_{j=1}^{m}u_j(t)H_j\big)U(t),& t\in (0,T)\\
				U(0)=I,\\
			\end{cases}
		\end{equation}
		with control $u(t)=(u_k(t))_{k=1}^m$;
		\item $[\cdot,\cdot]$ denotes the commutator;
		\item the norm is the matrix norm induced by the $L^2$ norm\footnote{Let $A$ be a $N\times N$ matrix with complex entries. The $L^2$-induced matrix norm is
			\begin{equation}
				\left\|A\right\|=\sup_{\mathbf{z}\in \mathcal{H}\setminus  \left\{\mathbf{0}_{\mathcal{H}}\right\}}\frac{\left\|A\mathbf{z}\right\|}{\left\|\mathbf{z}\right\|}.
			\end{equation}
		}
	\end{itemize}
	
	Reminding the constraints $|u(t)|\leq M$ for control in $\mathscr{U}^T$, by the Direct Methods in the Calculus of Variations (DMCV) \cite{dacorogna2007direct}, there exists a minimizer $u_T\in \mathscr{U}^T$ for \eqref{functionalJ_T_commutator_alongtime}.
	
	For any time horizon $T>0$ and special unitary operator $\Gamma$, define the value function
	\begin{equation}\label{label_value_function}
		V(T,\Gamma)\defeq\inf_{u\in \mathscr{U}^T}J_T=\min_{u\in \mathscr{U}^T}J_T.
	\end{equation}
	
	In the next proposition, we present a necessary condition for QA in terms of estimates of the value function \eqref{label_value_function}.
	
	\begin{proposition}\label{prop_7}
		Let $\Gamma$ be a special unitary operator. Assume \eqref{ger equation} is operator controllable. In the context of Definition \ref{quantum_advantage_definition}, suppose there is Quantum Advantage (QA). Then, the value function can be estimated as follows
		\begin{equation}\label{QA_turnpike}
			V(T,\Gamma)\leq 2\left\|H_0\right\|^2Cn^p,
		\end{equation}
		the constant $C$ and the integer $p$ being the same of \eqref{label_t_min_estimate}.
	\end{proposition}
	The necessary condition \eqref{QA_turnpike} can be seen as a specific stabilization-turnpike property on the commutator. A huge literature is available on stabilization of control systems and the turnpike phenomenon; see, for instance, the following articles and books and the references therein \cite{BRC,LIO,MIRRAHIMI20051987,hou2014realization,porretta2013long,PZ2,trelat2015turnpike,trelat2018steady,pighin2020turnpike,trelat2024control}.
	\begin{proof}[Proof of Proposition \ref{prop_7}]
		\textit{Step 1} \ \textbf{Definition of a special control} \\
		Let $T_{\mbox{\tiny{min}}}$ be the minimal time (as defined in \eqref{def_min_time}) for the target operator $\Gamma$. By remark \ref{remark_min_time}, there exists a control $u_{\mbox{\tiny{min}}}\in \mathscr{U}_{\mbox{\tiny{ad}}}$, such that the following operator controllability problem is fulfilled
		\begin{equation}\label{label_controllability_problem_T_minimal_time}
			\begin{cases}
				i\frac{d}{dt}U_{\mbox{\tiny{min}}}(t) = \big(H_0 + \sum_{j=1}^{m}u_{\mbox{\tiny{min}},j}(t)H_j\big)U_{\mbox{\tiny{min}}}(t),& t\in (0,T_{\mbox{\tiny{min}}})\\
				U_{\mbox{\tiny{min}}}(0)=I,\\
				U_{\mbox{\tiny{min}}}(T_{\mbox{\tiny{min}}})=\Gamma.\\
			\end{cases}
		\end{equation}
		Hence, let us define the control		
		\begin{equation}
			\hat{u}(t)\coloneqq \begin{dcases}
				u_{\mbox{\tiny{min}}}(t) \quad &t\in [0,T_{\mbox{\tiny{min}}})\\
				0 \quad &t\in [T_{\mbox{\tiny{min}}},+\infty),\\
			\end{dcases},
		\end{equation}
		Consider the state $\widehat{U}$ satisfying \eqref{matrix_schrodinger_inftime}, with control $u$. Then, by uniqueness of solutions to Cauchy Problem, we have
		\begin{equation}
			\widehat{U}(t)\coloneqq \begin{dcases}
				U_{\mbox{\tiny{min}}}(t) \quad &t\in [0,T_{\mbox{\tiny{min}}})\\
				\Gamma \quad &t=T_{\mbox{\tiny{min}}}\\
				\exp(-i H_0(t-T_{\mbox{\tiny{min}}}))\Gamma \quad &t\in [T_{\mbox{\tiny{min}}},+\infty),\\
			\end{dcases}.
		\end{equation}
		Then, for any time $t\in [T_{\mbox{\tiny{min}}},+\infty)$, the commutator
		\begin{equation}\label{cmmutator_t_0}
			[\widehat{U}(t)\Gamma^*,H_0]=[\exp(-i H_0(t-T_{\mbox{\tiny{min}}}))\Gamma\Gamma^*,H_0]=[\exp(-i H_0(t-T_{\mbox{\tiny{min}}})),H_0]=0.
		\end{equation}
		
		\textit{Step 2} \ \textbf{Estimate of $\left\|[\widehat{U}(t)\Gamma^*,H_0]\right\|^2$, for every time $t\in [0,T_{\mbox{\tiny{min}}}]$} \\
		For all time instances in $[0,T_{\mbox{\tiny{min}}}]$, the matrix $\exp(-i(H_0 + \sum_{j=1}^{m}u_{\mbox{\tiny{min}},j}(t)H_j)t)$ is unitary, whence, by definition of operator norm
		\begin{equation}
			\left\|\widehat{U}(t)\right\|=\left\|\exp(-i(H_0 + \sum_{j=1}^{m}u_{\mbox{\tiny{min}},j}(t)H_j)t)\right\|=1,
		\end{equation}
		namely the {S}chr\"{o}dinger group with Hamiltonian $H(t)\coloneqq H_0 + \sum_{j=1}^{m}u_{\mbox{\tiny{min}},j}(t)H_j$ is unitary.
		
		Therefore, 
		% for all time instances in $[0,T_{\mbox{\tiny{min}}}]$,
		we have
		\begin{eqnarray}
			\left\|\widehat{U}(t)\Gamma^*H_0\right\|^2&=&\left\|H_0\right\|^2\nonumber\\
			\left\|H_0\widehat{U}(t)\Gamma^*\right\|^2&=&\left\|H_0\right\|^2.\nonumber
		\end{eqnarray}
		
		We can then estimate the integral
		\begin{eqnarray}
			\int_0^{T_{\mbox{\tiny{min}}}}\left\|[\widehat{U}(t)\Gamma^*,H_0]\right\|^2dt&\leq&2\int_0^{T_{\mbox{\tiny{min}}}}\left[\left\|\widehat{U}(t)\Gamma^*H_0\right\|^2+\left\|H_0\widehat{U}(t)\Gamma^*\right\|^2\right]dt\nonumber\\
			&=&2\int_0^{T_{\mbox{\tiny{min}}}}\left[\left\|H_0\right\|^2+\left\|H_0\right\|^2\right]dt\nonumber\\
			&=&4\left\|H_0\right\|^2T_{\mbox{\tiny{min}}}.\nonumber
		\end{eqnarray}
		Hence, employing also \eqref{cmmutator_t_0}, we obtain
		\begin{equation}\label{integral_0_+infinity}
			\int_0^{+\infty}\left\|[\widehat{U}(t)\Gamma^*,H_0]\right\|^2dt\leq 4\left\|H_0\right\|^2T_{\mbox{\tiny{min}}}.
		\end{equation}
		
		\textit{Step 3} \ \textbf{Conclusion} \\
		By definition of value function \eqref{label_value_function}, for any time horizon $T>0$,
		\begin{equation}\label{}
			V(T,\Gamma)\leq J_T(\hat{u})\leq \frac12 \int_0^{+\infty}\left\|[\widehat{U}(t)\Gamma^*,H_0]\right\|^2dt\leq 2\left\|H_0\right\|^2T_{\mbox{\tiny{min}}},
		\end{equation}
		where the last inequality uses \eqref{integral_0_+infinity}. This concludes the proof.
	\end{proof}
	
	\section{Application to the two representative problems}\label{sec:surrogate_application}
	
	The surrogate bound \eqref{QA_turnpike} is a \emph{necessary} condition for Quantum Advantage in the sense of Definition \ref{quantum_advantage_definition}. The two representative problems of this manuscript fit the scheme as follows.
	
	\begin{itemize}
		\item[(a)] \textbf{QFT on superconducting hardware.} Combining Theorem \ref{theorem_QA_QFT} with Proposition \ref{prop_7}, we obtain the sharp quantitative form
		\begin{equation}\label{surrogate_QFT}
			V(T,\Gamma_{\mbox{\tiny{QFT}}})\leq 2\left\|H_0\right\|^2\tau n^2,\qquad \forall T>0,
		\end{equation}
		already recorded in Corollary \ref{corollary_necessary_condition_QFT}. Equation \eqref{surrogate_QFT} is checkable on finite-dimensional models of \texttt{ibm\_brisbane}: if one can exhibit a sequence of target unitaries $\Gamma_n \in SU(2^n)$ along which the value function $V(T,\Gamma_n)$ grows faster than $n^2$, then the uniformly-bounded-elementary-gate-time assumption of Definition \ref{def_elementary_gate_time} fails for the given hardware.
		\item[(b)] \textbf{MIS on neutral-atom hardware.} In the analog setting of chapter \ref{chapter:Quantum Discrete Optimization}, full operator controllability on $SU(2^n)$ is ruled out by Proposition \ref{prop_pasqal_no_full_controllability}. Hence, Proposition \ref{prop_7} does not apply \emph{directly} to the state-oriented QA of Definition \ref{def_QA_MIS}. A natural line of research (cf. relaxation (R3) in section \ref{sec:Is there quantum advantage?}) is to formulate a state-targeted analog of \eqref{QA_turnpike}, in which the commutator $[U(t)\Gamma^*,H_0]$ is replaced by the residual energy $\langle \psi(t)|H_C|\psi(t)\rangle - E_{\min}(H_C)$, and the minimal time $T_{\mbox{\tiny{min}}}$ by the adiabatic time dictated by the instantaneous spectral gap of the interpolating Hamiltonian $u(t)H_B+(1-u(t))H_C$.
	\end{itemize}
	
	In both cases, the surrogate functional plays the role of a \emph{certificate}: a polynomial-in-$n$ upper bound on the value function is a necessary condition for polynomial-in-$n$ minimal time, hence for Quantum Advantage.
	
	% to update the bibliography, open an Ubuntu terminal and run
	% cp "/home/dario/Documents/work_in_progress/elaborazione scrittura/my_references.bib" "/home/dario/Documents/work_in_progress/QFT"
	
	% NOTE (added with the rigorous adiabatic estimates): the proof of
% Theorem~\ref{thm_exact_controllability_MIS} and relaxation (R3) now cite
% Jansen-Ruskai-Seiler. Ensure my_references.bib contains:
%   @article{jansen2007bounds,
%     author  = {Jansen, Sabine and Ruskai, Mary-Beth and Seiler, Ruedi},
%     title   = {Bounds for the adiabatic approximation with applications to quantum computation},
%     journal = {Journal of Mathematical Physics},
%     volume  = {48}, number = {10}, pages = {102111}, year = {2007},
%     doi     = {10.1063/1.2798382},
%   }
	\bibliography{my_references}
	\bibliographystyle{siam}
	
%	\medskip
%	% The data information below will be filled by AIMS editorial staff
%	Received xxxx 20xx; revised xxxx 20xx.
%	\medskip
	
\end{document}